\documentclass[11pt]{amsart}
\usepackage{amsmath}
\usepackage{amssymb}
\usepackage{amsthm}
\usepackage{amscd}
\usepackage{fontenc}
\usepackage{url}
\usepackage{hyperref}
\usepackage[all]{xy}
\usepackage{graphicx}
\usepackage{tikz}
 \usetikzlibrary{arrows}

\numberwithin{equation}{section}

\newcounter{AbcT}

\newtheorem {Theorem}    {Theorem}[section]
\newtheorem* {Theorem1.9}    {Theorem 1.9}
\newtheorem* {Theorem1.6}    {Theorem 1.6}
\newtheorem* {Question1.7}    {Question 1.7}

\newtheorem {Definition}[Theorem] {Definition} 
\newtheorem* {Remark}	{\bf{Remark}}

\newtheorem {Lemma}      [Theorem]    {Lemma}
\newtheorem {Corollary}   [Theorem] {Corollary}
\newtheorem {Proposition}[Theorem]    {Proposition}
\newtheorem {Claim}      [Theorem]    {Claim}
\newtheorem {Observation}[Theorem]    {Observation}

\theoremstyle{remark}

\newcounter{DM@bibnum}

\newcommand{\la}{\langle}
\newcommand{\ra}{\rangle}

\def\IA{{\rm IA}}
\def\IAC{{\rm IAC}}
\def\IAR{{\rm IAR}}
\def\Col{{\rm Col}}
\def\Row{{\rm Row}}
\def\SL{{\rm SL}}
\def\GL{{\rm GL}}
\def\Sp{{\rm Sp}}

\def\deg{{\rm deg\,}}

\def\supp{{supp\,}}

\def\Aut{{\rm Aut}}
\def\SAut{{\rm SAut}}

\def\Mod{{\rm Mod}}

\def\Ker{{\rm Ker\,}}

\def\rk{\rm{rk\,}}

\def\NSL_2{{\mathcal N SL_2}}

\def\eps{\varepsilon}

\def\phi{\varphi}

\def\calC{{\mathcal C}}

\def\calF{{\mathcal F}}
\def\calG{{\mathcal G}}

\def\calI{{\mathcal I}}

\def\calR{{\mathcal R}}

\def\calT{{\mathcal T}}
\def\calU{{\mathcal U}}

\def\calX{{\mathcal X}}

\def\hbar{\bar h}

\def\dbN{{\mathbb N}}

\def\dbQ{{\mathbb Q}}
\def\dbR{{\mathbb R}}

\def\dbZ{{\mathbb Z}}

\def\skv{{\vskip .12cm}}

\begin{document}

\title{On finite presentability of some partial Torelli subgroups of $\Aut(F_n)$}



\author{Mikhail Ershov}
\address{University of Virginia}
\email{ershov@virginia.edu}

\begin{abstract} Let $F_n$ be the free group of rank $n$, and let $\rho_{ab}:\Aut(F_n)\to \GL_n(\dbZ)$ be the map
induced by the natural projection $F_n\to\dbZ^n$.
It is a long-standing open problem whether the subgroup of $\IA$-automorphisms $\IA_n=\Ker\rho_{ab}$ is finitely presented for $n\geq 4$. In this paper we establish finite presentability of certain infinite index subgroups of $\Aut(F_n)$ containing $\IA_n$. In the terminology of Putman, these subgroups are natural analogues of partial Torelli subgroups of mapping class groups. 
\end{abstract}

\maketitle
\section{Introduction}

Given $n\geq 2$, let $F_n$ be the free group of rank $n$ and $\Aut(F_n)$ its automorphism group. We begin the paper with a brief overview of the past work on finite presentability of $\Aut(F_n)$ and some of its subgroups.

\skv
\paragraph{\bf Finite presentability of $\Aut(F_n)$ and some of its subgroups.} Finite generation of $\Aut(F_n)$ was established by Nielsen in 1921~\cite{Ni1}, and shortly afterwards, Nielsen
proved that $\Aut(F_n)$ is finitely presented~\cite{Ni}. While the proof of finite generation in \cite{Ni1} and the description of the finite presentation in \cite{Ni} were elementary, justification of this presentation used sophisticated geometric techniques. Whitehead~\cite{Wh} gave an algorithm which determines whether two given $n$-tuples of elements of $F_n$
lie in the same $\Aut(F_n)$-orbit, using what is now known as the {\it peak reduction lemma}.
This was another fundamental result about $\Aut(F_n)$ with a simple algebraic statement proved by a non-algebraic method.

Rapaport~\cite{Ra} gave an algebraic proof of Whitehead's peak reduction lemma, which was later simplified by Higgins and Lyndon~\cite{HL}.
McCool~\cite{Mc1} used a variation of the results from \cite{HL} to find another presentation for $\Aut(F_n)$, and 
then used it in \cite{Mc2} to give another (this time purely algebraic) proof of the correctness of Nielsen's presentation from \cite{Ni}. Later, using similar peak-reduction techniques, McCool~\cite{Mc} established finite presentability of several classes of subgroups of $\Aut(F_n)$, including stabilizers of finite subsets as well as algebraic mapping class groups. More recently, Day generalized McCool's finite presentability
results to the automorphism groups of the right-angled Artin groups~\cite{D09} and their corresponding subgroups \cite{D14}. 
\skv

\paragraph{\bf The main result.} One prominent subgroup of $\Aut(F_n)$ whose presentability does not seem to be tractable by McCool's method (or its variations) is $\IA_n$, called the subgroup of $\IA$-automorphisms or the {\it Torelli subgroup}. It is defined as the kernel
of the map $\rho_{ab}:\Aut(F_n)\to \Aut(\dbZ^n)=\GL_n(\dbZ)$ induced by the natural projection $F_n\to\dbZ^n$. Magnus~\cite{Ma} proved that $\IA_n$ is finitely generated and found a simple finite generating set. This result immediately implies that $\IA_2$ is free of rank $2$ and thus trivially finitely presented.
Krstic and McCool~\cite{KM} proved that $\IA_3$ is not finitely presented, and the question whether $\IA_n$ is finitely presented for $n\geq 4$ remains open.

We note that if for some $n$ the Torelli group $\IA_n$ is finitely presented, then so is any subgroup of the form $\rho_{ab}^{-1}(P)$
where $P$ is a finitely presented subgroup of $\GL_n(\dbZ)$. On the other hand, to the best of our knowledge, the groups of this form were only known to be finitely presented when $P$ has finite index in $\GL_n(\dbZ)$ (in which case the result follows automatically from finite presentability of $\Aut(F_n)$). The goal of this paper is to establish finite presentability for a natural family of groups of the form $\rho_{ab}^{-1}(P)$ with $P$ of infinite index in $\GL_n(\dbZ)$ (Theorem~\ref{thm:finpres} below). By analogy with \cite{Pu}, we will call these groups {\it partial Torelli subgroups}.

\skv
\paragraph{\bf Partial Torelli subgroups.} Given integers $1\leq d\leq n$, let $\Col_{n,d}$ (resp. $\Row_{n,d}$) denote the subgroup of $\GL_n(\dbZ)$ consisting of all matrices whose first $d$ columns (resp. rows) coincide with those of the identity matrix. Let $\IAC_{n,d}$ (resp. $\IAR_{n,d}$) denote the preimage of $\Col_{n,d}$ (resp. $\Row_{n,d}$) under $\rho_{ab}:\Aut(F_n)\to \GL_n(\dbZ)$.
Thus, $$\IA_n=\IAC_{n,n}\subseteq \IAC_{n,n-1}\subseteq \cdots\subseteq \IAC_{n,0}=\Aut(F_n)$$ and similarly with $\IAC_{n,d}$
replaced by $\IAR_{n,d}$.

\begin{Theorem} 
\label{thm:finpres}
The groups $\IAC_{n,d}$ and $\IAR_{n,d}$ are finitely presented when $n\geq d+115$.
Moreover, the groups $\IAC_{n,1}$ and $\IAR_{n,1}$ are finitely presented  for all $n\geq 26$.
\end{Theorem}
\begin{Remark}\rm The question whether $\IAR_{n,1}$ is finitely presented for $n\geq 4$ was asked by Krstic and McCool~\cite[Problem~3]{KM}.
It was proved in \cite{KM} that $\IAR_{3,1}$ is not finitely presented.
\end{Remark}

\paragraph{\bf Motivation: BNSR invariants.}
Our proof of Theorem~\ref{thm:finpres} is inspired by the work of Renz on the second BNSR invariant. Given a finitely generated group $G$,
its BNS invariant $\Sigma(G)=\Sigma^1(G)$ (which is also the first BNSR invariant) was introduced in the celebrated paper of Bieri, W. Neumann and Strebel~\cite{BNS}. Among other things, it was proved in \cite{BNS} that $\Sigma(G)$ completely determines which coabelian
subgroups of $G$ are finitely generated (we call a subgroup $N$ of $G$ coabelian if $N$ is normal and $G/N$ is abelian). This result is often called the {\it BNS criterion}. Higher order analogues of $\Sigma(G)$, now called {\it BNSR invariants}, were introduced and studied in the Ph.D. thesis of Renz~\cite{Re} (homotopical BNSR invariants) and in the paper of Bieri and Renz~\cite{BR} (homological BNSR invariants). In particular, Renz~\cite{Re} proved that the second homotopical BNSR invariant $\Sigma^2(G)$, which can be associated to any finitely presented group $G$,  determines which coabelian subgroups of $G$ are finitely presented (this result will be referred to as {\it Renz's criterion}).

In the proofs of both BNS criterion and Renz's criterion, geometry of the Cayley graph of $G/N$ plays the key role in determining when 
a coabelian subgroup $N$ of $G$ is finitely generated (resp. finitely presented). A. Putman\footnote{private communication} suggested that this geometric approach may be applicable beyond the coabelian setting, and in this paper we will implement this idea in the case where $G=\SAut(F_n)$ and $N=\IAR_{n,d}\cap G$ or $\IAC_{n,d}\cap G$ (with $n$ and $d$ as in Theorem~\ref{thm:finpres}). Here
$\SAut(F_n)$ is the subgroup of ``orientation-preserving'' automorphisms of $F_n$ defined by $\SAut(F_n)=\rho_{ab}^{-1}(\SL_n(\dbZ))$
(it has index $2$ in $\Aut(F_n)$). Note that in these cases $N$ is not normal in $G$, so instead of the Cayley graph we will deal with the Schreier graph of $G/N$. 

The reason we will be working with $\SAut(F_n)$ instead of $\Aut(F_n)$ is that $\SAut(F_n)$ has a particularly nice Steinberg-type presentation
found by Gersten~\cite{Ge}, which is very similar to the standard presentation of $\SL_n(\dbZ)$ -- see \S~\ref{sec:groupandrel}. 
\skv
 
\paragraph{\bf About the proof of Theorem~\ref{thm:finpres}.} We will prove Theorem~\ref{thm:finpres} using a general criterion 
in terms of van Kampen diagrams: 

\begin{Proposition} 
\label{thm:basiccriterion}
Let $G=\la X|R \ra$ be a  finitely presented group, $H$ a finitely generated subgroup of $G$
and $\rho:G\to G/H$ the natural projection.
Let $A\subseteq G/H$ be a finite subset such that $\rho^{-1}(A)$ is connected in the Cayley graph $Cay(G,X)$,
and suppose that there exists a finite subset $B$ of $G/H$ with the following property:
\begin{itemize} 
\item[(*)] For any simple closed path $\gamma$ in $Cay(G,X)$ all of whose vertices lie in $\rho^{-1}(A)$ there exists a 
disk van Kampen diagram $\Omega$ relative to $(X,R)$ such that $\partial \Omega = \gamma$ (as defined below) and all vertices of $\Omega$ lie in $\rho^{-1}(B)$.
\end{itemize}
Then $H$ is finitely presented.
\end{Proposition}
In the statement of Proposition~\ref{thm:basiccriterion} we assume that the vertices of $\Omega$ are labeled by elements of $G$
(and as usual the edges of $\Omega$ are labeled by elements of $X$) such that $l(w)=l(e)l(v)$ whenever $e$ is an edge from $v$ to $w$.
If we fix some vertex $v$ of $\Omega$ and some $g\in G$, there exists a unique such vertex labeling with $l(v)=g$. The equality
$\partial \Omega = \gamma$ means that there exists a vertex $v$ of $\partial \Omega$ and a vertex $g\in G$ such that $l(v)=g$
and $l(\partial \Omega)=l(\gamma)$ where $l(\partial \Omega)$ and $l(\gamma)$ are the labels of the paths $\Omega$ and $\gamma$
read starting from $v$ and $g$, respectively.

\begin{Remark}\rm
A finite subset $A\subseteq G/H$ such that $\rho^{-1}(A)$ is connected always exists since $H$ is finitely generated
(see \cite{Str,EF}).
\end{Remark}

Proposition~\ref{thm:basiccriterion} is a special case of a well-known criterion of Brown~\cite[Theorem~3.2]{Br2}.
Although Proposition~\ref{thm:basiccriterion} cannot 
be formally deduced from the statement of \cite[Theorem~3.2]{Br2}, it immediately follows from
its proof (see Theorem~\ref{thm:Br87}). 
In \S~2 we will also provide a self-contained and purely group-theoretic proof of Proposition~\ref{thm:basiccriterion} 
(we will slightly reformulate it using additional terminology introduced at the beginning of \S~2 --
see Proposition~\ref{thm:criterion}).

The proof of the aforementioned Renz's criterion in \cite{Re} uses a special case of Proposition~\ref{thm:basiccriterion}
where condition (*) is assumed to hold for $B=A$. In this case
one can use the basic fact that a group $H$ is finitely presented if it acts freely
on a simply-connected CW-complex $\calC$ such that the quotient $\calC/H$ has finite 2-skeleton 
(see, e.g., \cite[Theorem~4]{Br}).
\skv

Let us now describe a general method for verifying condition (*) in Proposition~\ref{thm:basiccriterion} for a specific pair $(G,H)$.
This method was introduced in \cite{Re} in a more specialized setting. 
Suppose that we have
a norm function $N:G/H\to Z_0$ for some well-ordered set $Z_0$ with the property that for all
$z\in G/H$ the set $\{y\in G/H: N(y)\leq N(z)\}$ is finite (in particular, $N$ has finite fibers). 
Since $B$ can always be enlarged without violating (*), 
we can assume that $B$ contains $\{z\in G/H: N(z)\leq N(a)\mbox{ for some }a\in A\}$.
By van Kampen's lemma, for any simple closed path $\gamma$ as in (*) there exists some 
disk van Kampen diagram $\Omega$ relative to $(X,R)$  such that  $\partial \Omega=\gamma$.
Then $\Omega$ either satisfies the conclusion of (*) or contains an interior vertex $v$
of maximal norm. In the latter case our goal is to replace $\Omega$ by a modified diagram $\Omega'$ where $v$ is eliminated and every new vertex $w$ satisfies $N(w)<N(v)$. If such $\Omega'$
can always be constructed, an easy argument shows that after finitely many modifications we will obtain
a diagram satisfying (*).
\skv

Renz~\cite{Re} dealt with this problem in the special case where $H$ is normal in $G$ and $G/H$ is free abelian.
In this case the norm function $N$ is just the Euclidean norm on $G/H$ (relative to a fixed basis of $G/H$), and there exists a fairly general algorithm  for constructing the modified diagram $\Omega'$ (of course, the algorithm only works if the group $H$ is actually finitely presented).

The norm functions in this paper will be more involved and defined in a more ad hoc way. Likewise, at each step
we will need to construct $\Omega'$ from $\Omega$ in several substeps, and the algorithm will depend on the boundary labels
of the 2-cells containing $v$ (the chosen vertex of maximal norm).
\skv
The method we just described can also be adequately called {\it peak reduction} (at each step our goal is to remove the chosen ``peak'' vertex $v$), but the setting is different from those of Whitehead's lemma and its generalizations in McCool's papers ~\cite{Mc1,Mc}. In those papers one starts with a group $G$ defined by a faithful action on
some set $\Omega$ and then proves finite presentability of $G$ directly using this action. In our case we start with a group $G$ with a known finite presentation $(X,R)$ and prove finite presentability of a subgroup $H$ of $G$ using its action on the Cayley complex of $G$ corresponding to the presentation $(X,R)$.
\skv
 
\paragraph{\bf Some related questions.} Given a group $G$, one has the following implications
\skv
\centerline{$G$ is finitely presented $\Rightarrow$ $G$ has type $(\mathrm{FP}_2)$ $\Rightarrow$ $\rk H_2(G,\dbZ)<\infty$ $\Rightarrow$ 
$\dim H_2(G,\dbQ)<\infty$.}
\skv
In the case $G=\IA_n$, $n\geq 4$, it is not known whether any of the above properties hold. In this paper we prove finite presentability
for the groups $\IAR_{n,d}$ and $\IAC_{n,d}$, $n\geq d+115$, and it is natural to ask if similar ideas could be used to establish
some weak form of finite presentability from the above list for a certain group $H$ lying strictly between $\IAR_{n,d}$ or $\IAC_{n,d}$
and $\IA_n$, with the most ambitious case being $H=\IA_n$. In \cite{DP17}, Day and Putman proved that
for $n\geq 6$, the second homology $H_2(\IA_n,\dbZ)$ is finitely generated as a $\GL_n(\dbZ)$-module and
moreover, $\IA_n$ is ``finitely presented with respect to the conjugation action of $\Aut(F_n)$'' (in a suitable sense). It would be interesting to see if these results could be used to strengthen Theorem~\ref{thm:finpres}. 
\skv 
Another challenge is to adapt the proofs from this paper to the case of mapping class groups. Given non-negative integers $g$ and $b$, let $\Sigma_{g,b}$ be an orientable surface of genus $g$ with $b$ boundary components and let $\Mod_{g,b}=\Mod(\Sigma_{g,b})$ be its mapping class group. The subgroup $\calI_{g,b}$ consisting of elements of
$\Mod_{g,b}$ which act trivially on $H_1(\Sigma_{g,b},\dbZ)$ is called the Torelli subgroup of $\Mod_{g,b}$,
and in the cases $b=0,1$ there are many similarities between the groups $\Mod_{g,b}$ (resp. $\calI_{g,b}$)
and $\Aut(F_n)$ (resp. $\IA_n$). In \cite{Pu}, Putman considered  natural counterparts of $\IAR_{n,d}$ and $\IAC_{n,d}$ inside $\Mod_{g,b}$, $b\geq 1$, called {\it partial Torelli subgroups}, and defined as follows. Let us think of $\Sigma_{g,b}$ as a sphere with $g$ handles attached
and $b$ disks removed and enumerate the handles arbitrarily. For $1\leq d\leq g$ define the partial Torelli subgroup $\calI_{g,b;d}$ to be the subgroup of $\Mod_{g,b}$ consisting of mapping classes which acts trivially on the part of $H_1(\Sigma_{g,b},\dbZ)$ supported on the first 
$d$ handles (thus, $\Mod_{g,b}=\calI_{g,b;0}\supseteq \calI_{g,b;1}\supseteq\cdots \supseteq \calI_{g,b;g}=\calI_{g,b}$).
Putman~\cite{Pu} established various homological stability results for these groups. To the best of our knowledge, it is currently an open problem whether
the groups $\calI_{g,b;d}$ are finitely presented, apart from the case $d=0$. 
\skv
Several amazing breakthroughs on closely related questions about $\Mod_{g,b}$ have occured in the past two years.
First, Minahan~\cite{Mi2} proved that the Torelli group $\calI_{g,b}$, $b\leq 1$,
has finite dimensional second rational homology for $g\geq 51$. In \cite{MP25a}, Minahan and Putman extended this theorem
to all $g\geq 6$ and moreover gave an explicit description of $H_2(\calI_{g,b},\dbQ)$ as a $\Sp_{2g}(\dbZ)$-module;
see also \cite{Mi}, \cite{MP25b} and \cite{MP25c} for important related results, some of which are used in \cite{MP25a}.
We do not know if the techniques from these papers
could be applicable to the groups $\IA_n$ or any of the groups $\IAR_{n,d}$ and $\IAC_{n,d}$.

\skv 
 
\paragraph{\bf Outline of the paper.} The paper is organized as follows. In \S~2 we will introduce most of our terminology involving van Kampen diagrams, some of which is non-standard, and describe in detail our general method for proving finite presentability. We will also introduce Gersten's presentation for $\SAut(F_n)$ and its universal central extension which will provide a starting point for proving finite presentability of $\IAR_{n,d}$ and $\IAC_{n,d}$. In \S~3 and \S~4 we will prove finite presentability for
the groups $\IAR_{n,1}$ and $\IAC_{n,1}$, respectively. Since the proofs in these two cases will follow the same general outline, in \S~4  
we will concentrate on parts of the proof for $\IAC_{n,1}$ that require non-trivial modifications. Finally, in \S~5, we will prove 
Theorem~\ref{thm:finpres} in the general case. 
The proof will be essentially inductive, although we are unable to formally use induction on $d$. More precisely, in order to prove Theorem~\ref{thm:finpres} for $\IAR_{n,d}$ and $\IAC_{n,d}$ it will not be enough to assume the result for smaller values of $d$. Instead, we will need to imitate certain steps of the proof for $d=1$, but in a more general setting.

\vskip .2cm
\paragraph{\bf Acknowledgments} I am extremely grateful to Andrew Putman for explaining the proof of the BNS criterion from \cite{Str} and sharing his ideas about possible generalizations during his visit to the University of Virginia in Fall 2017. This project would have never started without that conversation. I would also like to thank Matthew Day, Andrei Jaikin, Daniel Minahan and Dmitriy Sonkin for useful discussions.

\section{Preliminaries}

\paragraph{\bf Cayley and Schreier graphs and Schreier sets.} In this paper we will adopt a slightly unusual convention and work with left Cayley graphs instead of more commonly used right Cayley graphs. We will also view Cayley graphs as labeled oriented graphs. Thus if $G$ is a group and $S$ is its generating set, we define $Cay(G,S)$ to be the graph whose vertex set is $G$ and where for each $g\in G$ and $s\in S$ there is an oriented edge from $g$ to $sg$ labeled by $s$. Similarly, one defines the Schreier graph $Sch(Z,S)$ where $Z$ is any left $G$-set.

To be consistent with this convention we define a subset $A$ of a free group $F(X)$ to be {\it Schreier} if for every $a\in A$,
every suffix of $a$ also lies in $A$ (in the usual definition prefixes are used instead of suffixes). Geometrically this means that all vertices on the unique path in $Cay(F(X),X)$ form $1$ to $a\in A$ must lie in $A$.

\subsection{Van Kampen diagrams} 
\label{sec:vanKampen} 
We start with the definition of a van Kampen diagram.  Let $X$ be a finite alphabet and let
$\Omega$ be a finite connected oriented plane graph whose edges are labeled by elements of $X$. We will refer to
the bounded connected components of $\dbR^2\setminus\Omega$ as {\it cells} of $\Omega$.
Given a cell $\calF$, let $e_1,\ldots, e_k$ be its edges listed as we traverse $\partial\calF$,
the boundary of $\calF$, starting from some vertex, either clockwise or counterclockwise. We define $l(\partial\calF)$, the boundary label of $\partial\calF$, to be the word $\prod\limits_{i=k}^1 l(e_i)^{\eps_i}$ where $l(e_i)$ is the label of $e_i$ and
$\eps_i=1$ or $-1$ depending on whether $e_i$ is traversed in the positive or negative direction
(note that $l(\partial F)$ is only defined up to cyclic shifts and inverses).

\begin{Definition}\rm 
Fix a presentation $(X,R)$ of a group $G$ where each $r\in R$ is cyclically reduced, and let $\Omega$ be as above. 

\begin{itemize}
\item[(1)] Suppose that for every cell $\calF$ of $\Omega$,
its boundary label is equal to a cyclic shift of some element of $R\cup R^{-1}$. Then $\Omega$ is called a
{\it van Kampen diagram} over $(X,R)$. 

\item[(2)] Let $\calU$ be the unique unbounded  connected component of $\dbR^2\setminus\Omega$. By abuse of notation we define the boundary $\partial \Omega$ to be $\partial \calU$, and we define its label $l(\partial\Omega)$ in the same way as labels for the cells of $\Omega$.

\item[(3)] We say that $\Omega$ is a {\it disk diagram} if $\dbR^2\setminus\calU$
is homeomorphic to a (closed) disk.
\end{itemize}
\end{Definition}

\begin{Remark}\rm Note that the boundary label of a cell in
a van Kampen diagram should be read ``from right to left'' -- for instance, the boundary label of the cell shown in Figure~\ref{orient} is $dcba$, not $abcd$. This is because we are working with left Cayley graphs. 
\end{Remark} 
\input{figure_orientation}
\skv

The following basic result is known as van Kampen's Lemma:

\begin{Lemma}[van Kampen's Lemma]
\label{lem:vanKampen} 
Let $(X,R)$ be a presentation of a group $G$, with each $r\in R$ cyclically reduced. Let $f\in F(X)$. 
Then $f=1$ in $G$ if and only if there exists a van Kampen diagram $\Omega$ over $(X,R)$ whose boundary label is $f$.
\end{Lemma}

Contrary to a common convention, we do not allow $0$-edges (edges labeled by the trivial element $1$). 
This means that the diagram $\Omega$ in Lemma~\ref{lem:vanKampen} is not necessarily a disk diagram; in general it consists of several disk subdiagrams connected by (possibly degenerate) arcs.

In the definition of a van Kampen diagram only edges (not vertices) come with labels. However the `if' direction of Lemma~\ref{lem:vanKampen}
easily implies that for any van Kampen diagram $\Omega$, one can label the vertices by elements of $G$ such that whenever $v$ and $w$ are vertices and $e$ is an edge from $v$ to $w$ we have $l(w)=l(e)l(v)$; in fact, such a labeling is uniquely determined by the label of one vertex (which can be chosen arbitrarily). More generally, if $Z$ is any left $G$-set, one can label the vertices of $\Omega$
by elements of $Z$ such that if $e$ is an edge from $v$ to $w$, then $l(w)=l(e).l(v)$. In this case we will say that $\Omega$
is a {\it $Z$-labeled diagram over $(X,R)$}. 

We will also use the following shortcut notation. Suppose that $Z=G/H$ for some subgroup $H$ and $\rho:G\to Z$
the natural projection. If $\Omega$ is a $Z$-labeled diagram over $(X,R)$ and $\gamma$ is a simple closed path in $Cay(G,X)$,
we will write $\partial \Omega = [\gamma]_Z$ if there exists a vertex $v$ of $\partial \Omega$ and a vertex $g\in G$ such that 
$l(v)=\rho(g)$ and $l(\partial \Omega)=l(\gamma)$ where the path labels $l(\partial \Omega)$ and $l(\gamma)$ are 
read starting from $v$ and $g$, respectively.
\skv

We will now give a self-contained proof of Proposition~\ref{thm:basiccriterion}, which we slightly rephrase below
(see Proposition~\ref{thm:criterion}) using vertex labels we have just introduced.

\begin{Proposition} 
\label{thm:criterion}
Let $G=\la X|R \ra$ be a finitely presented group, with $R$ cyclically reduced. Let $H$ be a subgroup of $G$. Set 
$Z=G/H$, and let $\rho:G\to Z$ be the natural projection. 
Let $A$ be a finite subset of $Z$ such that $\rho^{-1}(A)$ is connected in the Cayley graph 
$Cay(G,X)$, and suppose that there exists a finite subset $B$ of $Z$ with the following property:
\begin{itemize} 
\item[(*)] For any simple closed path $\gamma$ in $Cay(G,X)$ all of whose vertices lie in $\rho^{-1}(A)$ there exists a 
$Z$-labeled van Kampen diagram $\Omega$ relative to $(X,R)$ 
such that $\partial \Omega=[\gamma]_Z$ and all vertices of $\Omega$ lie in $B$.
\end{itemize}
Then $H$ is finitely presented.
\end{Proposition}

\begin{proof} 
First we reduce to the case $1\in \rho^{-1}(A)$. Indeed, suppose that the hypotheses of Proposition~\ref{thm:criterion} hold 
for some quadruple $(G,H,A,B)$, and choose finite sets $\widetilde A, \widetilde B\in G$ such that $\rho^{-1}(A)=\widetilde A H$
and $\rho^{-1}(B)=\widetilde B H$. The group $G$ acts on the set of $G$-labeled van Kampen diagrams by right multiplication,
and for any $g\in G$ we have $\widetilde A H = (\widetilde A g)(g^{-1}Hg)g^{-1}$ and 
$\widetilde B H =(\widetilde B g)(g^{-1}Hg)g^{-1}$. Thus, the hypotheses of Proposition~\ref{thm:criterion} also
hold for $(G,g^{-1}Hg,\rho_g(\widetilde A g),\rho_g(\widetilde B g))$ where $\rho_g$ is the projection onto $G/(g^{-1}Hg)$. 
Thus, for any $g\in G$ we can replace $H$ by $g^{-1}Hg$ and $A$ by $\rho_g(\widetilde A g)$, and choosing any 
$g\in {\widetilde A}^{-1}$ ensures that $\rho_g(1)\in \rho_g(\widetilde A g)$.

Thus, from now on we assume that $1\in \rho^{-1}(A)$.
We will first prove that $H$ has a finite generating set with a nice property (see Claim~\ref{claim:ess} below).
Let $N$ be the normal closure of $R$ in $F(X)$ (so that $G=F/N$), and let $F_H$ be the unique subgroup of $F(X)$ such that 
$H=F_H/N$. Denote by $\theta$ the natural projection from $F(X)$ to $Z$.

\begin{Claim}
\label{claim:ess}
There exists a finite subset $S$ of $F_H$ such that
\begin{itemize}
\item[(a)] the image of $S$ in $G$ generates $H$;
\item[(b)] $S$ is contained in some free generating set $X_H$ of $F_H$; 
\item[(c)] every suffix of every element of $S\cup S^{-1}$ lies in $\theta^{-1}(A)$.
\end{itemize} 
\end{Claim}
\begin{proof} 
Recall that by our assumption $\rho^{-1}(A)$ is connected in $Cay(G,X)$ and hence $A$ is connected in the Schreier graph 
$Sch(Z,X)$ (note that the latter property is much weaker, and the original assumption will be used again later in the proof). 

Choose any maximal tree $\calT_A$ inside the subgraph of $Sch(Z,X)$ spanned by $A$ 
(so that $A$ is the set of vertices of $\calT_A$). For every $a\in A$ let $\gamma_a$ be the unique path
inside $\calT_A$ from $\rho(1)$ to $a$ (recall that $\rho(1)\in A$), and let $t_a$ be the unique element of $F(X)$ corresponding to $\gamma_a$ (that is, $t_a$ is the product of the edge labels of $\gamma_a$). 
Let $T_A=\{t_a: a\in A\}$. By construction, $\theta:F(X)\to Z=G/H$ maps $T_A$ bijectively onto $A$
and $T_A$ is a Schreier subset (that is, every suffix of an element of $T_A$ lies in $T_A$).

Given $u\in \theta^{-1}(A)$, let $\overline{u}$ be the unique element of $T_A$ such that $\theta(\overline{u})=\theta(u)$.
Let $S$ be the set of all non-identity elements of the form $\overline{xt\,}^{-1}xt$ where $t\in T_A$ and $x\in X$
are such that $xt\in \theta^{-1}(A)$ as well. We claim that $S$ has the required properties.

First, $S$ lies in $F_H$ since $\theta(u)=\theta(v)$ if and only if $u^{-1}v\in F_H$ for $u,v\in F(X)$.
Since $\rho^{-1}(A)$ is connected in $Cay(G,X)$ (this time we need the full power of this assumption),
by \cite[Theorem~2.12]{EF} $S$ satisfies (a). 

Since $T_A$ is a Schreier subset, it is contained inside
some Schreier transversal $T$, and by Schreier's theorem the set 
$X_H=\{\overline{xt\,}^{-1}xt: t\in T, x\in X\}\setminus\{1\}$ (which clearly contains $S$) freely generates $F_H$.
This proves (b).
\skv

Finally, we prove (c). Take any $s\in S$, so by definition $s=\overline{xt\,}^{-1}xt$ for some $t\in T_A$ and $x\in X$
such that $\theta(xt)\in A$ as well. Let $q$ be any suffix of $s$, so that $s=pq$ for some $p\in F(X)$ and the word
$pq$ is reduced. Since $s=\overline{xt\,}^{-1}xt$, either $q$ is a suffix of $xt$ or $p$ is a prefix of $\overline{xt\,}^{-1}$.
In the former case either $q=xt$, so that $\theta(q)\in A$ by assumption, or $q$ is a suffix of $t$, in which case
$\theta(q)\in A$ since $t\in T_A$ and $T_A$ is Schreier. Suppose now that $p$ is a prefix of $\overline{xt\,}^{-1}$. Then
$p^{-1}$ is a suffix of $\overline{xt\,}\in T_A$, so $\theta(p^{-1})\in A$. But $pq\in H$, so $\theta(q)=qH = p^{-1}H=\theta(p^{-1})\in A$.

Thus, we proved that $\theta(q)\in A$ for every suffix $q$ of $s$. Further, the argument in the last case shows that
$\theta(p^{-1})\in A$ for every prefix $p$ of $s$. Since the inverses of the prefixes of $s$ are precisely the suffixes of $s^{-1}$,
we showed that $\theta(q)\in A$ for every suffix $q$ of $s^{-1}$, which completes the proof of (c).
\end{proof}

We proceed with the proof of Proposition~\ref{thm:criterion}. Let $S$ and $X_H$ be as in Claim~\ref{claim:ess}.

Take any word $r\in F(S)$ which represents $1$ in $G$, and let $r'\in F(X)$ be the word obtained from $r$ by expressing each $s\in S$
in terms of $X\cup X^{-1}$. Let $\delta$ be the closed (but not necessarily simple) path in $Cay(G,X)$ starting at $1$ and representing $r'$.

By condition~(c) in Claim~\ref{claim:ess}, every vertex of $\delta$ lies in $\rho^{-1}(A)$. Hence by condition (*) in 
Proposition~\ref{thm:criterion}, for any simple closed subpath $\gamma$ of $\delta$
there exists a $Z$-labeled disk diagram $\Omega_{\gamma}$ with $\partial \Omega_{\gamma}=[\gamma]_Z$ whose vertices lie in $B$. 
By the standard correspondence
between disk diagrams and relators this means that $r'=\prod\limits_{i=1}^k w_i^{-1} r_i^{\pm 1} w_i$ for some $w_i\in \theta^{-1}(B)$
and $r_i\in R$.

Choose a finite subset $V$ of $F(X)$ such that $\theta(V)=B$, and for each $1\leq i\leq k$ let
$v_{i}\in V$ be such that $\theta(w_i)=\theta(v_i)$ . Then $w_i^{-1} v_i$ represents an element of $H$ in $G$, that is, 
$w_i^{-1} v_i\in F(X_H)$. Hence 
$$r'=\prod\limits_{i=1}^k w_i^{-1} r_i^{\pm 1} w_i=\prod\limits_{i=1}^k (w_i^{-1} v_i) (v_i^{-1} r_i^{\pm 1} v_i) (w_i^{-1} v_i)^{-1}$$
lies in the closure of the finite set $R^V=\{r\in R, v\in V\}$ in the group $F_H=F(X_H)$.

Recall that $N$ denotes the kernel of the natural projection $F(X)\to G$. We just showed that $N\cap F(S)$ lies in the normal closure of 
$R^V$ in $F(X_H)$. Thus, the embedding $F(S)\to F(X_H)$ induces a map from $F(S)/N\cap F(S)$ to $H^*=\la X_H | R^V\ra$.
Since $S$ generates $H$, $F(S)/N\cap F(S)$ is naturally isomorphic to $H$. 

Now define the groups $H'$ and $H''$ as follows. For each $u\in X_H\setminus S$ choose any $s_u\in F(S)$ such that $u=s_u$ in $G$.
Let $H'$ be the group obtained from $H^*$ by imposing the relations $u=s_u$ for all $u\in X_H\setminus S$, and
let $H''$ be the group obtained from $H'$ by removing all the generators from $X_H\setminus S$ together with relations
of the form $u=s_u$ with $u\in X_H\setminus S$ and replacing each $u$ by $s_u$ in relations from $R^V$.
By a standard argument $H''$ is isomorphic to $H'$, and by construction $H''$ is finitely presented.

Since $H'$ is generated by the image of $X_H$ and every defining relation of $H'$ holds in $H$, there
is a well-defined projection $\alpha_3:H'\to H$.
Now consider the following sequence of maps:
$$H\stackrel{\alpha_1}{\longrightarrow} F(S)/N\cap F(S)\stackrel{\alpha_2}{\longrightarrow}
H'\stackrel{\alpha_3}{\longrightarrow} H.$$
By construction each $\alpha_i$ is surjective. On the other hand, it is straightforward to check that the composition
$\alpha_1\circ\alpha_2\circ\alpha_3$ is the identity map on $H$, so each $\alpha_i$ must also be injective and hence an isomorphism. Thus $H\cong H'\cong H''$, so $H$ is finitely presented.
\end{proof}

Let us now explain why Proposition~\ref{thm:criterion} also follows from a criterion of Brown. The following result can be deduced
from the proof of \cite[Theorem~3.2]{Br2}:

\begin{Theorem}\label{thm:Br87}
Suppose that a group $H$ has a cellular action on a CW-complex $C$ such that the vertex stabilizers are
finitely presented and the edge stabilizers are finitely generated. Assume that $C_1\subseteq C_2$ are connected 
$H$-invariant subcomplexes of $C$ such that
\begin{itemize}
\item[(i)] the action of $H$ on $C_2$ (and hence on $C_1$) is cocompact;
\item[(ii)] the induced map $\pi_1(C_1)\to \pi_1(C_2)$ is trivial.
\end{itemize}
Then $H$ is finitely presented.
\end{Theorem}

\begin{proof}[Another proof of Proposition~\ref{thm:criterion}.]
Define $C=Cay(G;X,R)$ to be the Cayley complex associated to the presentation $(X,R)$,
and let $C_1$ (resp. $C_2$) be the full subcomplexes spanned by all $g\in \rho^{-1}(A)$ (resp. all $g\in \rho^{-1}(B)$). 

The group $H$ acts on $C$ by right multiplication. This action is cellular and free on vertices and thus has trivial vertex stabilizers and finite
edge stabilizers. Since $\rho(gh)=\rho(g)$
for all $g\in G$ and $h\in H$, both $C_1$ and $C_2$ are $H$-invariant. The subcomplex $C_1$ is connected by the hypotheses of 
Proposition~\ref{thm:criterion}, and we can make $C_2$ connected by enlarging $B$ if needed (pick a ball in $Cay(G,X)$ centered at $1$
whose image in $G/H$ contains $A$ and let $B$ be the image of that ball in $G/H$). The action of $H$ on $C_2$ is cocompact since $B$ is finite. 

Finally, $\pi_1(C_1)$ is generated by closed paths
in the $Cay(G,X)$ which stay in $C_1$. Condition (*) in Proposition~\ref{thm:criterion} implies that any such path
is homotopic in $C_2$ to the trivial path, so the induced map $\pi_1(C_1)\to \pi_1(C_2)$ is trivial. Thus, we verified all hypotheses
of Theorem~\ref{thm:Br87} and hence $H$ is finitely presented.
\end{proof}

We finish this subsection with two important definitions.
\vskip .1cm

\paragraph{\bf Super-reduced presentations}

\begin{Definition}\rm
Let $(X,R)$ be a presentation of a group $G$. We will say that $r\in R$
is {\it super-reduced} if it is cyclically reduced and
any non-trivial proper subword of $r$ represents a non-trivial element of $G$. We will say that $(X,R)$ is {\it super-reduced} if every $r\in R$ is super-reduced.
\end{Definition}

From now on we will assume that the presentation $(X,R)$ is super-reduced. This is not a major restriction.
Indeed, if $(X,R)$ is any presentation of $G$ and some $r\in R$ is not super-reduced, then some
cyclic permutation of $r$ can be written as $r_1r_2$ where $r_1$ and $r_2$ are relators of $G$
shorter than $r$. Replacing each such $r$ by the corresponding pair $\{r_1,r_2\}$ and repeating the procedure if
$r_1$ or $r_2$ is not super-reduced, we obtain a super-reduced presentation of $G$.

The importance of having a super-reduced presentation $(X,R)$ is that for any van Kampen diagram $\Omega$ over $(X,R)$
and any cell $\calF$ of $\Omega$ with $n$ vertices, the associated map from the topological $n$-gon $P_n$ to $\calF$ is 
a homeomorphism on the entire $P_n$ and not just its interior.
\skv

\paragraph{\bf Cancellation of cells.}

\begin{Definition}\rm
Let $\Omega$ be a disk van Kampen diagram, and let $\calF_1$ and $\calF_2$ be distinct cells of $\Omega$. We will say that
the pair $(\calF_1,\calF_2)$ is {\it cancellable} if the following conditions hold:
\begin{itemize}
\item[(i)] the intersection $\partial\calF_1\cap \partial\calF_2$ contains at least one edge $e$;
\item[(ii)] the boundary labels of $\calF_1$ and $\calF_2$ are the same if they are read starting from some fixed vertex
$v\in \partial\calF_1\cap \partial\calF_2$ and $\partial\calF_1\cap \partial\calF_2$ is traversed in the same direction;
\item[(iii)] the intersection $(\partial\calF_1\cup \partial\calF_2)\cap \partial\Omega$ is contained in $\partial\calF_i$
for some $i=1,2$.
\end{itemize}
\end{Definition}

For any cancellable pair $(\calF_1,\calF_2)$ we can construct a new disk diagram $\Omega'$
with the same boundary label. If the intersection $\partial\calF_1\cap \partial\calF_2$ is connected, we simply remove the cells 
$\calF_1$ and $\calF_2$ and identify the parts of their boundaries away
from their intersection (this is possible since we are assuming that $R$ is super-reduced and hence the boundaries 
$\partial F_1$ and $\partial F_2$ are not self-intersecting). 

In general we let $\Delta$ be the smallest disk subdiagram containing of $\Omega_1\cap\Omega_2$.
One can deduce from our hypotheses that $l(\partial\Delta)$ represents the trivial element of $F(X)$; more specifically, 
$\partial\Delta=(\partial\Delta\cap \partial F_1)\cup(\partial\Delta\cap \partial F_2)$,
$\partial\Delta\cap \partial F_1$ and $\partial\Delta\cap \partial F_2$ only intersect at the endpoints,
and $l(\partial\Delta\cap \partial F_2)=l(\partial\Delta\cap \partial F_2)^{-1}$ if both paths are traversed in the same
direction starting from the same point in $\partial\Delta\cap \partial F_1\cap \partial F_2$. We now construct a new diagram
by removing the interior of $\Delta$ and identifying $\partial\Delta\cap \partial F_1$ and $\partial\Delta\cap \partial F_2$.

 The operation we just described will be referred to as {\it cancellation} of the pair $(\calF_1,\calF_2)$ --
see Figure~\ref{cancel} for an illustration. The term {\it reduction} commonly used for such operation will have a completely different meaning in this paper (see Definition~\ref{def:reduction} in the next subsection).

\input{figure_cancellation.tex}

\begin{Remark}\rm The technical condition (iii) is necessary to ensure that in the process of identifying the boundaries of
$\calF_1$ and $\calF_2$ we do not glue distinct edges of $\partial\Omega$ (since in the latter case the new diagram will
have a different boundary label and in fact will not even be a disk diagram). Typically one simply requires that 
$\calF_1$ and $\calF_2$ are interior cells (in which case (iii) is automatic).
\end{Remark}

\subsection{Diagram maps and reductions}
\label{sec:diagrams}

Throughout this section we fix the following objects and notations:
\begin{itemize}
\item a group $G$ given by a finite super-reduced presentation $(X,R)$;
\item a finitely generated subgroup $H$ of $G$;
\item $Z=G/H$ considered as a left $G$-set. We denote by $\rho:G\to Z$ and $\theta:F(X)\to Z$ the natural projections.
\end{itemize}

By a {\it diagram} we will always mean a $Z$-labeled disk van Kampen diagram relative to $(X,R)$ (the labeling will not
play a role in some of the definitions). Starting with the definition of a {\it reduction} later in this subsection 
(Definition~\ref{def:reduction}), 
we will also assume that $Z$ is endowed with a {\it super-Artinian} partial order (an Artinian partial order satisfying an additional condition -- see Definition~\ref{def:Artinian} below).
\skv
The sets of vertices and edges of a diagram $\Omega$ will be denoted by $V(\Omega)$ and $E(\Omega)$, respectively.

\begin{Definition}\rm
Let $\Omega$ be a diagram and $\Delta$ a disk subdiagram of $\Omega$. Let $\Delta'$ be another diagram
with $\partial \Delta'=\partial \Delta$. Consider the diagram obtained from $\Omega$ by replacing $\Delta$ by
$\Delta'$ (that is, by first removing $\Delta$ and then gluing $\Delta'$ to $\Omega\setminus\Delta$ along $\partial\Delta$). In this case
we will say that $\Omega'$ is obtained from $\Omega$ by a {\it $\Delta$-map} and symbolically write $(\Omega,\Delta)\to (\Omega',\Delta')$ or simply $\Omega\to \Omega'$. 
\begin{itemize}
\item By a {\it diagram map} we will mean a $\Delta$-map for some $\Delta$.
\item If $\phi:(\Omega,\Delta)\to (\Omega',\Delta')$ is a diagram map, 
\begin{itemize}
\item the subdiagram  $\Delta$ will be called the {\it domain} of $\phi$ and 
\item the subdiagram $\Delta'$ will be called the {\it replacement diagram} of $\phi$. We will also symbolically write 
$\Delta'=\phi(\Delta)$.
\end{itemize}
\end{itemize}
\end{Definition}

It is clear that any diagram map does not change the boundary of the diagram. Any two
diagrams $\Omega$ and $\Omega'$ with the same boundary can be obtained from each other by a diagram map (since we can take
$\Delta=\Omega$).  We will be primarily using $\Delta$-maps in the case where all the cells of $\Delta$
share a common vertex. Of particular importance to us will be {\it single-cell} maps.

\begin{Definition}\rm
\label{def:singlecellmap}
Let $\phi:(\Omega,\calF)\to (\Omega',\Delta')$ be a diagram map. We will say that $\phi$ is a {\it single-cell map} if
\begin{itemize}
\item[(i)] $\calF$ is a cell of $\Omega$;
\item[(ii)] for each vertex $v$ of $\calF$ there exists a unique edge $e(v)\in E(\Delta')\setminus E(\calF)$ which contains $v$.
Edges of the form $e(v)$  will be called the {\it side edges} of $\phi$.
\end{itemize}
Condition (ii) ensures that for each edge $e$ of $\calF$, the diagram $\Delta'$ has a unique
cell $\calF(e)$ containing $e$. Cells of this from will be called the  {\it side cells} of $\phi$. 
\end{Definition}
\begin{Remark}
\rm
It would probably be more accurate to call a map of this form a {\it single-cell refinement}. Likewise it might be more natural to drop condition
(ii) from Definition~\ref{def:singlecellmap}. However, we decided to stick with a simpler name and more restrictive definition given how frequently such maps will appear in the paper. 
\end{Remark}

\paragraph{\bf Commuting and conjugating single-cell maps.}
We now introduce a simple way to construct single-cell maps.

\begin{Definition}\rm
\label{def:commutingmap}
Let $\calF$ be a cell of a diagram $\Omega$, let $t\in X$, and suppose that $t$ commutes with every edge label $e$
of $\calF$ and all the relators $[t,e]$ lie in $R$.
Define the $\calF$-map $\phi:\Omega\to\Omega'$ as follows:
\begin{itemize}
\item[(1)] all side edges of $\phi(\calF)$ (the replacement diagram) are labeled by $t$ and all point in the same direction;
\item[(2)] for every edge $e$ of $\calF$, the corresponding side cell $\calF(e)$ has boundary label $e^{-1}t^{-1}et$;
\item[(3)] $\phi(\calF)$ has a unique interior cell with the same boundary label as $\calF$.
\end{itemize} 
The map $\phi$ will be called a {\it commuting $\calF$-map} and denoted by
$C_{\calF}(t^{\eps})$ where $\eps=1$ (resp. $\eps=-1$) if all side edges of $\phi(\calF)$ point away from (resp. towards) $\partial F$.
See Figure~\ref{fig:commut} for an illustration.
\end{Definition} 

\input{figure_commuting.tex}

\paragraph{\bf Conjugating maps.}
Let us now describe a certain generalization of commuting $\calF$-maps called {\it conjugating $\calF$-maps}.

Fix a diagram $\Omega$, its cell $\calF$ and $t\in X$, and suppose that for every edge label
$e$ of $\calF$ there exists a word $w_{t,e}\in F(X)$ such that $w_{t,e}=t^{-1}et$ is a defining relation
(that is, $w_{t,e}^{-1}t^{-1}et\in R$). Let us now remove $\calF$ from $\Omega$,
add side edges as in (1) above, and for each edge $e$ of $\partial\calF$ add the side cell $\calF(e)$ with boundary label $w_{t,e}^{-1}t^{-1}et$.
This produces an annular diagram $\Omega_{1}$ whose inner boundary $\partial_{inn}\Omega_{1}$ has
label $r_1=\prod_{e\in \partial \calF}w_{t,e}$ (note that $r_1$ need not be cyclically reduced or even reduced).

Suppose first that $r_1$ is super-reduced. In this case we simply choose any diagram $\calG_1$ with $l(\partial \calG_1)=r_1$ 
(note that $\calG_1$ must be a disk diagram since $r_1$ is super-reduced) and use it to fill the hole in $\Omega_{1}$,
thereby producing a disk diagram $\Omega'$. 
\skv
In general we proceed as follows. First suppose that $r_1$ is not reduced. Then $\partial_{inn}\Omega_{1}$ contains
two consecutive edges $e$ and $e'$ with the same label and opposite orientation. We can then produce a new diagram
identifying $e$ and $e'$. Applying this operation several times, we produce an annular diagram 
$\Omega_{2}$ whose inner boundary label $r_2$ is reduced (and so are its cyclic shifts) and thus cyclically reduced.
If $r_2$ is super-reduced, we are done as before, so assume that $r_2$ is not super-reduced. We will fill the hole
in $\Omega_2$ in several steps.
\skv
Since $r_2$ is not super-reduced, $\partial_{inn}\Omega_{2}$ contains a proper non-empty edge
subpath $\gamma'$ with edges $e_1,\ldots, e_m$ such that $r'=\prod_{i=m}^1 l(e_i)$
is a relator of $G$, and if we assume that $m$ is smallest possible, then $r'$ is super-reduced. Since $r'$ is a relator
of $G$, we can produce a new diagram identifying the endpoints of $\gamma'$ (dividing the hole in $\Omega_2$ in two parts)
and then fill the hole whose boundary relator is $r'$ with a disk diagram $\calG'$ such that $l(\partial \calG')=r'$. 
Repeating this operation several times, we eventually fill the entire hole in $\Omega_2$, as desired.
\skv

\paragraph{\bf Super-Artinian orders.}
For the remainder of this section we will assume that $Z$ is endowed with a {\it super-Artinian} (as defined below) partial order $<$.

\begin{Definition}\rm 
\label{def:Artinian}
A partial order $<$ on a set $P$ will be called 
\begin{itemize}
\item[(a)] {\it Artinian} if $P$ does not have infinite strictly descending chains;
\item[(b)] {\it super-Artinian} if it is Artinian and whenever
$p,q\in P$ are incomparable we have $\{r\in P: r<p\}=\{r\in P: r<q\}$ and $\{r\in P: r>p\}=\{r\in P: r>q\}$;
\item[(c)] {\it strongly Artinian} if it is super-Artinian and for any $p\in P$ the set $\{r\in P: r<p\}$ is finite.
\end{itemize}
\end{Definition}
It is easy to check that super-Artinian orders on a set $P$ are precisely the orders of the following form: choose some well-ordered
set $P_0$ and a map $N:P\to P_0$ (the ``norm'' function), and for $p,q\in P$ define $p<q$ if and only if $N(p)<N(q)$. 
The order is strongly Artinian if in addition we can require that $N$ has finite fibers and $P_0\cong\dbN$ as posets.

It might be convenient to visualize super-Artinian orders in this way, although it will usually be unnecessary 
to define the norm function explicitly.
\skv

Given vertices $v$ and $w$ of some diagram, we will write $v<w$ or $v\leq w$ if their labels satisfy the corresponding inequality.
A {\it maximal vertex} of a diagram $\Omega$ is any vertex of $\Omega$ whose label is a maximal element of the set of labels
of the vertices of $\Omega$. 
\skv
\paragraph{\bf Diagram reductions.} We now define the key notion of a diagram reduction (with respect to the chosen partial order $<$).

\begin{Definition}\rm 
\label{def:reduction}
Let $\phi: \Omega\to \Omega'$ be a diagram map and $v$ a maximal vertex of $\Omega$. We will say that
\begin{itemize}
\item[(i)] $\phi$ is a {\it reduction} if $w<v$ for every vertex $w\in V(\Omega')\setminus V(\Omega)$; 
\item[(ii)] Given a maximal vertex $v\in V(\Omega)$, $\phi$ is a {\it full reduction at $v$} if $\phi$ is a reduction which also eliminates $v$.
\end{itemize}
By a full reduction we will mean a full reduction at some maximal vertex.
If $\Delta$ is a disk subdiagram of $\Omega$, a $\Delta$-reduction is a $\Delta$-map which is a reduction.
\end{Definition}
Condition (b) in the definition of a super-Artinian order (Definition~\ref{def:Artinian}) implies that 
the definition of a reduction (Definition~\ref{def:reduction}) does not depend
on the choice of a maximal vertex $v$. Moreover, the following holds:
\begin{Observation}
\label{obs:red} A diagram map $\phi: \Omega\to \Omega'$ is a reduction if and only if for every  $w'\in V(\Omega')\setminus V(\Omega)$ there exists $w\in V(\Omega)$ with $w'<w$.
\end{Observation}

The next key lemma uses only the fact that $<$ is Artinian. 

\begin{Lemma}\rm
\label{obs:Artinian}
Any diagram $\Omega$ cannot admit an infinite sequence of full reductions. Moreover, $\Omega$ cannot admit
a sequence of reductions which includes infinitely many full reductions.
\end{Lemma}
\begin{proof} First note that the composition of reductions is a reduction, and the composition of a full
reduction and a reduction (in either order) is a full reduction. Thus, it suffices to prove the first assertion
of Lemma~\ref{obs:Artinian}.

Assume that, on the contrary, some diagram $\Omega$ admits an infinite sequence of full reductions
$\Omega=\Omega_1\to \Omega_2\to \ldots$, so that for each $n$ there exists a vertex $v_n\in V(\Omega_{n})\setminus V(\Omega_{n+1})$ which is maximal for $\Omega_n$ (if there is more than one such vertex,
we choose $v_n$ arbitrarily).
Construct the graphs $\{\Gamma_n\}_{n=1}^{\infty}$ inductively as follows.
The graph $\Gamma_1$ has no edges and its vertices are exactly the vertices of $\Omega_1$. Suppose now
that $\Gamma_n$ is defined for some $n\geq 1$. We define $\Gamma_{n+1}$ by adding some vertices and edges to $\Gamma_n$.
The extra vertices of $\Gamma_{n+1}$ are precisely the vertices in $V(\Omega_{n+1})\setminus V(\Omega_n)$, 
and we add a directed edge from $v_n$ to $w$ for every $w\in V(\Omega_{n+1})\setminus V(\Omega_{n})$. 
Finally, define $\Gamma$ as the union of all $\Gamma_n$.

Every vertex of $\Gamma$ is connected by a path to a vertex of $\Gamma_1$, so $\Gamma$ is an infinite graph with finitely many connected components and thus must have an infinite
connected component $C$. Since for every $n\in\dbN$ the map $\Omega_n\to\Omega_{n+1}$ is a reduction at $v_n$, each vertex of $\Gamma$ has finite degree,
and therefore $C$ must have an infinite directed path. But by construction each vertex in a directed path in $\Gamma$
is smaller than the previous one, so we obtain an infinite descending chain in $(Z,<)$, contrary to the assumption
that $<$ is Artinian.
\end{proof}

Our next result (Lemma~\ref{obs:Renz}) is a simple consequence of Proposition~\ref{thm:criterion} and provides a criterion
for finite presentability of $H$ in terms of reductions. Note that unlike Proposition~\ref{thm:criterion},
condition (*) in Lemma~\ref{obs:Renz} does not include any assumptions on the boundary relator of the diagram.

\begin{Lemma}
\label{obs:Renz}
Let $H$ be a finitely generated subgroup of $G$, and assume that the chosen order on $Z=G/H$ is strongly Artinian.
Suppose that there exists a finite subset $B$ of $Z$ 
with the following property:
\begin{itemize}
\item[(*)] If $\Omega$ is any $Z$-labeled diagram which has an interior maximal vertex $v$ outside of $B$, then $\Omega$ admits a full reduction.
\end{itemize}
Then $H$ is finitely presented.
\end{Lemma}
\begin{proof} Choose any finite $A\subseteq Z$ such that $\rho^{-1}(A)$ is connected (since $H$ is finitely generated,
such $A$ exists by \cite[Theorem~2.12]{EF}). Clearly we can make $B$ larger (as long as it remains finite).
Thus we can assume that $A\subseteq B$; further, since the order on $Z$ is strongly Artinian, we can assume that
$B=\{z\in Z: z<z_0\}$ for some $z_0\in Z$.

By Proposition~\ref{thm:criterion}, to
prove that $H$ is finitely presented it suffices to show that for any simple closed path $\gamma$ in $Cay(G,X)$
all of whose vertices lie in $\rho^{-1}(A)$ there exists a $Z$-labeled diagram $\Omega$
with $\partial\Omega=[\gamma]_Z$ all of whose vertices lie in $B$. 

Start with any $Z$-labeled diagram $\Omega_0$ with $\partial\Omega_0=[\gamma]_Z$.
If all the vertices of $\Omega_0$ lie in $B$, we are done, so suppose that some vertex $v_0$ does not lie in $B$.
Since $\partial\Omega_0=[\gamma]_Z$, all boundary vertices of $\Omega_0$ lie in $A$, and by assumption $A\subseteq B$.
Thus, $v_0$ must be an interior vertex. While $v_0$ need not be maximal, since $B=\{z\in Z: z<z_0\}$,
any maximal vertex $v$ of $\Omega_0$ also lies outside of $B$ (and hence is also interior).
Thus by (*) $\Omega_0$ admits a full reduction $\Omega_0\to\Omega_1$.
 
Apply the same procedure to $\Omega_1$ and keep going as long as we can. By Lemma~\ref{obs:Artinian}, after finitely many steps we will obtain a diagram $\Omega$ with 
$\partial\Omega=[\gamma]_Z$ which does not admit a full reduction. But this means that all the vertices of $\Omega$ lie in $B$, as desired.
\end{proof}

\paragraph{\bf How to construct full reductions.} Note that a single-cell map does not eliminate any vertices from the diagram and in particular cannot be a full reduction. However, as we will explain next, one can eliminate a vertex $v$ by composing single-cell maps at all cells adjacent to $v$ followed by suitable cancellations.

\begin{Definition}\rm 
\label{def:gallery}
Let $v$ be a vertex of a diagram $\Omega$ and $k\in\dbN$. A subdiagram $\Delta$ of $\Omega$
will be called a {\it gallery of length $k$ at $v$} if $\Delta$ has $k$ cells, all of which contain $v$, and there exists an ordering
$\calF_1,\ldots, \calF_k$ of the cells of $\Delta$ such that for each $1\leq i\leq k-1$, the cells $\calF_{i}$ and $\calF_{i+1}$ share at least one edge
which contains $v$. If $\Delta$ is the union of all cells containing $v$, we will say that $\Delta$ is the {\it full gallery} at $v$.
\end{Definition}

\begin{Definition}\rm 
\label{def:gallerycompatible}
Let $\Omega$ be a diagram, and let $\calF_1$ and $\calF_2$ be distinct cells of $\Omega$ which have a common edge $e$. Suppose that we defined single-cell maps
$\phi_i:(\Omega,\calF_i)\to (\Omega_i',\Delta_i)$ for $i=1,2$. We will say that $\phi_1$ and $\phi_2$ are {\it compatible at $e$}
if the side cells at $e$ arising from $\phi_1$ and $\phi_{2}$ form a cancellable pair.
\end{Definition}

\begin{Definition}\rm 
\label{def:gallerymap}
Let $\Omega$ be a diagram, $v$ a vertex of $\Omega$ and $\Delta$ a gallery of length $k$ at $v$ with
cells $\calF_1,\ldots, \calF_k$, ordered as in Definition~\ref{def:gallery}. For each $1\leq i\leq k$ we fix an
edge $e_i\in E(\calF_i)\cap E(\calF_{i+1})$ containing $v$. Suppose that
\begin{itemize}
\item[(i)] for each $1\leq i\leq k$ we defined a single-cell map $\phi_i:(\Omega,\calF_i)\to (\Omega_i',\Delta_i)$;
\item[(ii)] for each $1\leq i\leq k-1$ the maps $\phi_i$ and $\phi_{i+1}$ are compatible at $e_i$.
\end{itemize}
Define $\phi=\cup\phi_i$ as follows: first apply each $\phi_i$ individually and then
cancel all the cancellable pairs arising from condition (ii).

A map $\phi$ of this form will be called a {\it $k$-fold cell map at $v$} ({\it double-cell map} for $k=2$ and
{\it triple-cell map} for $k=3$). We will also say that $\phi$ is a {\it gallery $\Delta$-map at $v$}.
\end{Definition}

\begin{Remark}\rm It is clear that $\phi$ is a reduction if and only if each $\phi_i$ is a reduction.
\end{Remark}

\begin{Definition}\rm Suppose that in Definition~\ref{def:gallerymap} each $\phi_i$ is a commuting (resp. conjugating) map 
$C_{\calF_i}(t^{\eps})$ for the same $t$ and $\eps$, in which case condition (ii) holds automatically. 
The corresponding map $\phi$ will be called a 
{\it commuting $\Delta$-map} (resp. {\it conjugating $\Delta$-map}) and denoted by $C_{\Delta}(t^{\eps})$.
\end{Definition}

If $\Delta$ is the full gallery at an interior vertex $v$ and $\phi$ is a gallery $\Delta$-map at $v$, then $\phi$ eliminates $v$
(see Figure~\ref{fullreduction} for an illustration), so if $\phi$ is also a reduction, it must be a full reduction at $v$. 
In \cite{Re}, it is shown that if $G/H$ is free abelian and $H$ happens to be finitely presented, one can establish finite presentability of $H$ via Lemma~\ref{obs:Renz} by using only full reductions of this form. In our setting, we will need to work with more complex full reductions constructed as compositions of several gallery reductions.

\input{figure_fullreduction.tex}

\begin{Lemma}
\label{lem:kgallery}
Let $v$ be an interior vertex of a diagram $\Omega$, let $\Delta$ be a gallery of length $k$ at $v$, and let
$\phi:\Omega\to \Omega'$ be a gallery $\Delta$-map. Assume that $k<\deg(v)$ (that is, $\Delta$ is not the full gallery at $v$), so that
$\phi$ does not eliminate $v$. Then $\deg_{\Omega'}(v)-\deg_{\Omega}(v)\leq 2-k$.
In particular, $\deg_{\Omega'}(v)<\deg_{\Omega}(v)$ if $k\geq 3$.
\end{Lemma}
\begin{proof}
Recall that $\phi$ is constructed by first composing single-cell maps at each cell of $\Delta$, followed by removing
certain pairs of cells which share an edge containing $v$. Each single-cell map at $v$ increases the degree of $v$ by $1$,
and each cancellation decreases the degree of $v$ by $2$. By definition of a gallery of length $k$, we will perform at least
$k-1$ cancellations. Therefore, $\deg_{\Omega'}(v)-\deg_{\Omega}(v)\leq k-2(k-1)=2-k$.
\end{proof}

Lemma~\ref{lem:kgallery} shows that if for some interior vertex $v$ we can always construct a $\Delta$-reduction at $v$
for some gallery $\Delta$ at $v$ of length $\min(3,\deg(v))$, then we can eliminate $v$ by applying Lemma~\ref{lem:kgallery} enough times. Hence if $v$ is also maximal, there exists a full reduction at $v$. This observation (which generalizes the basic idea of \cite{Re}) will still be insufficient for our purposes.
Our general strategy for removing a vertex $v$ from a diagram will be as follows (the actual algorithm will be even more involved,
and at this point we just want to outline the main idea):

Fix a maximal vertex $v$, and assume that $v$ is interior.
\begin{itemize}
\item[(i)] We first apply several single-cell reductions to remove cells with ``complicated'' boundary labels at $v$.
The total number of cells containing $v$ will increase at this stage. 
\item[(ii)]Then we apply double-cell reductions to remove edges with complicated labels containing $v$.
\item[(iii)] Once all cells and edges at $v$ with complicated labels have been eliminated, we will always be able to define
a triple-cell reduction (which will not reintroduce new complicated cells and edges). Since 
triple-cell reductions decrease the degree of $v$, after finitely many steps $v$ will be eliminated from the diagram.
\end{itemize}

\subsection{Group and Relations}
\label{sec:groupandrel}

While our main goal in this paper is to prove finite presentability of the subgroups
of $\IAR_{n,d}$ and $\IAC_{n,d}$ of $\Aut(F_n)$, it will be more convenient to work with the corresponding subgroups of the group $\widetilde{\SAut(F_n)}$ defined below. As usual, $\SAut(F_n)$
is the preimage of $\SL_n(\dbZ)$ under the natural homomorphism $\Aut(F_n)\to \GL_n(\dbZ)$; thus it is an index $2$ subgroup of
$\Aut(F_n)$. We define $\widetilde{\SAut(F_n)}$ to be the universal central extension
of $\SAut(F_n)$. 

Gersten~\cite{Ge} gave a very simple Steinberg-type presentation of $\SAut(F_n)$ for all $n\geq 3$. Moreover, he proved that
if $n\geq 5$, one can obtain a presentation of $\widetilde{\SAut(F_n)}$ from that presentation of $\SAut(F_n)$ by dropping 
one family of relations and that the kernel of the projection 
$\widetilde{\SAut(F_n)}\to \SAut(F_n)$ is cyclic of order $2$. We define
$\widetilde{\IAR_{n,d}}$ (resp. $\widetilde{\IAC_{n,d}}$) to be the preimage
of $\IAR_{n,d}\cap \SAut(F_n)$ (resp. $\IAC_{n,d}\cap \SAut(F_n)$) in $\widetilde{\SAut(F_n)}$.
Clearly, if $n\geq 5$, finite presentability of $\widetilde{\IAR_{n,d}}$ (resp. $\widetilde{\IAC_{n,d}}$)
implies that of $\IAR_{n,d}$ (resp. $\IAC_{n,d}$).
\skv

We proceed with describing Gersten's presentations for $\SAut(F_n)$ and $\widetilde{\SAut(F_n)}$ established in \cite{Ge}. 
This presentation uses the standard generating set consisting of Nielsen maps $R_{ij}: x_i\mapsto x_i x_j$ and
$L_{ij}: x_i\mapsto x_j x_i$ with $i\neq j$, but it is easiest to describe it by adding extra generators $w_{ij}$ 
given by $w_{ij}=L_{ij}L_{ji}^{-1}R_{ij}$.

\begin{Theorem}[Gersten~\cite{Ge}] 
\label{Gersten_pres}
The group $\SAut(F_n)$ has a presentation with a generating set 
$\{R_{ij},L_{ij},w_{ij}: 1\leq i\neq j\leq n\}$ subject to the following relations.
Here $T_{ij}$ stands for either $R_{ij}$ or $L_{ij}$ and in each relation distinct letters denote distinct indices:
\begin{align}
\label{row1}
&[R_{ij},L_{ij}]=1&
&[T_{ij},T_{kl}]=1&
&[T_{ij},T_{kj}]=1&
&&
\\
\label{row2}
&[R_{jk},R_{ij}]=R_{ik}&
&[R_{ij},R_{jk}^{-1}]=R_{ik}&
&[R_{ij}^{-1}, L_{jk}]=R_{ik}&
&[L_{jk}^{-1}, R_{ij}]=R_{ik}& \\
\label{row3}
&[L_{jk},L_{ij}]=L_{ik}&
&[L_{ij},L_{jk}^{-1}]=L_{ik}&
&[L_{ij}^{-1}, R_{jk}]=L_{ik}&
&[R_{jk}^{-1}, L_{ij}]=L_{ik}& \\
\label{row4}
&w_{ij}=L_{ij}L_{ji}^{-1}R_{ij}&
&w_{ij}=R_{ij}R_{ji}^{-1}L_{ij}&
&&
&&
\\
\label{row5}
&w_{ij}^{-1}R_{ij}w_{ij}=L_{ji}^{-1}&
&w_{ij}^{-1}L_{ij}w_{ij}=R_{ji}^{-1}&
&w_{ij}^{-1}R_{ji}w_{ij}=R_{ij}^{-1}&
&w_{ij}^{-1}L_{ij}w_{ij}=L_{ij}^{-1}&
\\
\label{row6}
&w_{ij}^{-1}R_{ik}w_{ij}=L_{ik}^{-1}&
&w_{ij}^{-1}L_{ik}w_{ij}=R_{ik}^{-1}&
&w_{ij}^{-1}R_{ki}w_{ij}=R_{kj}^{-1}&
&w_{ij}^{-1}L_{ki}w_{ij}=L_{kj}^{-1}&
\\
\label{row7}
&w_{ji}=w_{ij}^{-1}&
&&
&&
&&
\\
\label{row8}
&w_{ij}^4=1&
&&
&&
&&
\end{align}
In addition, the following hold for $n\geq 5$:
\begin{itemize}
\item[(a)] The group $\widetilde{\SAut(F_n)}$ can be defined by the presentation obtained from the one above by removing 
the last family of relations ($w_{ij}^4=1$).
\item[(b)]The kernel of the
projection $\widetilde{\SAut(F_n)}\to\SAut(F_n)$ is cyclic of order $2$. 
\end{itemize}
\end{Theorem}
\begin{Remark}\rm 1. Presentation in \cite{Ge} actually has fewer relations -- 
\eqref{row5},\eqref{row6}~and~\eqref{row7} follow from the other relations; however, it will be
convenient for us to include them to simplify the description of reduction maps.

2. Gersten's presentation was derived from McCool's presentation for $\Aut(F_n)$ given in \cite{Mc1}, but the latter has a much larger generating set consisting of Whitehead maps. These are automorphisms of the form $T_{i,A,B}$ where $A$ and $B$ are disjoint subsets of $\{1,\ldots, n\}\setminus \{i\}$
and 
$$T_{i,A,B}(x_j)=\left\{
\begin{array}{ll}
x_j x_i &\mbox{ if } j\in A\setminus B; \\
x_i^{-1} x_j &\mbox{ if } j\in B\setminus A; \\
x_i^{-1} x_j x_i &\mbox{ if } j\in A\cap B; \\
x_j &\mbox{ if } j\not\in A\cup B.
\end{array}
\right.
$$ 
\end{Remark}

\paragraph{\bf Extended Nielsen generators.} We will refer to the generating set from Theorem~\ref{Gersten_pres} as the {\it extended set of Nielsen generators}. The elements $R_{ij}$ and $L_{ij}$ will be called {\it Nielsen generators} and the elements $w_{ij}$
will be called {\it Weyl generators}.
\skv

\paragraph{\bf Optimized Gersten's presentation.} Figure~\ref{5++GA} shows that the relations $[R_{ij},L_{ij}]=1$ actually follow from other relations involving Nielsen generators and therefore can (and will) be removed. 
It will also be convenient for us to remove half of the Weyl generators -- namely we will keep only the generators $w_{ij}$ with
$i<j$, thereby removing $w_{ij}$ for $i>j$ along with the relations in \eqref{row7} (and replacing
$w_{ij}$ by $w_{ji}^{-1}$  for $i>j$ in other relations). 

Thus, for $n\geq 5$ the group $\widetilde{\SAut(F_n)}$ has a presentation $\la \calX | \calR\ra$ where
$\calX=\{R_{ij},L_{ij}: 1\leq i\neq j\leq n\}\sqcup\{w_{ij}: 1\leq i<j\leq n\}$
and $\calR$ is the set of all relations in Theorem~\ref{Gersten_pres} except $[R_{ij},L_{ij}]=1$, \eqref{row7} and \eqref{row8}
(where $w_{ij}$ for $i>j$ should be interpreted as $w_{ji}^{-1}$ in \eqref{row4},\eqref{row5} and \eqref{row6}).
The obtained presentation of $\widetilde{\SAut(F_n)}$
will be called the {\it optimized Gersten's presentation}. 
\skv
\input{figure5++GA.tex}

\skv
\paragraph{\bf Action of the signed symmetric group.} 
Now let $\Sigma_n=S_n\rtimes \{\pm 1\}^n$ be the signed symmetric group of degree $n$, and consider the natural action of $\Sigma_n$
on the free group $F_n=F(x_1,\ldots, x_n)$. This action induces an action of $\Sigma_n$ on $\Aut(F_n)$ (by conjugation) which preserves
$\calX\cup \calX^{-1}$ where $\calX=\{R_{ij},L_{ij}: 1\leq i\neq j\leq n\}\sqcup\{w_{ij}: 1\leq i<j\leq n\}$ as above. Thus we can extend this action to the free group $F(\calX)$. The obtained action on $F(\calX)$ almost preserves the set $\calR$ of relations from optimized Gersten's presentation; more precisely it does preserve $\calR$ if we view elements of $\calR$ as {\it geometric relations} in the following
sense:

\begin{Definition}\rm
Let $X$ be a set. A {\it geometric relator} on $X$ is an equivalence class of cyclically reduced words in $X\sqcup X^{-1}$
where two words $v$ and $w$ are equivalent if $w$ can be obtained from $v$ or $v^{-1}$ by a cyclic permutation.
\end{Definition}

\begin{Remark}\rm
1. The reason we call such equivalence classes geometric relators is the following. If $\Omega$ is a van Kampen diagram with a cyclically 
reduced boundary label, then the set of all possible ways to read the 
boundary label of $\Omega$ (where the starting point and the direction can be chosen arbitrarily) form a geometric relator.

2. Let $u,v\in F(X)$ be reduced words which start with different symbols and end with different
symbols. Then the 4 words $u^{-1}v, uv^{-1}, v^{-1}u, vu^{-1}$ are cyclically reduced and all lie in the same equivalence class.
Thus, we can safely write geometric relators in the form $u=v$ (rather than $w=1$) without ambiguity.
\end{Remark}

It will be convenient to write as many relators as possible in the form $u^{-1}vu=f(u,v)$
where $u$ and $v$ are generators or their inverses and $f(u,v)$ is a short word since this will simplify
construction of conjugating single-cell maps (see the paragraph after Definition~\ref{def:commutingmap} in \S~2.2). Of course, it is trivial to write relations of the
form $[u,v]=1$ (where $u,v$ are generators) in this way. 
Below we give such expressions for relations in \eqref{row2} and \eqref{row3}. In fact, we give two different expressions for some of these relations, swapping the roles of $u$ and $v$.

\begin{Observation}
\label{rel:conj}
The following relations hold:
\begin{align*}
&R_{jk}^{-1}R_{ij}R_{jk}=R_{ij} R_{ik}^{-1}& 
&R_{ij}^{-1}R_{jk}R_{ij}=R_{jk} R_{ik}&
&R_{jk} R_{ij}R_{jk}^{-1}=R_{ij} R_{ik}&
&R_{ij}^{-1}R_{jk}R_{ij}=R_{ik} R_{jk}&
\\
& L_{jk}^{-1} R_{ij}L_{jk}=R_{ik}^{-1} R_{ij}&
&R_{ij}L_{jk}R_{ij}^{-1}=L_{jk} R_{ik}^{-1}&
& L_{jk} R_{ij} L_{jk}^{-1}=R_{ik} R_{ij}&
&R_{ij}L_{jk}R_{ij}^{-1}= R_{ik}^{-1}L_{jk}&
\\
&L_{jk}^{-1}L_{ij}L_{jk}=L_{ij} L_{ik}^{-1}&
&L_{ij}^{-1}L_{jk}L_{ij}=L_{ik} L_{jk}&
&L_{jk} L_{ij}L_{jk}^{-1}=L_{ij} L_{ik} &
&L_{ij}^{-1}L_{jk}L_{ij}=L_{jk} L_{ik}&
\\
&R_{jk}^{-1}L_{ij}R_{jk}=L_{ik}^{-1}L_{ij}&
&L_{ij}R_{jk}L_{ij}^{-1}=R_{jk} L_{ik}^{-1}&
&R_{jk} L_{ij}R_{jk}^{-1}=L_{ik} L_{ij}&
&L_{ij}R_{jk}L_{ij}^{-1}= L_{ik}^{-1}R_{jk}&
\end{align*}
\end{Observation}

\subsection{Actions on the diagrams and the partial order}

Fix integers $n\geq 2$ and $1\leq d\leq n$, let $G=\widetilde{\SAut(F_n)}$, and let $H$ be either 
$\widetilde{\IAC_{n,d}}$ or $\widetilde{\IAR_{n,d}}$.
Our goal is to prove finite presentability of $H$ (for $d$ as in Theorem~\ref{thm:finpres}) using the general method described in 
\S~\ref{sec:diagrams}. Recall that in the setting of \S~\ref{sec:diagrams}, the vertices of the van Kampen diagrams 
are labeled by elements of the space $G/H$ which was denoted by $Z$ in \S~\ref{sec:diagrams}. For the discussion below
it will be convenient to define $Z$ not as $G/H$ itself, but as a simple-to-describe $G$-set which is isomorphic to $G/H$ (as a $G$-set).
\skv

Given $g\in G$, let $[g]$ be its natural projection to $\SL_n(\dbZ)$. There are two natural actions of $\SL_n(\dbZ)$ on the set 
$Mat_{d\times n}(\dbZ)$ of $d\times n$ matrices over $\dbZ$. Using the projection $g\mapsto [g]$, we will view these as actions of $G$:
\begin{gather}
\label{eq:actionrow}
g.A= A [g]^{-1} \qquad \mbox{(row action)} \\
\label{eq:actioncolumn}
g.A= A [g]^{T} \qquad \mbox{(column action).}
\end{gather}
Let $I_{n,d}=(I_d | 0_{d\times n})\in Mat_{d\times n}(\dbZ)$ be the matrix whose first $d$ columns form the identity $d\times d$
matrix and whose last $n-d$ columns are zero. For both actions, the $G$-orbit of $I_{n,d}$ consists of matrices which we will
call unimodular (although it would have been more accurate to call them column-unimodular):

\begin{Definition}\rm
A matrix $A\in Mat_{d\times n}(\dbZ)$ will be called {\it unimodular} if its columns span $\dbZ^d$. We will
denote the set of all $d\times n$-unimodular matrices by $Um_{d\times n}(\dbZ)$.
\end{Definition}

As in the introduction, let $\Row_{n,d}$ (resp. $\Col_{n,d}$) be the subgroup of $\GL_n(\dbZ)$ consisting of matrices whose first $d$ rows (resp. first $d$ columns) coincide with those of the identity matrix. We can reformulate the definition of $\widetilde{\IAR_{n,d}}$ and
$\widetilde{\IAC_{n,d}}$ as follows: 
$$\widetilde{\IAR_{n,d}}=\{g\in G: [g]\in \Row_{n,d}\}\mbox{ and }\widetilde{\IAC_{n,d}}=\{g\in G: [g]\in\Col_{n,d}\}.$$ 
On the other hand, it is straightforward to check that $\widetilde{\IAR_{n,d}}$ (resp. $\widetilde{\IAC_{n,d}}$) is the stabilizer of $I_{n,d}$ with respect to 
\eqref{eq:actionrow} (resp. \eqref{eq:actioncolumn}) -- this explains the names `row action' and `column action'.

Thus, if $H=\widetilde{\IAR_{n,d}}$ (resp. $H=\widetilde{\IAC_{n,d}}$), then $G/H$ is isomorphic as a $G$-set to 
$Z=Um_{d\times n}(\dbZ)$ with the row (resp. column) action above. We will refer to the groups $\widetilde{\IAR_{n,d}}$ (resp. $\widetilde{\IAC_{n,d}}$) as the {\it row-stabilizer} (resp. {\it column-stabilizer}) groups.
\skv

The following two simple observations describe the actions of the Nielsen generators on $Mat_{d\times n}(\dbZ)$
and an explicit isomorphism of $G$-sets between $G/H$ and $Um_{d\times n}(\dbZ)$. Both results follow immediately
from the definitions.

\begin{Lemma}
\label{obs:actions}
Let $A\in Mat_{d\times n}(\dbZ)$.
\begin{itemize}
\item[(a)] If $G$ acts on $Mat_{d\times n}(\dbZ)$ via the row action \eqref{eq:actionrow}, 
then $R_{ij}.A=L_{ij}.A$ is the matrix obtained from $A$
by subtracting the $j^{\rm th}$ column of $A$ from the $i^{\rm th}$ column of $A$.
\item[(b)] If $G$ acts on $Mat_{d\times n}(\dbZ)$ via the column action \eqref{eq:actioncolumn}, then
$R_{ij}.A=L_{ij}.A$ is the matrix obtained from $A$
by adding the $i^{\rm th}$ column of $A$ to the $j^{\rm th}$ column of $A$.
\end{itemize}
\end{Lemma}

\begin{Lemma}
\label{obs:isomorphism}
Recall that $g\mapsto [g]$ is the natural projection from $G$ to $\SL_n(\dbZ)$, and
define $\phi:G\to Um_{d\times n}(\dbZ)$ as follows:
\begin{itemize}
\item[(a)] If $H=\widetilde{\IAR_{n,d}}$, let $\phi(g)$ be the matrix consisting of the first $d$ rows of $[g]^{-1}$. 
\item[(b)] If $H=\widetilde{\IAC_{n,d}}$, let $\phi(g)$ be the matrix consisting of the first $d$ rows of $[g]^T$.
 \end{itemize} 
Then $\phi$ induces an isomorphism of $G$-sets $G/H\to Um_{d\times n}(\dbZ)$.
\end{Lemma}

From now on we will identify $G/H$ with $Um_{d\times n}(\dbZ)$ via the isomorphism from Lemma~\ref{obs:isomorphism},
although an explicit formula for the latter will only be used in the proof of Proposition~\ref{prop:essential2} at the end of the paper. By contrast, the formulas from Lemma~\ref{obs:actions} will be frequently used in computations without further mention.

\skv
In order to proceed with the method outlined in \S~\ref{sec:diagrams}, 
we need to fix a super-Artinian partial order on $Um_{d\times n}(\dbZ)$. 
We will define such an order on $Mat_{d\times n}(\dbZ)$ (but will only use its restriction to $Um_{d\times n}(\dbZ)$).
\skv

For the remainder of this section we will only consider the case $d=1$ (in which case we will write $\dbZ^n$ instead of $Mat_{1\times n}(\dbZ)$). For $d>1$, the proof of Theorem~\ref{thm:finpres} will be more involved, and we will use different orders in different parts of the proof.
The description of those orders for $d>1$ will be postponed until \S~5.   
 
\begin{Definition}\rm Take any element $x=(x_1,\ldots,x_n)$ in $\dbZ^n$, let $M=\|x\|_{\infty}=\max\{ |x_i|\}$ and fix $1\leq i\leq n$.
We will say that the $i^{\rm th}$ coordinate of $x$ is
\begin{itemize}
\item {\it maximal} if $|x_i|=M$;
\item {\it zero} if $x_i=0$;
\item {\it good} otherwise.
\end{itemize}
We will also say that a coordinate is {\it bad} if it is not good.
We will denote the number of maximal, zero and good coordinates of $x$ by $m(x)$, $z(x)$ and $g(x)$, respectively.
\end{Definition}
\skv
Let us now endow $\dbZ_{\geq 0}^4$ with the (total) lexicographical order where we compare two elements by first comparing their first coordinates, then their second coordinates, etc. Define the map $N:\dbZ^n\to \dbZ_{\geq 0}^4$ by
$$N(x)=(\|x\|_{\infty}, m(x), n-g(x),\|x\|_1),$$
and define a partial order on $\dbZ^{n}$ by setting $x<y$ $\iff$ $N(x)<N(y)$ 
(with respect to the lexicographical order on $\dbZ_{\geq 0}^4$).

The obtained order on $\dbZ^n$ is strongly Artinian (in particular, super-Artinian) since $\ell^{\infty}$-balls in $\dbZ^n$ are finite. 
\skv
 
More explicitly, two vectors in $\dbZ^n$ can be compared by successively applying the following criteria (we start by applying the first criterion; if it separates the vectors, we stop; if it is indecisive, we move to the second one, etc.):

\begin{itemize}
\item[(1)] a vector with the larger $\ell^{\infty}$-norm is larger;
\item[(2)] a vector with the larger number of maximal coordinates is larger;
\item[(3)] a vector with the smaller number of good coordinates is larger;
\item[(4)] a vector with the larger $\ell^1$-norm is larger. 
\end{itemize}
If two vectors cannot be separated using the criteria (1)-(4), they are declared incomparable.

\skv
\paragraph{\bf Motivating the order.} Before proceeding, let us provide some motivation for the above order. We start with a general definition.

\begin{Definition}\rm
Let $G$ be a group generated by a finite set $S$, and assume that $G$ acts transitively on a poset $P$.
Given $v\in P$, its {\it return degree} $\deg_S^{-}(v)$ is the number of pairs $(w,s)\in P\times S\cup S^{-1}$ such that
$w<v$ and $w=sv$. In other words, $\deg_S^{-}(v)$ is the number of edges in the Schreier graph $Sch(G,S)$ connecting $v$
to a smaller vertex.
\end{Definition}

Suppose now that $G$ is finitely presented and we are given a $P$-labeled van Kampen diagram $\Omega$ over some finite presentation
$(S,R)$ of $G$. Fix a maximal vertex $v$ of $\Omega$. Intuitively, one would expect that the larger $\deg_S^{-}(v)$ is, the easier it is to construct a full reduction of $\Omega$ at $v$.

The following lemma shows that in the case we are interested in,
all vertices whose labels lie outside of a fixed finite set have sufficiently large return degree:

\begin{Lemma}
\label{lem:ordermotiv}
Let $G=\widetilde{ \SAut(F_n)}$, $\calX$ its optimized Nielsen generating set
 and $P=Um_{1,n}(\dbZ)$ with one of the $G$-actions from \eqref{eq:actionrow}, \eqref{eq:actioncolumn}. 
Let $v\in P$ with $\|v\|_{\infty}\geq 2$. Then $\deg_{\calX}^-(v)\geq \lfloor \frac{n+1}{3}\rfloor$.
\end{Lemma}
\begin{Remark}\rm
Lemma~\ref{lem:ordermotiv} does not seem to have a useful counterpart for some more natural partial orders on $Um_{1,n}(\dbZ)$. For instance, if we order vectors first by $\ell^{\infty}$-norm and then by $\ell^1$-norm, then for any $n$ there will be infinitely
many elements of $Um_{1,n}(\dbZ)$ of return degree $2$. Indeed, let $v=(a,b,0,\ldots,0)$ where $a$ and $b$ are coprime and
$|a|>2|b|$. The neighbors of $v$ in the Schreier graph $Sch(G,\calX)$ are obtained from $v$ either by a signed permutation of coordinates
(in which case we get an element incomparable to $v$) or by adding or subtracting the $i^{\rm th}$ coordinate from the $j^{\rm th}$
coordinate for some $i$ and $j$, in which case we can only get a smaller element when $i=2$ and $j=1$, and the sign is uniquely determined by $a$ and $b$. Thus, $v$ has at most 1 neighbor $w$ with $w<v$, and there are $2$ elements $s\in S\cup S^{-1}$ such that $w=sv$ 
(the possible choices for $s$ are $\{R_{ij}^{\eps},L_{ij}^{\eps}\}$ or $\{R_{ji}^{\eps},L_{ji}^{\eps}\}$ with $\eps=\pm 1$ depending on
the signs of $a$ and $b$ and whether we work with the row or the column action).
 It is not hard to check that extending this partial order to a total order will not resolve the problem.
\end{Remark}

We will not formally apply Lemma~\ref{lem:ordermotiv} in this paper, but its proof should provide a good preview of how single-cell reductions will be constructed in the next section.

\begin{proof}[Proof of Lemma~\ref{lem:ordermotiv}]
We will give a proof for the row $G$-action \eqref{eq:actionrow}; the proof for the column action is analogous.
Also without loss of generality we can assume that all coordinates of $v$ are non-negative.

Recall that $m(v)$, $g(v)$ and $z(v)$ denote the number of maximal, good and zero coordinates of $v$, respectively. Thus,
$m(v)+g(v)+z(v)=n$. For simplicity of notation
let $t=\lfloor \frac{n+1}{3}\rfloor$, so that $n\geq 3t-1$. Hence one of the following 3 inequalities holds:
 $m(v)\geq t+1$, $g(v)\geq t$ or $z(v)\geq t$. We consider the 3 cases accordingly.
\skv
{\it Case 1: $m(v)\geq t+1$}. Without loss of generality we can assume that $v_1,\ldots, v_{t+1}$ (the first $t+1$ coordinates of $v$) are maximal. Then the vectors $R_{12}v, \ldots, R_{1,t+1}v$ are all distinct (since $v_i\neq 0$ for $2\leq i\leq t+1$) and 
all less than $v$ since each of them is obtained from $v$ by replacing a maximal coordinate by $0$. Thus $v$ has at least $t$
smaller neighbors, so in particular $\deg_{\calX}^-(v)\geq t$.
\skv

{\it Case 2: $g(v)\geq t$}. Without loss of generality assume that the first coordinate $v_1$ is maximal and $v_2,\ldots, v_{t+1}$ are all good.
Then the vectors $R_{12}v, \ldots, R_{1,t+1}v$ are all less than $v$ since each of them either has smaller
$\ell^{\infty}$-norm than $v$ (this happens if $v_1$ is the unique maximal coordinate) or has the same $\ell^{\infty}$-norm but fewer
maximal coordinates (if $v$ has more than $1$ maximal coordinate). Unlike Case~1, the vectors $\{R_{1i}v\}_{i=2}^{t+1}$ need not be distinct,
but the pairs $(R_{1i}v,R_{1i})$ are definitely distinct, so we still have $\deg_{\calX}^-(v)\geq t$.

\skv
{\it Case 3: $z(v)\geq t$}. This is the only case where we use the assumptions that $\|v\|_{\infty}\geq 2$ and $v$ is unimodular. Together 
they imply that $v$ has at least one good coordinate.  Without loss of generality assume that the first coordinate $v_1$ is good and $v_2,\ldots, v_{t+1}$ are all zero. Then the vectors $R_{21}v, \ldots, R_{t+1,1}v$ are all distinct and all less than $v$ since each of them is obtained from $v$ by replacing a zero coordinate by a good coordinate, so again $\deg_{\calX}^-(v)\geq t$.
\end{proof}

Let us now establish a simple, but very useful, criterion for a commuting map to be a reduction
with respect to the above partial order.

\begin{Definition}\rm $\empty$
\label{def:support}
\begin{itemize}
\item[(a)] Let $w\in F(\calX)$. Its support $\supp(w)$ is the set of all indices which appear in $w$ (written as a reduced word in
$\calX\sqcup \calX^{-1}$).
\item[(b)] Now let $\Omega$ be a diagram. For every edge $e\in E(\Omega)$ its support $\supp(e)$ is the support of its label $l(e)$.
\item[(c)] Given a subdiagram $\Delta$ of $\Omega$, we define $\supp(\Delta)$ to be the union of the supports of all edges of $\Delta$.
\end{itemize}
\end{Definition}

\begin{Lemma} 
\label{commuting}
Recall that $d=1$ in this part of \S~2. Let $v$ be an interior maximal vertex of a diagram $\Omega$ and $M=\|v\|_{\infty}$. Let $\calF$ be a cell
containing $v$, let $i\neq j$ be distinct indices not contained in $\supp(\calF)$, and 
let $S_{ij}$ be either $\{R_{ij}^{\pm 1}, R_{ji}^{\pm 1}\}$ or $\{L_{ij}^{\pm 1}, L_{ji}^{\pm 1}\}$. 
Then one of the commuting maps 
$C_{\calF}(x)$ with $x\in S_{ij}$ is a reduction with the exception of
the following cases:
\begin{itemize}
\item[(i)] $v_i=v_j=0$;
\item[(ii)] $|v_i|=M$ and $v_j=0$ or vice versa;
\item[(iii)] $v_i$ and $v_j$ are both good and $|v_i|=|v_j|$.
\end{itemize}
\end{Lemma}
\begin{proof} We will give a proof for $S_{ij}=\{R_{ij}^{\pm 1}, R_{ji}^{\pm 1}\}$; the proof in the other case is identical.
Below we will show that there exists $x\in S_{ij}$ such that $xv<v$ only using the above assumptions on $v_i$
and $v_j$ (and not the fact that $v$ is maximal).
Since $i,j\not\in \supp(\calF)$, all vertices of $\calF$ have the same $i^{\rm th}$ and $j^{\rm th}$ coordinates.
Thus, the same argument would imply that $xw<w$ for any vertex $w$ of $\calF$, and therefore $C_{\calF}(x)$ is a reduction.
\skv

As in Lemma~\ref{lem:ordermotiv}, we will give a proof for the row action \eqref{eq:actionrow}.
Suppose that none of (i)-(iii) holds.
Without loss of generality we can assume that $0\leq v_i\leq v_j$. Since (i) does not hold, we must have $v_j>0$. 
\skv
{\it Case 1: $v_j=M$.} Since (ii) does not hold, $v_i>0$ which implies that $R_{ji}v < v$.

\skv
{\it Case 2: $0<v_j<M$.} In this case $v_j$ is good. Since (iii) does not hold, $v_i< v_j$. If $v_i=0$,
then $R_{ij}v$ and $v$ have the same $\ell^{\infty}$-norm and the same number of maximal coordinates, but
$R_{ij}v$ has more good coordinates, so $R_{ij}v<v$. And if $v_i>0$, then $R_{ji}v$ and $v$ have the same 
$\ell^{\infty}$-norm, the same numbers of maximal and good coordinates, but
$R_{ij}v$ has smaller $\ell^1$-norm, so again $R_{ij}v<v$.  
\end{proof}

\paragraph{\bf Maximal, good and bad relative to a diagram.} We finish this section with a technical variation of an earlier definition
which will simplify the terminology later in the paper.

\begin{Definition}\rm 
\label{def:omegagood}
let $\Omega$ be a diagram, let $M=\|\Omega\|=\max\{\|w\|: w\in V(\Omega)\}$, and let $c\in\dbZ$ with $|c|\leq M$.
We will say that $c$ is {\it $\Omega$-maximal (resp. $\Omega$-good)} if $|c|=M$ (resp. $0<|c|<M$) and $c$ is {\it $\Omega$-bad}
if it is not $\Omega$-good.
\end{Definition}

The following observation summarizes the basic relation between the above properties with respect to a diagram and with respect to 
a particular vertex of that diagram. 

\begin{Observation}
\label{obs:relation}
Let $w$ be a vertex of a diagram $\Omega$ and $c$ a coordinate of $w$. The following hold:
\begin{itemize}
\item[(a)] If $c$ is a good coordinate of $w$, then $c$ is $\Omega$-good.
\item[(b)] If $c$ is $\Omega$-maximal, then $c$ is a maximal coordinate of $w$ and $\|w\|=\|\Omega\|$.
\item[(c)] If $\|w\|=\|\Omega\|$, then $c$ is a maximal (resp. good) coordinate of $w$ if and only if c is $\Omega$-maximal 
(resp. $\Omega$-good).
\end{itemize}
\end{Observation}

The following obvious fact turns out to be very important:

\begin{Observation}
\label{obs:badsubgp}
Let $a,b\in\dbZ$, and suppose that $a$, $b$ and $a+b$ are coordinates of some vertices (possibly distinct) of some diagram $\Omega$. Then either all three of them are $\Omega$-bad or at least two of them are $\Omega$-good.
\end{Observation}
\begin{proof} This holds since if $M=\|\Omega\|$ and $c\in\dbZ$ with $|c|\leq M$, then $c$ is $\Omega$-bad of and only if
$M$ divides $c$.
\end{proof}

\section{Finite presentability of $\IAR_{n,1}$}

\subsection{Preliminaries}
In this section we will prove finite presentability of the row stabilizer group $\IAR_{n,1}$ for $n\geq 26$.

We will fix the following notations throughout this section: $G=\widetilde{\SAut(F_n)}$, $H=\widetilde{\IAR_{n,1}}$ and 
$Z=Um_{1\times n}(\dbZ)$ with the row $G$-action \eqref{eq:actionrow}. Recall that $Z\cong G/H$ as a $G$-set
and we identify $Z$ with $G/H$ via the isomorphism from Lemma~\ref{obs:isomorphism}. 
We also let $(\calX, \calR)$ be the optimized Gersten's presentation of $G$ (see \S~2.3).

The majority of this section will be devoted to proving the following result:

\begin{Proposition} 
\label{rowstab1}
Assume that $n\geq 6$. Let $\Omega$ be a $Z$-labeled diagram over $(\calX,\calR)$, let $v\in V(\Omega)$ be an interior maximal vertex and $\calF$
a cell of $\Omega$ containing $v$. Assume in addition that either $\|v\|_{\infty}>1$ or $\|v\|_{\infty}=1$ and $\|v\|_{1}\geq 6$. 
Then $\Omega$ admits a single-cell $\calF$-reduction. 
\end{Proposition}
We will prove Proposition~\ref{rowstab1} by a case-by-case argument depending on the type of $\calF$ (as defined below). Ultimately we will need not
just this result, but some additional information about the reductions constructed in the proof.

We start with some simple preliminary observations.

\skv
\paragraph{\bf Action on diagrams.} Recall that $\Sigma_n$ denotes the signed symmetric group of degree $n$. 
As we explained in \S~2.3, $\Sigma_n$ acts on $F(\calX)$ preserving $\calR$ as a set of geometric relations and hence also acts on $G$.
Given $\sigma\in \Sigma_n$, let $M_{\sigma}\in \GL_n(\dbZ)$ be the corresponding signed permutation matrix.
Recall that for $g\in G$ we denote by $[g]$ its canonical image in $\GL_n(\dbZ)$. It is straightforward to check that
$$[\sigma(g)]=M_{\sigma}[g]M_{\sigma}^{-1}\mbox{ for all } g\in G \mbox{ and }\sigma\in \Sigma_n.$$
There is also a natural action of $\Sigma_n$ on $Z=Um_{1\times n}(\dbZ)$ given by $\sigma(z)=z M_{\sigma}^{-1}$ for all
$z\in Z$ and $\sigma\in \Sigma_n$.

\begin{Observation}
\label{obs:diagramaction}
For all $g\in G$ and $z\in Z$ we have $\sigma(g.z)=\sigma(g).\sigma(z)$.
\end{Observation}
\begin{proof}
Recall that $g.z=z[g]^{-1}$, so $\sigma(g.z)=z[g]^{-1} M_{\sigma}^{-1}$.
On the other hand,
$$\sigma(g).\sigma(z)=\sigma(z)[\sigma(g)]^{-1}=zM_{\sigma}^{-1}(M_{\sigma}[g]M_{\sigma}^{-1})^{-1}=
zM_{\sigma}^{-1}(M_{\sigma}[g]^{-1}M_{\sigma}^{-1})=z[g]^{-1} M_{\sigma}^{-1}.\qedhere$$ 
\end{proof}

Observation~\ref{obs:diagramaction} implies that there is an induced action of $\Sigma_n$ on the set of $Z$-labeled diagrams over $(\calX,\calR)$.

When proving Proposition~\ref{rowstab1} we can replace the given diagram $\Omega$ by $\sigma\Omega$ for any 
$\sigma\in \Sigma_n$. Indeed, suppose that the hypotheses of Proposition~\ref{rowstab1} hold for some diagram $\Omega$ and
its cell $\calF$ and we constructed a single-cell reduction $\phi:(\sigma\Omega,\sigma\calF)\to (\Omega',\Delta')$ 
of the diagram $\sigma\Omega$ at its cell $\sigma\calF$.
We can consider the induced diagram map $\sigma^{-1}\phi:(\Omega,\calF)\to (\sigma^{-1}\Omega',\sigma^{-1}\Delta')$.
The map $\sigma^{-1}\phi$ is also a reduction since our partial order on $Z$ is $\Sigma_n$-invariant (that is, $z<z'$ 
implies $\sigma(z)<\sigma(z')$ for all $\sigma\in \Sigma_n$).

\skv
In addition to the action of $\Sigma_n$, we also have a simple action of $\dbZ/2\dbZ$ on the set of $Z$-labeled diagrams: the non-trivial
element of $\dbZ/2\dbZ$ simply multiplies all the vertex labels by $-1$. Note that this action commutes with the $\Sigma_n$-action.

\skv

Recall that we defined the notion of support for words in $F(\calX)$ and $Z$-labeled diagrams over $(\calX,\calR)$ at the end of \S~2 (see Definition~\ref{def:support}). The following observation records some of its basic properties:

\begin{Observation}
\label{obs:support} Let $\Omega$ be any $Z$-labeled diagram over $(\calX,\calR)$. The following hold:
\begin{itemize}
\item[(i)] $|\supp(r)|\leq 4$ for any $r\in\calR$ and hence $|\supp(\calF)|\leq 4$ for any cell $\calF$ of $\Omega$.
\item[(ii)] If $\Delta$ is a gallery of length $k$ in $\Omega$, then $\supp(\Delta)\leq 2k+2$.
\item[(iii)] Assume that $n\geq 6$, and let $\Sigma_{n}(1)$ be the subgroup of $\Sigma_n$ consisting of permutations which stabilize the set $\{x_1,x_1^{-1}\}$. Then any orbit of the action of $\Sigma_n(1)$ on $\calR$ contains a relator $r$ such that $4\not\in \supp(r)$.
\end{itemize}
\end{Observation}
\begin{proof} (i) follows immediately from the form of the defining relations. To prove (ii), let $\calF_1,\ldots, \calF_k$ be the cells
of $\Delta$ ordered as in the definition of a gallery (Definition~\ref{def:gallery}). Since $\calF_i$ and $\calF_{i+1}$ share at least one edge, we have
$|\supp(\calF_i) \cap \supp(\calF_{i+1})|\geq 2$. Thus, each cell starting with $\calF_2$ contributes at most 2 new indices to $\supp(\Delta)$
and hence $|\supp(\Delta)|\leq 4+2(k-1)=2k+2$.
\skv
(iii) Take any $r_0\in R$ in the given orbit, and assume that $4\in\supp(r_0)$. Since $n\geq 6$ and $|\supp(r_0)|\leq 4$, there exists
$i\in\{2,3,5,6\}$ such that $i\not\in\supp(r_0)$. But then the transposition $(i,4)$ lies in $\Sigma_n(1)$ and
$r=(i,4)r_0$ does not contain $4$ in its support.
\end{proof}

Let us now fix $\Omega$, $v$ and $\calF$ as in Proposition~\ref{rowstab1}, and write $v=(v_1,\ldots, v_n)$. 
Since $|\supp(\calF)|\leq 4$ and $n\geq 6$, we have $n\geq |\supp(\calF)|+2$, whence at least one of the following conditions holds:
\begin{itemize}
\item[(1)] (maximal coordinate case): there exist $i,j\not\in\supp(\calF)$ such that $v_i$ and $v_j$ are both maximal;
\item[(2)] (good coordinate case): there exists $i\not\in\supp(\calF)$ such that $v_i$ is good;
\item[(3)] (zero coordinate case): there exists $i\not\in\supp(\calF)$ such that $v_i$ is zero.
\end{itemize}

In the maximal coordinate case the proof of Proposition~\ref{rowstab1} is completely straightforward -- indeed, in this case
one of the four maps
$C_{\calF}(R_{ij}^{\pm 1})$, $C_{\calF}(R_{ij}^{\pm 1})$ (with $i,j$ as in (1) above) is a reduction by Lemma~\ref{commuting}.

In the good and zero coordinate cases we will use the action of $\Sigma_n\times \dbZ/2\dbZ$ to make some additional assumptions
which will greatly simplify the notations in the proofs. First, let us choose a set of representatives
$\calR_{rep}$ of the orbits of $\Sigma_n(1)$ on $\calR$ (considered as a set of geometric relators) with the property
that $4\not\in \supp(r)$ for all $r\in \calR_{rep}$ (this is possible by Observation~\ref{obs:support}(iii)).

\begin{Lemma}
\label{obs:good} In the good coordinate case, it suffices to prove Proposition~\ref{rowstab1} under the following additional assumptions:
\begin{itemize}
\item[(i)] $v_1$ is maximal and positive;
\item[(ii)] $v_4$ is good and positive;
\item[(iii)] $l(\partial\calF)\in \calR_{rep}$; in particular $4\not\in\supp(l(\partial\calF))$.
\end{itemize}
\end{Lemma}
\begin{proof} Let $r=l(\partial\calF)$. We will show that conditions (i)-(iii) can be achieved after acting by a suitable element of 
$\Sigma_n\times \dbZ/2\dbZ$. To simplify the notations we will change the definition of $r$ throughout the proof as we act by different
elements.
\skv

{\it Step 1:} We first permute the coordinates so that $v_1$ is maximal. 
\skv

{\it Step 2:} Next we act by some $\sigma\in\Sigma_n(1)$ so that 
$l(\partial\calF)\in \calR_{rep}$. The definition of $\Sigma_n(1)$ ensures that $v_1$ will remain maximal (although it may change its sign).
\skv

{\it Step 3:} Next we ensure that $v_4$ is good. By definition of $\calR_{rep}$, after Step~2 we have $4\not\in \supp(r)$.
Since we are in the good coordinate case, we also know that $v_i$ is good for some $i\not\in\supp(r)$, and moreover $i\neq 1$
since $v_1$ is maximal. If $i=4$, we are done (with Step~3). If $i\neq 4$, we act by the transposition $(i,4)$; this makes
$v_4$ good without affecting the conclusions of Step~1 and 2 since $i\neq 1$ and $i,4\not\in\supp(r)$.
\skv
{\it Step 4:} Now we ensure that $v_1>0$.
Since $v_1$ is maximal, it is nonzero. If $v_1>0$, we are done; otherwise we act by the non-trivial element of $\dbZ/2\dbZ$
(multiplying all coordinates of $v$ by $-1$). This operation does not change any relators, so the conclusion of Step~2 still holds,
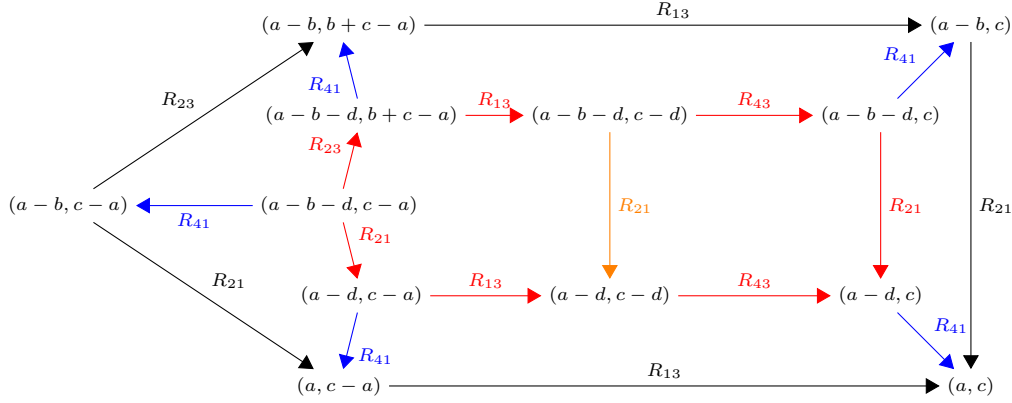
and clearly the conclusions of Steps~1~and~3 are not affected.

\skv
{\it Step 5:} Finally, we ensure that $v_4>0$. Since $v_4$ is good, it is nonzero. If $v_4>0$, we are done; otherwise
we act by the signed permutation $x_4\mapsto x_4^{-1}$. Since $4\not\in\supp(r)$, this does not affect the conclusions of Steps~1-4.
\end{proof}

The following result is an analogue of Lemma~\ref{obs:good} in the zero coordinate case. 

\begin{Lemma}
\label{obs:zero} In the zero coordinate case, it suffices to prove Proposition~\ref{rowstab1} under the following additional assumptions:
\begin{itemize}
\item[(i)] $v_1$ is good and positive;
\item[(ii)] $v_4$ is zero;
\item[(iii)] $l(\partial\calF)\in \calR_{rep}$; in particular $4\not\in\supp(l(\partial\calF))$.
\end{itemize}
\end{Lemma}
\begin{proof} If $\|v\|>1$, then $v$ must have at least one good coordinate, and we can argue as in the proof of Lemma~\ref{obs:good}.
If $\|v\|=1$, then by the hypotheses of Proposition~\ref{rowstab1} there exist at least $2$ indices $i,j\not\in\supp(\calF)$
such that $|v_i|=|v_j|=1$, so $v_i$ and $v_j$ are both maximal, and we are done by the maximal coordinate case.
\end{proof}

Recall that $\Sigma_{n}(1)$ is the stabilizer of $\{x_1,x_1^{-1}\}$ in $\Sigma_n$. 
Let $\calR_1$ be the set of relators in $\calR$ whose support includes index $1$.
The action of $\Sigma_{n}(1)$ on $\calR$ preserves $\calR_1$, and it is straightforward to check that the action of $\Sigma_{n}(1)$ 
on $\calR_1$  has 14 orbits whose representatives are listed in Table~1 below. We will say that two cells in the diagram have the same type if their boundary relators lie in the same orbit. The types will be denoted by integers $1$ through $12$ as well as symbols $6'$ and $7'$.
Single-cell reductions for type $6'$ (resp. type $7'$) will be completely analogous to those for type $6$ (resp. $7$)
and hence will not be discussed explicitly. 

We will also keep track of the orbits under the smaller subgroup $\Sigma_{n}^+(1)$ which consists of signed permutations fixing both $x_1$ and $x_1^{-1}$. Since $[\Sigma_{n}(1):\Sigma_{n}^+(1)]=2$, each $\Sigma_{n}(1)$-orbit on $\calR_1$ either remains a single orbit
or splits into two orbits under $\Sigma_{n}^+(1)$. In the latter case we will distinguish between the left subtype and right subtype denoted by $iL$ and $iR$ where $i$ is the underlying type. 
\skv

Table~1 below is structured as follows. Each row of the table corresponds to a cell type or subtype, which is listed in the first column,
and the second column contains a (chosen) relator for that type/subtype. Columns 3 and 4 contain information about the single-cell reduction map for each type/subtype in the good coordinate case: column 3 lists the possible labels of side edges (this information is only provided for cell types which will appear in double-cell and triple-cell reductions) and column 4 lists the possible types of side cells (we only care about side cells containing $v$, the chosen maximal vertex, but in most cases $v$ will not be specified on the diagram, so all side cells need to be considered). Finally, columns 5 and 6 contain the analogous information for the zero coordinate case.

 \vskip .3cm
\hskip -.5cm
\begin{table}
\label{table:IAR}
\caption{Single-cell maps for the group $\IAR_{n,1}$.}
\begin{tabular}{|c|c|c|c|c|c|}
\hline
&&\multicolumn{2}{c|}{good coordinate case}&\multicolumn{2}{c|}{zero coordinate case}\\[4pt]
\hline
type & relation & side edges & side cells& side edges & side cells\\[2pt]
\hline 
1R&$[R_{23},R_{12}]=R_{13}$&$L_{14}$&$5,6L$&&2, 7 \\[2pt]
\hline 
1L&$[L_{23},L_{12}]=L_{13}$&$R_{14}$&$5,6R$&&2, 7 \\[2pt]
\hline
2R&$[R_{13},R_{21}]=R_{23}$&$L_{14}$&$2L,5$&$R_{41}$ or $R_{43}$&2R, 4, 7 \\[2pt]
\hline
2L&$[R_{13},L_{21}]=R_{23}$&$R_{14}$&$2R,5$&$R_{41}$ or $R_{43}$&2L, 4, 7 \\[2pt]
\hline
3&$[R_{21},R_{32}]=R_{31}$& &$2,6$&&4, 7 \\[2pt]
\hline
4&$[R_{21},R_{31}]=1$& &$2$ &$R_{41}$& 4 \\[2pt]
\hline
5&$[R_{12},L_{13}]=1$&$L_{14}$& $5,1L$&&$2, 6, 6', 9$ \\[2pt]
\hline
6R&$[R_{12},R_{35}]=1$&$L_{14}$&$5,6L$&&2R, 7 \\[2pt]
\hline
6L&$[L_{12},R_{35}]=1$&$R_{14}$&$5,6R$&&2L, 7 \\[2pt]
\hline
$6'$R&$[R_{12},R_{32}]=1$&$L_{14}$&$5,6L$&&2R, 7 \\[2pt]
\hline
$6'$L&$[L_{12},R_{32}]=1$&$R_{14}$&$5,6R$&&2L, 7 \\[2pt]
\hline
7&$[R_{21},R_{35}]=1$&&$2,6$&$R_{41}$&4, 7 \\[2pt]
\hline
$7'$&$[R_{21},L_{23}]=1$&&$2,6$&$R_{41}$&4, 7 \\[2pt]
\hline
8&$L_{12}L_{21}^{-1}R_{12}=w_{12}$&&$1, 2, 5, 7', 11$&& $2, 3, 4, 6', 12$ \\[2pt]
\hline
9&$L_{21}L_{12}^{-1}R_{21}=w_{21}$&&$1, 2, 5, 7', 11$&& $2, 3, 4, 6', 12$ \\[2pt]
\hline
10&$w_{12}^{-1}R_{21}w_{12}=R_{12}^{-1}$&&$1,2,5,7',11$&&$2, 3, 12$ \\[2pt]
\hline
11L&$w_{12}^{-1}R_{13}w_{12}=L_{23}^{-1}$&$L_{14},R_{24}^{-1}$ & $5,6,11R$&&$2, 6, 7, 12$ \\
11L&$w_{12}R_{13}w_{12}^{-1}=R_{23}$&$L_{14},R_{24}^{-1}$ & $5,6,11R$&&$2, 6, 7, 12$ \\[2pt]
\hline
11R&$w_{12}^{-1}L_{13}w_{12}=R_{23}^{-1}$&$R_{14},L_{24}^{-1}$&$5,6,11L$&&$2, 6, 7, 12$ \\
11R&$w_{12}R_{13}w_{12}^{-1}=R_{23}$&$R_{14},L_{24}^{-1}$ & $5,6,11R$&&$2, 6, 7, 12$ \\[2pt]
\hline
12&$w_{12}^{-1}R_{31}w_{12}=R_{32}^{-1}$&&$2,6,7,11$&$R_{41}, R_{42}$&$4, 12$ \\[2pt]
\hline
\end{tabular}
\end{table}
\vskip .4cm
Before turning to the proof of Proposition~\ref{rowstab1} in the good and zero coordinate cases we introduce some additional notations and terminology. Let $I$ be a non-empty subset
of $\{1,\ldots, n\}$. Given $v\in\dbZ^n$, we denote by $v_I\in\dbZ^{|I|}$ the vector obtained from $v$ by removing the $j^{\rm th}$
coordinate for all $j\not\in I$. 

\begin{Definition}\rm $\empty$
\label{def:supported}
\begin{itemize}
\item[(a)] Let $\Omega$ be a diagram. The {\it $I$-trace} of $\Omega$ denoted by $\Omega_I$ is obtained from $\Omega$ by keeping the same vertices, edges and edge labels and replacing $v$ by $v_I$ for every vertex label $v$ in $\Omega$ (note that
$\Omega_I$ is not a $Z$-labeled diagram).

\item[(b)]Let $\phi$ be a diagram map and $\Delta'$ its replacement diagram. Let $I$ be a subset of $\{1,\ldots, n\}$. 
We will say that $\phi$ is {\it supported on $I$} if for every $j\not\in I$ all vertices of $\Delta'$ have the same $j^{\rm th}$ coordinate.
\end{itemize}
\end{Definition}
From the definition of our partial order it is clear that if $\phi$ is supported on $I$, then in order to check whether 
$\phi:(\Omega,\Delta)\to (\Omega',\Delta')$ is a reduction, it suffices to know the $I$-trace of $\Delta'$ along with the original
diagram $\Omega$.

Clearly, $\phi$ is supported on $I$ whenever 
{$\supp(\Delta')\subseteq I$}. However, the latter condition
is not necessary. For example, the reduction for type 2R described below is supported on $I=\{1,2,4\}$ while 
$\supp(\Delta')=\{1,2,3,4\}$ (in fact, $3$ already lies in the support of $\Delta$).

\skv
For each single-cell reduction described in the next two subsections we will start by specifying a set $I$ on which the given map is supported
(in all cases we will have $I\subseteq\{1,2,3,4,5\}$) and then work only with $I$-traces of the vertices of $\Delta$ (rather than full
vertex labels).
\skv

\paragraph{\bf Simplified terminology and notations.} The following additional terminology and notations will be used in the description of single-cell reductions in the next two subsections (and also later in the paper). 
\begin{itemize}
\item The single-cell reduction maps described in \S~3,4 will always be denoted by $\phi$. The initial diagram will be denoted by
$\Omega$, the modified diagram by $\Omega'$, the domain of $\phi$ (which is a single cell) by $\calF$
and the replacement diagram of $\phi$ by $\Delta'$. Thus symbolically $\phi:(\Omega,\calF)\to (\Omega',\Delta')$.

\item The boundary relator of $\calF$ will often be written in the form $r=s$ (recall that formally this denotes the geometric relator
with representative $r^{-1}s$).

\item The vertices of $\calF$ will be called {\it old vertices} and the vertices in $V(\Delta')\setminus V(\calF)$ will be called {\it new vertices}.

\item By abuse of notation, 
for a vertex $w$ we will write $w=(a,b,c,\ldots)$ meaning that $(a,b,c,\ldots)$ is the label of $w$. 

\item A notation like $(a,b,*)$ will mean an element
of $\dbZ^n$ whose first and second coordinates are $a$ and $b$, respectively. 

\vskip .1cm

\item As already mentioned, in each case we will specify a set $I$ on which $\phi$ is supported. Each figure in \S~\ref{sec:singlecellgood},~\ref{sec:singlecellzero},~4.1~and~4.2 will depict the $I$-trace of the
replacement diagram $\Delta'$ for the corresponding single-cell reduction. 

\item Recall that $v$ denotes a chosen maximal vertex of $\calF$. For some types, we will either specify exactly where $v$ appears in the diagram
or at least give some restrictions. In these cases, (the labels of) vertices of $\calF$ which could be equal to $v$ will appear in boldface. 

\item As before, for any $w\in\dbZ^n$ by $w_I$ we will denote its projection to $\dbZ^{|I|}$. 
We will define the partial order on the set of labels of $V(\Delta'_I)$ by setting 
$w_I < w'_I$ if and only if $w<w'$. This order is well defined since by the choice of $I$ the map $w\mapsto w_I$
is a bijection from the labels of $V(\Delta_I)$ to the labels of $V(\Delta'_I)$; however, this order 
may be different from the order induced from $\dbZ^{|I|}$.
\vskip .1cm

\item As before, we will deal with both $\ell^{\infty}$ and
$\ell^1$-norms of vectors in $\dbZ^n$. Since the former will be used more frequently, we will often write $\|v\|$ instead of $\|v\|_{\infty}$.

\item For any diagram $\Psi$ we set $\|\Psi\|=\max\{\|v\|: v\in V(\Psi)\}$.

\item We will refer to the single-cell reductions constructed in \S~3,4 using expressions of the form 
$n$.GA, $n$.GB, $n$.ZA and $n$.ZB where $n$ is an integer between $1$ and $12$. Here $n$ denotes the type
of the boundary relation of $\partial \calF$, the first letter (G or Z) distinguishes between the good coordinate and zero
coordinate cases, and the second letter is A in the case of the row-stabilizer groups (considered in \S~3) and 
B for the column-stabilizer groups (considered in \S~4).

\item When we say that an integer $e$ is maximal (resp. good, bad), we will mean that $e$ is $\Omega$-maximal 
(resp. $\Omega$-good, $\Omega$-bad) unless we view $e$ as a coordinate of a specific vertex (see Definition~\ref{def:omegagood}
at the end of \S~2).

\end{itemize}

\subsection{Single-cell reductions in the good coordinate case}
\label{sec:singlecellgood}
$\empty$
In this subsection we will describe single-cell reductions in the good coordinate case, separately for each cell type. Recall our initial
conventions: $v=(v_1,v_2,\ldots)$ is the chosen maximal vertex of $\calF$ (which is 
interior and maximal for the entire diagram $\Omega$). We set $$M=\|v\|=\|\calF\|=\|\Omega\|.$$ By Observation~\ref{obs:good} we 
can assume that $v_1$ is maximal and positive, $v_4$ is good and positive and $l(\partial\calF)\in\calR_{rep}$. In particular,
$4\not\in\supp(\calF)$, so all vertices of $\calF$ have the same 4th coordinate, and for simplicity of notation we denote its value by $d$.
Thus, $v=(M,v_2,v_3,d,\ldots)$ where $M>d>0$.
\skv
\vskip .2cm 
{\it Easy Case: $1\not\in\supp(l(\partial\calF))$}. Since $g(v)>4$ and $|\supp(l(\partial\calF))|\leq 4$, after acting by some element of $\Sigma_n(1)$, we can assume that $4\not\in\supp(l(\partial\calF))$ and $v_4$ is still good.
Thus, each of the commuting maps $C_{\calF}(R_{14}^{\pm 1})$, $C_{\calF}(R_{41}^{\pm 1})$
is defined, and by Lemma~\ref{commuting} one of them is a reduction.
\vskip .2cm 

From now on we we will assume that $1\in\supp(l(\partial\calF))$, so by choosing suitable $\calR_{rep}$, we can assume that
$l(\partial\calF)$ is one of the relators in Table~1.
For each cell type $i$ where we distinguished subtypes $iR$ and $iL$ we will only describe a reduction for subtype $iR$, as the corresponding reduction for subtype
$iL$ can be obtained simply by swapping $L$ and $R$ in all the edge labels.
\skv

{\it Type 1R: $l(\partial\calF)= ([R_{23},R_{12}]=R_{13})$}. In this case $I=\{1,2\}$.
All edge labels in $\partial\calF$ commute with $L_{14}$. Below we will show that the commuting 
map $C_{\calF}(L_{14})$ is a reduction (see Figure~\ref{1GA})

\input{figure1GA.tex}

\begin{Claim}
\label{claim:nonneg}
Every vertex of $\calF$ has a nonnegative first coordinate. 
\end{Claim}
\begin{proof}
The vertices of
$\calF_I$ are $(a,b)$, $(a+b,b)$, $(a+b,b+c)$, $(a+b+c,b+c)$ and $(a,b+c)$ for some $a,b,c$. Since $v=(M,*)$, by our assumption one of the numbers $a, a+b$ or $a+b+c$ equals $M$. Thus, if some vertex of $\calF$ has a negative first coordinate, then
the difference between the first coordinates of some pair of vertices of $\calF$ is larger than $M$. However, these differences (up to sign)
are $b,c$ and $b+c$, each of which appears as a coordinate of some vertex of $\calF$ and thus cannot exceed $M=\|\calF\|$ in absolute value,
a contradiction.
\end{proof}
Now take any new vertex $w'$ of $\phi$. We need to show that $w'<v$.
There exists an old vertex $w$ (that is, a vertex of $\calF$) such that $w_I=(x,y)$ and $w'_I=(x-d,y)$. Since $x\geq 0$ by Claim~\ref{claim:nonneg} and $x\leq M$ 
(as $x$ is a coordinate of $w$), $w'$ has a non-maximal first coordinate. If either $y=b+c$ or $v_I\neq (a+b+c,b+c)$,
there exists an old vertex $u$ which has a maximal first coordinate
(equal to $M$) and agrees with $w'$ in all other coordinates, so $w'<u$ and hence $w'<v$.

Assume now that $y=b$ and $v_I=(a+b+c,b+c)$, so $a+b+c=M$ and $w'=(a-d,b)$ or
$(a+b-d,b)$. 

{\it Subcase 1: $|b|< M$}. Then the first two coordinates of $w'$ are non-maximal,
so $w'<v$. 

{\it Subcase 2: $b=M$}. Then $a+c=0$ and since both $a+b$ and $b+c$
are $\leq M$ being coordinates of $\calF$, we must have $a=c=0$. Then $v_I=(M,M)$,
so $w'<v$ as $v$ has more maximal coordinates than $w'$.

{\it Subcase 3: $b=-M$}. Then $a+c=2M$, so $a=c=M$ and hence $v_I=(M,0)$ while 
$w'_I=(M-d,-M)$ or $(-d,-M)$, so $w'$ and $v$ have the same number of maximal
coordinates, but $w'$ has more good coordinates, so again $w'<v$. 

Thus we proved that $\phi$ is a reduction. 
\skv

\vskip .2cm

{\it Type 2R: $l(\partial\calF)= ([(R_{13},R_{21}]=R_{23})$}, $I=\{1,2\}$. We will use the map shown in Figure~\ref{2GA}. 

\input{figure2GA.tex}

Let us prove that this map is a reduction. This map produces $7$ new vertices, and we need to check that their $I$-traces are all $<v_I$.
The first coordinate of every vertex of $\calF$ is either $a$ or $a+c$, so by assumption $a=M$ or
$a+c=M$ and hence both $(M,b)$ and $(M,a+b)$ are present among the vertices of $\calF_I$. Since $c$ is also a coordinate of an $\calF$-vertex, $|c|\leq M$, whence $a,a+c\geq 0$. Since $d$ is a good positive coordinate, we have $|a+c-d|<M, |a-d|<M$.
Hence $(a+c-d,b)$ and $(a-d,b)$ are both $< (M,b)$ and $(a+c-d,a+b)$ and $(a-d,a+b)$ are both $< (M,a+b)$, so 
$(a+c-d,b)$, $(a-d,b)$, $(a+c-d,a+b)$ and $(a-d,a+b)$ are all $< v_I$.

Next observe that $b\leq 0$ (for otherwise $a+b>M$ or $a+b+c>M$), whence $|b+d|<M$. Hence both coordinates of the vertices $(a-d,b+d)$ and $(a+c-d,b+d)$ are non-maximal, so they are also $<v_I$. 

The only remaining new vertex is $(a+c-d,a+b+c)$. Its first coordinate is non-maximal. Since $a+c\geq 0$ and $b\leq 0$, the only way $a+b+c$ can be maximal is if
either $b=0$ and $a+c=M$ or  $b=-M$ and $a+c=0$, and in the second case
we must have $a=-c=M$ (since either $a$ or $a+c$ must be maximal).
In the first case we have $(a+c-d,a+b+c)=(M-d,M)<(M,M)=(a+c,a+b+c)$,
and in the second case we have $(a+c-d,a+b+c)=(-d,-M)<(M,-M)=(a,b)$,
so we proved that $(a+c-d,a+b+c)<v_I$.
\vskip .1cm

Let us now formulate a simple general observation which can be used to prove that a new vertex is less than $v$.

\begin{Lemma}
\label{lem:goodeasy}Let $\Omega$ be a diagram and $M=\|\Omega\|$.
The following hold:
\begin{itemize}
\item[(a)] Let $w$ be a vertex of $\Omega$, let $J$ be a subset of $\{1,\ldots, n\}$, and assume that $|w_t|=M$
for some $t\in J$ (so in particular $\|w\|=M$). Let $\widetilde w\in\dbZ^n$ be such that
\begin{itemize}
\item[(i)] $\widetilde w_j=w_j$ for all $j\not\in J$;
\item[(ii)] $|\widetilde w_j|<M$ for all $j\in J$.
\end{itemize}
If $w'\in\dbZ^n$ is obtained from $\widetilde w$ by a signed permutation of coordinates, then $w'<w$.
\skv
\item[(b)] Let $\phi:\Omega\to\Omega'$ be a diagram map and assume that every new vertex $w'$ of $\phi$ 
is obtained from some old vertex $w$ as in (a) (for some choice of $\widetilde w$). Then $\phi$ is a reduction.
\end{itemize}
\end{Lemma}
\begin{proof}(a) Conditions (i) and (ii) imply that either
$\|w'\|<\|w\|$ or $\|w'\|=\|w\|$, but $w'$ has fewer maximal coordinates than $w$, and in either case $w'<w$.

(b) follows directly from (a) and the definition of reduction (Definition~\ref{def:reduction}). 
\end{proof}

Note that our argument for type 2R was based entirely on Lemma~\ref{lem:goodeasy}. Lemma~\ref{lem:goodeasy}
will also be applicable to the majority of the remaining types (in those cases we will simply check that certain coordinates
are non-maximal, skipping the remaining details as they are completely straightforward).
\skv

{\it Type 3R: $l(\partial\calF)= ([R_{21},R_{32}]=R_{31})$}, $I=\{1,2,3\}$. We will use the map $\phi$ in Figure~\ref{3GA}. 

\input{figure3GA.tex}

All the vertices of $\calF$ have the same first coordinate $a$, so by our assumption $a=M$. Since $a+b$, $a+b+c$, $b$
and $b+c$ are all
coordinates of $\calF$, we must have $-M\leq b,b+c\leq 0$.
Since $d$ is good and positive, it follows that $|a-d|,|b+d|,|b+c+d|<M$, so
$\phi$ is a reduction by Lemma~\ref{lem:goodeasy}.

{\it Type 4R: $l(\partial\calF)= [R_{21},R_{31}]$}, $I=\{1,2,3\}$. We will use the map in Figure~\ref{4GA}.
The justification is completely analogous to type 3.

\input{figure4GA.tex}

{\it Type 5: $l(\partial\calF)= [L_{12},R_{13}]$}, $I=\{1,2\}$. Write $v=(a,b,c,d,*)$. Acting by the signed permutations
$x_2\mapsto x_2^{-1}$ and/or $x_3\mapsto x_3^{-1}$ if needed, we can assume that both edges at $v$ are incoming,
so $l(\partial\calF)$ is exactly as in Figure~\ref{5GA}. The map shown in Figure~\ref{5GA} is a reduction -- the proof is again analogous to type 3.

\input{figure5GA.tex}

\skv
{\it Type 6: $l(\partial\calF)= [R_{12},R_{35}]$}, $I=\{1,3\}$. All edge labels in $\partial\calF$ commute with $L_{14}$. The commuting map $C_{\calF}(L_{14})$
is a reduction by the same argument as for type 1.  

{\it Type 7: $l(\partial\calF)= [R_{21},R_{35}]$}, $I=\{1,2,3\}$. This case is analogous to (and easier than) type 4.

{\it Type 9: $l(\partial\calF)= (L_{21}L_{12}^{-1}R_{21}=w_{21})$}, $I=\{1,2\}$. The vertices of $\calF_I$
are $(a,b)$, $(a,a+b)$, $(a+b,-a)$ and $(a+b,b)$, so $a=M$ or $a+b=M$.
We will use one of the 2 maps in Figure~\ref{9GA} -- the first map if $a=M$ and the second map if $a+b=M$ but $a\neq M$.

Since $a=M$ or $a+b=M$ and $a,b,a+b$ are all $\calF$-coordinates, it follows that $0\leq a,a+b\leq M$,
whence $|a+b-d|,|a-d|<M$. 
\skv

{\it Case 1: $a=M$.} In this case the first map in Figure~\ref{9GA} is a reduction by Lemma~\ref{lem:goodeasy}.
\skv

{\it Case 2: $a+b=M$, $0\leq a< M$.} Since $a+b=M$, we must have $b\geq 0$ and hence $|b-d|<M$.
Since we also assume that $|a|< M$, the second map in Figure~\ref{9GA} is a reduction by Lemma~\ref{lem:goodeasy}.

\input{figure9GA.tex}
\skv

{\it Type 8: $l(\partial\calF)= (L_{12}L_{21}^{-1}R_{12}=w_{12})$}, $I=\{1,2\}$. 
The vertices of $\calF_I$ are $(a,b)$, $(a,a+b)$, $(-b,a+b)$ and $(a+b,b)$, so $M$ is equal to $a,a+b$ or $-b$.
\skv 

{\it Case 1: $a+b=M$}. As in type 9, we have $a,a+b\geq 0$ and $|a+b-d|,|a-d|<M$. Hence the map
in Figure~\ref{8GA} is a reduction by Lemma~\ref{lem:goodeasy}.

\skv 

{\it Case 2: $a=M$}. In this case we slightly modify the map in Figure~\ref{8GA}, changing the subpath
$L_{14}L_{24}$ starting at the vertex $(a,a+b)$ by $L_{24}L_{14}$. This changes the new vertex
adjacent to $(a,a+b)$ from $(a,a+b-d)$ to $(a-d,a+b)$. The rest of the argument is analogous to Case~1.
\skv

{\it Case 3: $-b=M$}. In this case $(-b,a+b)$ is a maximal vertex of $\calF_I$. But then
$(a+b,b)$ is also a maximal vertex with a maximal second coordinate. Thus we can swap the roles of the first and second coordinates,
in which case we are reduced to type 9 discussed above.

\input{figure8GA.tex}

\skv

{\it Type 10: $l(\partial\calF)= (w_{12}^{-1}L_{21}w_{12}=R_{12}^{-1})$}, $I=\{1,2\}$. The maps in this case are shown in Figure~\ref{10GA}.

\input{figure10GA.tex}

The vertices of $\calF_I$ are $(a,b)$, $(a+b,b)$, $(b,-a)$ and $(b,-a-b)$. If $v=(b,*)$, that is, $b=M$, we use the first map
in Figure~\ref{10GA}. Since $b=M$ and $a+b$ is an $\calF$-coordinate, $a\leq 0$, whence $|b-d|,|a+d|<M$, and we are done as in type 8
by Lemma~\ref{lem:goodeasy}.

If $v=(*,b)$, we use the second map in Figure~\ref{10GA}. In this case $a,a+b\geq 0$, whence $|a-d|, |a-b-d|<M$, and again we are done
as in the first case.
\skv

{\it Type 11R: $l(\partial\calF)= (w_{12}L_{23}w_{12}^{-1}=R_{13}^{-1})$ or $l(\partial\calF)= (w_{12} R_{13}w_{12}^{-1}=R_{23})$}, $I=\{1,2\}$. 
Note that here we are not specifying the boundary relator in advance and will choose one of these two relators
(both of which lie in the same orbit under $\Sigma(1)^+$) depending on the labels of edges containing $v$.

Initially we only assume that $\calF$ has type 11R. One of the edge labels at $v$ is a Weyl generator of the form $w_{1i}$
and the other is a Nielsen generator. After acting by a suitable element of $\Sigma(1)^+$, we can assume that
$i=2$, $\supp(\calF)=\{1,2,3\}$ and the edge at $v$ labeled by $w_{12}$ is incoming. If $e$ is the other edge
at $v$, then $\supp(e)=\{1,3\}$ or $\{2,3\}$, and we will consider these 2 cases separately:

\begin{itemize}
\item subtype 11.R.1: $\supp(e)=\{1,3\}$. 

\noindent
In this case we can assume that $l(\partial\calF)= (w_{12}L_{23}w_{12}^{-1}=R_{13}^{-1})$.

\item subtype 11.R.2: $\supp(e)=\{2,3\}$. 

\noindent
In this case we can assume that $l(\partial\calF)= (w_{12} R_{13}w_{12}^{-1}=R_{23})$.
\end{itemize}

The reduction maps for both subtypes are shown in  Figure~\ref{11GA}. Note that in subtype 11.R.1 $v_I=(a,b)$ or $(a-c,b)$
and in subtype 11.R.2 $v_I=(a,b)$ or $(a,b-c)$. The justification is similar to previous cases.

\skv
Note that we are forced to use different boundary relator representatives for the subtypes 11.R.1 and 11.R.2 because we require 
that the edge labeled by $w_{12}$ is incoming at $v$. The reason we want to have representatives with the latter property is 
that it will simplify the description of double-cell reductions, which we will need to construct for type 11 later in this section.  

\input{figure11GA.tex}

\noindent

\skv
{\it Type 12: $l(\partial\calF)= (w_{12}^{-1}R_{31}w_{12}=R_{32}^{-1})$}, $I=\{1,2,3\}$. The vertices of $\calF_I$ are $(a,b,c)$, $(b,-a,c)$,
$(b,-a,a+c)$ and $(a,b,a+c)$. The reduction maps are shown in Figure~\ref{12GA}. The justification is very similar to Type 10. 

\input{figure12GA.tex}

\subsection{Single-cell reductions for $\IAR_{n,1}$ in the zero coordinate case}
\label{sec:singlecellzero}
We start with a simple observation which will be applicable to many cell types. Its role in the zero coordinate case
will be similar to that of Lemma~\ref{lem:goodeasy} in the good coordinate case. 

\begin{Lemma}
\label{lem:zeasy}
Let $\phi:\Omega\to\Omega'$ be a diagram map and $v$ a maximal vertex of $\Omega$.
The following hold:
\begin{itemize}
\item[(a)] Let $w$ be a vertex of $\Omega$, and suppose that $w'\in\dbZ^n$ is obtained from $w$ by replacing one or several $\Omega$-bad coordinates by 
$\Omega$-good coordinates. Then $w'<w$. Hence if every new vertex $w'$ of $\phi$ is obtained from some old vertex $w$ in this way, then $\phi$ is a reduction.
\item[(b)] Let $u$ be an old vertex of $\phi$ which agrees with $v$ apart from the $j^{\rm th}$ coordinate for some $j$,
and assume that the $j^{\rm th}$ coordinate of $v$ is good (and hence $\Omega$-good as $v$ is maximal). 
Then the $j^{\rm th}$ coordinate of $u$ is $\Omega$-good.
\end{itemize} 
\end{Lemma}

\begin{proof} (a) Let $M=\|\Omega\|$. By assumption $\|w\|\leq M$, and for any $i$ such that $w'_i\neq w_i$ we have
then $|w'_i|<M$, so either $\|w'\|< M=\|v\|$ (which automatically implies that $w'<v$) or $\|w'\|=\|w\|=M$. But in the latter case
for every $i$ such that $w'_i\neq w_i$, the $i^{\rm th}$ coordinate of $w$ is bad while the $i^{\rm th}$ coordinate of $w'$ is good,
so by definition $w'<w$.

(b) If the $j^{\rm th}$ coordinate of $u$ is $\Omega$-bad, we can apply (a) with $w=u$ and $w'=v$ to conclude that $v<u$, a contradiction.
\end{proof}

For the rest of this subsection we fix $\Omega$, $\calF$ and $v$ as in Proposition~\ref{rowstab1}, and we assume that we are in the zero
coordinate case. By Observation~\ref{obs:zero} we can (and will) assume that  $v_1$ (the first coordinate of $v$) is good and positive 
and the fourth coordinate of $v$ is $0$. Thus, $v=(v_1,v_2,v_3,0,*)$ where $0<v_1<M=\|v\|$.
\skv

{\it Type 1R: $l(\partial\calF)= ([R_{23},R_{12}]=R_{13})$}, $I=\{1,2,4\}$.  
We consider two cases. If the third coordinate $c$ (which is the same for all vertices of $\calF$) is good, we can simply use the commuting reduction $C_{\calF}(R_{43})$. So let us assume that $c$ is bad and use the map shown in Figure~\ref{1ZA}.

\input{figure1ZA.tex}

Each new vertex coincides with one of the vertices of $\calF$ in all coordinates except the fourth one, and the fourth coordinate changes from $0$ to $a$, $a+b$ or $a+b+c$, so by Lemma~\ref{lem:zeasy}(a) it suffices to show that $a,a+b$ and $a+b+c$ are all 
good. Since $c$ is bad, $a+b$ and $a+b+c$ are both good or both bad by Observation~\ref{obs:badsubgp}. Also 
since $v$ has a good first coordinate, at least one of the integers $a,a+b$ and $a+b+c$ must be good. So we just need to rule out two possibilities:
\begin{itemize}
\item[(i)] $a$ is good and $a+b, a+b+c$ are bad; 
\item[(ii)] $a$ is bad and $a+b, a+b+c$ are good.
\end{itemize}
If (i) occurs, $v_I=(a,b,0)$ or $(a,b+c,0)$. 
But this is impossible since $(a+b,b,0)$ and $(a+b,b+c,0)$ are also vertices of $\calF_I$ and we have 
$(a,b,0)<(a+b,b,0)$ and $(a,b+c,0)<(a+b,b+c,0)$ since $a$ is good and $a+b$ is bad.

Suppose now that (ii) occurs. If $v_I=(a+b,b,0)$ or $(a+b,b+c,0)$, we get a contradiction as in case (i). The only remaining possibility is that  $v_I=(a+b+c,b+c,0)$. Since $a+b+c$ is good while $a$ is bad, $b+c$ must be good. But then $(a+b+c,b+c,0)<(a,b,0)$, a contradiction.

\vskip .2cm

{\it Type 2R: $l(\partial\calF)= ([R_{13},R_{21}]=R_{23})$}, $I=\{1,2,4\}$. The argument is similar to, 
but easier than type 1.
If $c$ is good, the commuting map $C_{\calF}(R_{43})$ is a reduction, so  assume that $c$ is bad. In this case we use the map shown in 
Figure~\ref{2ZA}. As with type~1, all we need to show is that $a$ and $a+c$ (the only possible values of the fourth coordinates of new vertices)
are good. Since $c$ is bad, they are both good or both bad. And they cannot be both bad since then none of the vertices of $\calF$ has a good
first coordinate. 

\input{figure2ZA.tex}

\vskip .2cm

{\it Types 3R, 4R and 7R: $l(\partial\calF)= ([R_{21},R_{32}]=R_{31}), [R_{21},R_{31}]$ or $[R_{21},R_{35}]$}, $I=\{2,3,4\}$. 
In these cases all the edge labels of $\calF$ commute with $R_{41}$, and all the vertices of $\calF$ have $a$ as the first coordinate (which is good by assumption), so the commuting map $C_{\calF}(R_{41})$ is a reduction.

\vskip .2cm

{\it Type 5: $l(\partial\calF)= [L_{12},R_{13}]$}, $I=\{1,2\}$. The map $\phi$ used for this type is quite different from the ones we have
seen so far, as index $4$ does not appear not only in the support of $\phi$, but even in the support of its replacement diagram.
Write $v=(a,b,c,0,*)$.  As in type 5.GA, we can assume that both edges at $v$ are incoming, in which case
the other three vertices of $\calF$ are $(a+b,b,c,0,*)$, $(a+b+c,b,c,0,*)$ and $(a+c,b,c,0,*)$. 

If either $b$ (resp. $c$) is good, the commuting map $C_{\calF}(R_{42})$ (resp. $C_{\calF}(R_{43})$) is a reduction, so assume that $b$ and $c$ are both bad. 
Since $a$ is good (being the first coordinate of $v$), the numbers $a+b,a+c$ and $a+b+c$ are all good as well by Observation~\ref{obs:badsubgp}. We shall use the map in Figure~\ref{5ZA}. Each of the new vertices is obtained from one of the vertices of $\calF$ by replacing the bad second coordinate $b$ by a good value ($-a$, $a+b$, $-(a+c)$ or $a+b+c$), so this map is a reduction by Lemma~\ref{lem:zeasy}(a).

\input{figure5ZA.tex}

\vskip .2cm

{\it Type 6: $l(\partial\calF)= [R_{12},R_{35}]$}, $I=\{1,3,4\}$.  We will use the map shown in Figure~\ref{6ZA}. The first coordinates of the vertices of $\calF$ are equal to $a$ and $a+b$, and both must be good by Lemma~\ref{lem:zeasy}(b).
Each new vertex is obtained from a vertex of $\calF$ by replacing $0$ by $-a$ or $-(a+b)$ in the fourth coordinate, and hence the map is a reduction by Lemma~\ref{lem:zeasy}(a). 

\input{figure6ZA.tex}

\vskip .1cm
{\it Type 8: $l(\partial\calF)= (L_{12}L_{21}^{-1}R_{12}=w_{12})$}, $I=\{1,2,4\}$. The vertices of $\calF_I$ are $(a,b,0)$, $(a,a+b,0)$,
$(-b,a+b,0)$ and $(a+b,b,0)$. We will use one of the two maps in Figure~\ref{8ZA}. 
By Lemma~\ref{lem:zeasy}(a), the first map is a reduction whenever $a$ and $a+b$
are both good, and the second map is a reduction whenever $a$ and $b$ are both good. Let us prove that one of these two cases
must occur. Indeed, as before, at least two of the numbers $a,b,a+b$ are good, so we only need to rule out the possibility that
$b$ and $a+b$ are good and $a$ is bad. If the latter case occurs, $v_I=(-b,a+b,0)$ or $(a+b,b,0)$ (since the first coordinate of $v$
must be good); on the other hand, both of these vertices are $<(a,b,0)$ since $a$ is bad, a contradiction.

\input{figure8ZA.tex}

\vskip .1cm
{\it Type 9: $l(\partial\calF)= (L_{21}L_{12}^{-1}R_{21}=w_{21})$}, $I=\{1,2,4\}$. The vertices of $\calF_I$ are $(a,b,0)$, $(a+b,b,0)$,
$(a+b,-a,0)$ and $(a,a+b,0)$. If $a$ and $a+b$ are both good, the map in Figure~\ref{9ZA} is a reduction.
Suppose now that $a$ or $a+b$ is bad. By our assumptions $\calF_I$ must have a maximal vertex with a good first coordinate and bad second coordinate. Thus $(a+b,-a,0)$ or $(a,a+b,0)$ is maximal, and since these two vertices differ by a signed permutation of coordinates, they are both maximal. Therefore $\calF_I$ (and hence $\calF$) has a maximal vertex with a good second coordinate, so we can switch the roles of the first and second coordinates, thereby reducing to Type~8.

\input{figure9ZA.tex} 

\vskip .1cm
{\it Types 10-12}: The reduction maps for these types are similar to the ones we used in the good coordinate case. The map
for type 12 is shown on Figure~\ref{12ZA} -- we emphasize the map for this type since it will be used for construction
of double-cell reductions.

\input{figure12ZA.tex} 
 
\subsection{Conclusion of the proof}

We start by constructing double-cell reductions for certain types of galleries of length $2$. As earlier in this section,
we assume that $v$ is an interior maximal vertex of some diagram $\Omega$ and the first coordinate of $v$
is positive (so that we can refer to cell types as before). 
\skv

Recall that $m(v),g(v)$ and $z(v)$ denote the number of maximal, good and zero coordinates of $v$, respectively.

\begin{Lemma}
\label{doublecell}
Let $\Delta$ be a gallery of length $2$ at $v$ with cells $\calF_1$ and $\calF_2$, and let $e\in E(\calF_1)\cap E(\calF_2)$ be an edge 
containing $v$. In each of the following cases $\Omega$ admits a $\Delta$-reduction which eliminates $e$:
\begin{itemize}
\item[(1)] $g(v)\geq 7$,  $\calF_1$ and $\calF_2$ both have type 2R or both have type 2L, and $e$ has label $R_{j1}$ or $L_{j1}$ for some $j$;
\item[(2)] $g(v)\geq 7$,  $\calF_1$ and $\calF_2$ both have type 5, and $e$ has label $R_{1j}$ for some $j$;
\item[(3)] $g(v)\geq 7$,  $\calF_1$ has type 5, $\calF_2$ has type 6R and $e$ has label $R_{1j}$ for some $j$;
\item[(4)] $g(v)\geq  7$, $\calF_1$ and $\calF_2$ both have type 11L and $e$ has label $w_{1j}$ for some $j$;
\item[(5)] $z(v)\geq 7$, $\|v\|_{\infty}>1$, $\calF_1$ and $\calF_2$ both have type 2, and $e$ has label $R_{1j}$ for some $j$;
\item[(6)] $z(v)\geq  7$, $\|v\|_{\infty}>1$, $\calF_1$ and $\calF_2$ both have type 12, and $e$ has label $w_{1j}$ for some $j$.
\end{itemize}  
\end{Lemma}
\begin{proof}By Proposition~\ref{rowstab1}, in each case there exists an $\calF_i$-reduction $\phi_i$ for $i=1,2$, and we only need
to check that $\phi_1$ and $\phi_2$ can be made compatible at $e$. Below we will give a detailed argument for case (1);
the other 5 cases can be handled similarly.

\skv
First, without loss of generality we can assume that $\calF_1$ and $\calF_2$ both have type 2R. Since $g(v)\geq 7$, we have $|\supp(\Delta)|\leq 6<g(v)$. Arguing as in the proof of Lemma~\ref{obs:good}, after acting by a suitable element of $\Sigma_n^+(1)$ (which does not change the cell types), we can assume that $v_4$ is good and positive, $4\not\in \supp(\Delta)$ (so in particular, $j\neq 4$)
and the label of $e$ is $R_{21}$. 

We can now construct single-cell reductions $\phi_i$ at $\calF_i$ for $i=1,2$, as in \S~3.1, such that in both cases both side edges at $e$ are labeled by $L_{14}$ and point away from $e$. Note that the replacement diagrams for
$\phi_1$ and $\phi_2$ may not be identical to the one in Figure~\ref{2GA} -- for instance, we only know that
$\supp(\calF_i)=\{1,2,k_i\}$ for some $k_i\not\in \{1,2,4\}$, and it is possible that $k_1\neq k_2$. However, an easy verification shows that 
if $\calG_i$ is the side-cell at $e$ for $\phi_i$, then $\supp(\calG_i)=\{1,2,4\}$ and
the boundary label of $\calG_i$ is exactly as in Figure~\ref{2GA}. Hence the boundary labels of $\calG_1$ and $\calG_2$
are mirror images of each other, and therefore the single-cell reductions $\phi_1$ and $\phi_2$ are compatible 
at $e$, as desired. 
\end{proof}

We are now ready to prove that the row-stabilizer group $\IAR_{n,1}$ is finitely presented for $n\geq 26$. 
First we claim that it suffices to prove the following proposition.

\begin{Proposition} 
\label{prop:removable}
Let $n\geq 26$, let $v$ be an interior maximal vertex of some diagram $\Omega$, and assume that either
$\|v\|_{\infty}>1$ or $\|v\|_{1}> 9$. Then $v$ is removable, that is, $\Omega$ admits a full reduction which
eliminates $v$.
\end{Proposition}

Indeed, there are only finitely many elements of $\dbZ^n$ of $\ell^{\infty}$-norm $1$. Therefore, combining Proposition~\ref{prop:removable}
and Lemma~\ref{obs:Renz} applied with $B=\{w\in\dbZ^n: \|w\|=1\}$, we deduce that $\IAR_{n,1}$ is finitely presented. 

The reason we are not restricting ourselves to vertices $v$ with $\|v\|_{\infty}>1$ in Proposition~\ref{prop:removable} and also allow some vertices
of $\ell^{\infty}$-norm $1$ is that we will need Proposition~\ref{prop:removable} in this form to prove finite presentability of $\IAR_{n,d}$ for $d>1$.

\begin{proof}[Proof of Proposition~\ref{prop:removable}]
Since $n> 25$, we must have $m(v)> 9$ or $g(v)>8$ or $z(v)>8$. We consider the 3 cases accordingly.

\skv
{\it Case 1:  $m(v)>9$}. Let $\Delta$ be any gallery at $v$ of length $\leq 3$. By Observation~\ref{obs:support}(ii)
we have $|\supp(\Delta)|\leq 8$. Since $m(v)>9$, there exist distinct $i,j\not\in\supp(\Delta)$
such that $v_i$ and $v_j$ are both maximal. Hence by Lemma~\ref{commuting} one of the commuting maps
$C_{\Delta}(x)$ with $x\in\{R_{ij}^{\pm 1},R_{ji}^{\pm 1}\}$ is a $\Delta$-reduction. Any such reduction decreases the degree of $v$, so after finitely many steps $v$ will be eliminated.
\skv
Recall that by our assumptions either $\|v\|_{\infty}>1$ or $\|v\|_{1}> 9$. If $\|v\|_{\infty}=1$, then every nonzero coordinate of $v$ is maximal and $m(v)=\|v\|_{1}> 9$, so Case~1 occurs. Thus from now on we can assume that $\|v\|_{\infty}>1$. This observation is not essential for Case~2, but will be used in Case~3.

\skv
{\it Case 2:  $g(v)>8$}. The proof in this case is considerably more involved and will be divided into 3 steps.
For brevity, we will say that a cell $\calF$ containing $v$ is {\it essential} if $1\in\supp(\calF)$. 
\skv

{\it Step 1:} We start by performing a single-cell reduction in the good coordinate case 
at every essential cell $\calF$ containing to $v$ (recall that such reductions have been described in \S~3.1). As one can see from Table~1 (see \S~3.1), some relation types never arise as side-cells, and in the new diagram all essential cells containing $v$ will have type $1,2,5,6$ or $11$. Moreover, after another round of single-cell reductions at cells of type $1R,2R,6R$ and $11R$, we can assume that the only remaining essential cell types are $1L,2L,5,6L$ and $11L$. 
Furthermore, it is easy to check that every cell $\calF$ containing $v$ which arises after the reductions described in the table has the following property (P):
\begin{itemize}
\item[(P)] If $e$ is any edge of $\calF$ such that $1\in\supp(e)$, then $\calF$ has an edge $e'$ (possibly equal to $e$) which contains
$v$ and such that $l(e')=l(e)$.
\end{itemize}

{\it Step 2:} We will now explain how to eliminate the remaining cells containing $v$. In the discussion below by a cell we will always mean a cell containing $v$. All the subsequent steps will only involve double-cell and triple-cell reductions, so the total number of cells will never go up for the rest of the reduction process, and moreover the process is guaranteed to terminate after sufficiently many triple-cell reductions.

Recall that after Step~1 we are left with cells of types $1L,2L,5,6L$ and $11L$ as well as non-essential cells. First observe that $w_{1i}$ only appears as an edge label in type $11L$, and by property (P) any cell of type $11L$ must have an edge labeled $w_{1i}$ and 
containing $v$. Thus, if there exists at least one cell of type $11L$, there must be two such cells $\calF_1,\calF_2$ sharing an edge $e$ labeled $w_{1i}$ and containing $v$. By Lemma~\ref{doublecell}(4) there exists a double-cell reduction which removes $e$. 
The side-cells that get canceled during this double-cell reduction have type $11L$, and the non-canceled side-cells have type $5$. Therefore, applying this operation several times we can eliminate all cells of type $11L$.

Next we use a similar argument to eliminate all cells of type $2L$ and then all cells of type $5$. In the case of type $2L$ we use the fact that these are the only remaining cells with edge labels $R_{i1}$ or $L_{i1}$, so we can apply Lemma~\ref{doublecell}(1). Once cells of type $2L$ are eliminated, cells of type $5$ are the only remaining cells with labels $R_{1i}$ and we can apply
Lemma~\ref{doublecell}(2).

\skv
{\it Step 3: }
We now arrive at a diagram where all essential cells at $v$ have type $1L$ or $6L$. For any such cell $\calF$, all edge labels of $\calF$ will commute
with $R_{1i}$ with $i\not\in \supp(\calF)$, so as in case~1, we can apply a triple-cell reduction to any gallery of length $3$ (or a double-cell reduction if $\deg(v)=2$). 

Note that after this reduction two cells of types $5$ and/or $6R$, sharing an edge $R_{1i}$, will appear; however, this is not a problem. Indeed, 
By Lemma~\ref{doublecell}(2) we can use a double-cell reduction to eliminate these two cells (and hence also the above edge $R_{1i}$), thereby producing a diagram where all cells again have type $1L$ or $6L$, but the total number of cells is 1 fewer than at the start of Step~3. Therefore, by repeating Step~3 sufficiently many times, we can eliminate all cells at $v$ (and hence $v$ itself).
This completes the proof in Case~2.

\skv {\it Case 3:  $z(v)>8$}. Recall that we are also assuming that $\|v\|_{\infty}>1$, so Lemma~\ref{doublecell}(5)(6) is applicable. The overall reduction procedure in this case is similar to (but substantially easier than) Case~2.

First, after two rounds of single-cell reductions we will be left with essential cells of type $2,4,7,12$. The only remaining cells containing edges with labels $w_{1i}$ or $w_{i1}$ are those of type $12$, and these can be eliminated using double-cell reductions
by Lemma~\ref{doublecell}(5). Similarly by Lemma~\ref{doublecell}(6) we can eliminate all cells of type $2$ as they are the only remaining  ones with edges labeled $R_{1i}$ or $L_{1i}$. 

At this point we are left with only cells of type $4$ and $7$ and non-essential cells. Note that in the zero coordinate case, single-cell reductions at cells of type $4$ and $7$ may only produce side cells of type $4$ and $7$. Moreover,
all edges in a cell of type $4$ or $7$ commute with $R_{i1}$ for any $i$. Therefore, using commuting triple-cell reductions of the form $C(R_{i1})$(for a suitable index $i$), we can eliminate all the remaining cells at $v$.
\end{proof}

\section{The column-stabilizer subgroup for $d=1$}

In this section we prove finite presentability of the column-stabilizer group $\IAC_{n,1}$ for $n\geq 26$. The proof is quite similar to the case of $\IAR_{n,1}$, and we will omit parts of the argument which are identical or analogous to the case of $\IAR_{n,1}$. Lemmas~\ref{lem:goodeasy}~and~\ref{lem:zeasy} remain valid for $\IAC_{n,1}$ (the proof remains the same).

The single-cell reductions for $\IAC_{n,1}$ will have almost the same edge labelings as those for $\IAR_{n,1}$, although  
the roles of the good coordinate and zero coordinate cases will be swapped, that is, the edge labelings used
in the good coordinate case for $\IAR_{n,1}$ will often work in the zero coordinate case for $\IAC_{n,1}$ and vice versa.

However, the vertex labels in the case of $\IAC_{n,1}$ will be very different (since $\IAC_{n,1}$ and $\IAR_{n,1}$ act 
differently on $\dbZ^n$), and therefore we will usually need a new argument to prove that a given single-cell map is a reduction. 

Once the single-cell reductions have been constructed, the remainder of the proof (elimination of an interior maximal vertex of sufficiently large norm) is mostly similar to the case of $\IAR_{n,1}$, but substantial extra work in the good coordinate case will be needed. It will be described in \S~\ref{sec:4GBonedim}.

As with $\IAR_{n,1}$, we start by presenting the summary table of single-cell reductions, followed by their individual descriptions. 

$\empty$
 \vskip .3cm
\begin{table}
\label{table:IAC}
\caption{Single-cell maps for the group $\IAC_{n,1}$.}

\begin{tabular}{|c|c|c|c|c|c|}
\hline
&&\multicolumn{2}{c|}{good coordinate case}&\multicolumn{2}{c|}{zero coordinate case}\\[4pt]
\hline
type & relation & side edges & side cells& side edges & side cells\\
\hline 
1R&$[R_{23},R_{12}]=R_{13}$&&$2,7$&$L_{14}$&$5, 6L$ \\[2pt]
\hline 
1L&$[L_{23},L_{12}]=L_{13}$&&$2,7$&$R_{14}$&$5, 6R$ \\[2pt]
\hline
2R&$[R_{13},R_{21}]=R_{23}$&$R_{41}$ or $L_{41}$&$2,4,7$&$L_{14}$&$2L, 5, 6L$ \\[2pt]
\hline
2L&$[L_{13},L_{21}]=L_{23}$&$R_{41}$ or $L_{41}$&$2,4,7$&$R_{14}$&$2R, 5, 6R$ \\[2pt]
\hline
3&$[R_{21},R_{32}]=R_{31}$&&$4,7$&&2, 6 \\[2pt]
\hline
4&$[R_{21},R_{31}]=1$&$L_{41}, L_{42}$&$3, 4$&& 2 \\[2pt]
\hline
5&$[R_{12},L_{13}]=1$&& $2,9$&$L_{21},R_{21}$& $1L, 5, 6R$ \\[2pt]
\hline
6R&$[R_{12},R_{35}]=1$&&$2,7$&$L_{14}$& $5, 6L$ \\[2pt]
\hline
6L&$[L_{12},R_{35}]=1$&&$2,7$&$R_{14}$& $5, 6R$ \\[2pt]
\hline
$6'$R&$[R_{12},R_{32}]=1$&&$2,7$&$L_{14}$& $5, 6L$ \\[2pt]
\hline
$6'$L&$[L_{12},R_{32}]=1$&&$2,7$&$R_{14}$& $5, 6R$ \\[2pt]
\hline
7&$[R_{21},R_{35}]=1$&$R_{41}$ or $L_{41}$&$4,7$&&$2, 6$ \\[2pt]
\hline
$7'$&$[R_{21},L_{23}]=1$&$R_{41}$ or $L_{41}$&$4,7$&&$2, 6$ \\[2pt]
\hline
8&$L_{12}L_{21}^{-1}R_{12}=w_{12}$&&$2, 4, 12$&& $1, 2, 5, 6', 11$ \\[2pt]
\hline
9&$L_{21}L_{12}^{-1}R_{21}=w_{21}$&&$2, 3, 4, 6', 12$&&$1, 2, 5, 6', 11$ \\[2pt]
\hline
10&$w_{12}^{-1}R_{21}w_{12}=R_{12}^{-1}$&&$2, 3, 12$&& $1,2,11$ \\[2pt]
\hline
11L&$w_{12}^{-1}R_{13}w_{12}=L_{23}^{-1}$&&$2, 6, 7, 12$&$L_{14},R_{24}^{-1}$& $5, 6, 11R$ \\[2pt]
\hline
11R&$w_{12}^{-1}L_{13}w_{12}=R_{23}^{-1}$&&$2, 6, 7, 12$&$R_{14},L_{24}^{-1}$& $5, 6, 11L$ \\[2pt]
\hline
12&$w_{12}^{-1}R_{31}w_{12}=R_{32}^{-1}$&$R_{41},R_{42}$&$4, 7, 12$&& $1, 2, 5, 6', 11$ \\[2pt]
\hline
\end{tabular}
\end{table}
\vskip .4cm

\subsection{Single-cell reductions for $\IAC_{n,1}$ in the good coordinate case}
$\empty$
\vskip .2cm
We will use the same notations and make the same basic assumptions as in the good coordinate case for $\IAR_{n,1}$. 
Thus, $v\in V(\calF)$ is a chosen interior maximal vertex, $M=\|\calF\|=\|v\|$, and we assume that $v=(M,v_2,v_3,d,*)$ where $M>d>0$,
that is, $d$ is good and positive. 

\vskip .2cm

{\it Type 1R: $l(\partial\calF)= ([R_{23},R_{12}]=R_{13})$}, $I=\{1,2,3\}$.  We will use the map shown in Figure~\ref{1GB}.

\input{figure1GB.tex}

The first coordinate of every vertex of $\calF$ is $a$, so we must have $a=M$. By Lemma~\ref{lem:goodeasy}(b), to prove that the map is a reduction it suffices to show that $|M-d|<M$, $|M-b-d|<M$ and $|M+c-b-d|<M$. The latter follows from the facts that $M>d>0$ and $b,b-M,c-b$ and $M+c-b$ are all $\calF$-coordinates and hence cannot exceed $M$ in absolute value.

\vskip .2cm

{\it Type 2R: $l(\partial\calF)= ([R_{13},R_{21}]=R_{23})$}, $I=\{1,3\}$. We will use the map $\phi$ in Figure~\ref{2GB}.
Thus, $M=a$ or $M=a-b$. In both cases we deduce that $|a-d|,|c-d|,|a-b-d|<M$ similarly to type~1. If $M=a-b$, these
inequalities are sufficient to prove that $\phi$ is a reduction using Lemma~\ref{lem:goodeasy}(a). Suppose now that $M=a$ and $a-b\neq M$ (whence $a-b$ is non-maximal). Lemma~\ref{lem:goodeasy}(a) is still applicable to the majority of new vertices of $\phi$, and the only possible exception is the vertex $w$ with $w_I=(a-b-c,b-d)$ in the case
where $a-b-c$ is maximal.  But the latter is only possible if $a=b=c=M$, in which case $w_I=(-M,b-d)<(a,c)$,
so $w<v$ as desired.

\input{figure2GB.tex}

\vskip .2cm

{\it Types 3R, 6R, 7: } 
 All the edge labels of $\calF$ commute with $R_{41}$. In addition, all vertices of $\calF$ have non-negative first coordinate
(this can be proved exactly as for type 1.GA). Therefore, the commuting map $C_{\calF}(R_{41}^{-1})$ is a reduction.
\vskip .2cm

\vskip .2cm

{\it Type 4: $l(\partial\calF)= [R_{21},R_{31}]$}, $I=\{1,2\}$. Unlike type 3R, we cannot use $C_{\calF}(R_{41}^{-1})$ since the vertex of $\calF$ opposite $v$ may have negative first coordinate. Instead
we will use the map $\phi$ in Figure~\ref{4GB}.

\input{figure4GB.tex}

Write $v=(a,b,c,d,*)$, so $a=M$. As in types 5.GA and 5.ZA we can assume that both edges at $v$ are incoming, so the other vertices of $\calF_I$
are as in Figure~\ref{4GB}. Since $a-b,a-c,b,c$ are all $\calF$-coordinates, we have
$0\leq a-b, a-c,b,c\leq M$ whence $|a-d|,|b-d|,|a-c-d|<M$. Lemma~\ref{lem:goodeasy}(a) is applicable to the majority
of new vertices of $\phi$ except possibly $(a-b,b-d,*)$ in the case $a-b=M$ and $(a-b-c,b-d,*)$ in the case
$|a-b-c|=M$. We claim in both cases these vetrices are $<(a,b,*)$, which would finish the proof.

Recall that $a=M$. Hence if $a-b=M$, we have $b=0$, whence $(a-b,b-d)=(M,-d)<(M,0)=(a,b)$.
Also $b,c\geq 0$, so $|a-b-c|=M$ forces $b=c=0$ or $b=c=M$. In the former case we have
$(a-b-c,b-d)=(M,-d)<(M,0)=(a,b)$, and in the latter case $(a-b-c,b-d)=(-M,M-d)<(M,M)=(a,b)$, as desired.
\skv

{\it Type 5: $l(\partial\calF)= [L_{12},R_{13}]$}, $I=\{1,2,3\}$. The map $\phi$ in this case is shown in Figure~\ref{5GB}.
Note that, similarly to type 5.ZA, index $4$ does not lie in the support of the replacement diagram of $\phi$.

\input{figure5GB.tex}

As with type~4, we can assume that $v=(a,b,c,*)$ and both edges at $v$ are incoming. Since $v$ is maximal and $a=M>0$, we must have $b,c>0$. This ensures that $|a-b|<M$ and either $|b+c-a|<M$ or $a=b=c=M$. If $|b+c-a|<M$, all the new vertices are smaller than $v$ by Lemma~\ref{lem:goodeasy}. And if $a=b=c=M$, then $v_I=(a,b,c)$ has $3$ maximal coordinates, while $w_I$ has at most 2 maximal coordinates for
any new vertex $w$. In both cases $\phi$ is a reduction.

\vskip .2cm
{\it Type 8: $l(\partial\calF)= (L_{12}L_{21}^{-1}R_{12}=w_{12})$}, $I=\{1,2\}$. We will use the map shown in Figure~\ref{8GB}.
The vertices of $\calF_I$ are
$(a,b), (a-b,b),(a-b,a),(a,b-a)$. Thus $a=M$ or $a-b=M$. One can prove that the map in Figure~\ref{8GB} is a reduction
similarly to type 9GA.

\input{figure8GB.tex}

\vskip .2cm

{\it Type 9: $l(\partial\calF)= (L_{21}L_{12}^{-1}R_{21}=w_{21})$}, $I=\{1,2\}$. The vertices of $\calF_I$ are
$(a,b), (a,b-a),(b,b-a),(a-b,b)$. Thus $a=M$, $b=M$ or $a-b=M$. If $b=M$ or $a-b=M$, then the second coordinate of $v$ is maximal (regardless of which of these 4 vertices equals $v$). In these cases we are reduced to type~8 by swapping the roles of the first and
second coordinates.

Thus, we only need to consider the case $a=M$. In this case we use the map in Figure~\ref{9GB}. One can prove that this map
is a reduction similarly to type~8.

\input{figure9GB.tex}

\vskip .2cm

{\it Type 10:  $l(\partial\calF)= (w_{12}^{-1}R_{12}w_{12}=L_{21}^{-1})$}, $I=\{1,2\}$. The vertices of $\calF_I$ are
$(a,b), (b,-a),(b-a,-a)$ and $(a,b-a)$. If $a=M$, the reduction map is shown in Figure~\ref{10GB}. In the remaining cases ($M=b$ or $b-a$) the map is constructed similarly, but will be simpler as all the side cells will be quadrilaterals.

\input{figure10GB.tex}

\vskip .2cm

{\it Type 11:  $l(\partial\calF)= (w_{12}^{-1}R_{13}w_{12}=R_{23})$}, $I=\{1,2,3\}$. The vertices of $\calF_I$ are
$(a,b,c), (-c,b,a)$, $(-c,b-a,a)$ and $(a,b-a,c)$. If $a=M$, the reduction map is shown in Figure~\ref{11GB}. In the remaining case $M=-b$ the diagram is constructed similarly, but will be simpler, as in type 10.

\input{figure11GB.tex}

\vskip .2cm

{\it Type 12:  $l(\partial\calF)= (w_{12}^{-1}R_{31}w_{12}=R_{32})$}, $I=\{1,2\}$. The reduction map in this case is similar to (but simpler than) those for types 10 and 11, as all the side cells will be quadrilaterals.

\subsection{Single-cell reductions for $\IAC_{n,1}$ in the zero coordinate case}
$\empty$

\vskip .2cm

As in the zero coordinate case for $\IAR_{n,1}$, we will assume that $v$ is an interior maximal vertex of $\calF$, 
$v_1$ is good and positive and $v_4=0$. As before, we also set $M=\|v\|=\|\calF\|$.

The majority of reduction maps described below will have side edges labeled by $R_{14}$ or $L_{14}$.

\vskip .2cm

{\it Type 1R: $l(\partial\calF)= ([R_{23},R_{12}]=R_{13})$}, $I=\{2,3,4\}$. All edge labels commute with $L_{14}$. In addition, all
vertices of $\calF$ have the same first coordinate, which by assumption is good. Therefore, the commuting map
$C_{\calF}(L_{14})$ is a reduction.

\vskip .2cm

{\it Type 2R: $l(\partial\calF)= ([R_{13},R_{21}]=R_{23})$}, $I=\{1,3,4\}$. If $b$ is good, the commuting map $C_{\calF}(L_{24})$ is a reduction.
And if $b$ is bad, then $a$ and $a+b$ are both good by Observation~\ref{obs:badsubgp} (since one of them must be good) and hence the map in 
Figure~\ref{2ZB} is a reduction by Lemma~\ref{lem:zeasy}(a).

\input{figure2ZB.tex}

\vskip .2cm

{\it Type 3R: $l(\partial\calF)= [R_{32},R_{21}]R_{31}^{-1}$}, $I=\{1,2,4\}$. All vertices of $\calF$ have the same third coordinate $c$.
If $c$ is good, the commuting map $C_{\calF}(L_{34})$ is a reduction,  so assume now that $c$ is bad. 
We claim that the map in Figure~\ref{3ZB} is a reduction. By Lemma~\ref{lem:zeasy}(a), it suffices to show that $a,a+b$ and $a-c$ are all good.

The vertices of $\calF_I$ are $(a-c,b,0)$, $(a,b,0)$, $(a+b,b,0)$, $(a+b,b+c,0)$ and $(a-c,b+c,0)$.
If $v_I$ is one of the first 3 vertices, $a-c,a$ and $a+b$ are good by Lemma~\ref{lem:zeasy}(b). And if $v$ is one of the last 2,
Lemma~\ref{lem:zeasy}(b) yields that $a-c$ and $a+b$ are good. Since $c$ is bad, we conclude that $a=(a-c)+c$ is also good.

\input{figure3ZB.tex}

{\it Type 4R: $l(\partial\calF)= [R_{21},R_{31}]$}, $I=\{1,4\}$. We will use the map $\phi$ in Figure~\ref{4ZB}. As in type 3R,
$a-b$, $a-c$ and $a-b-c$ are good by Lemma~\ref{lem:zeasy}(b), and hence $\phi$ is a reduction by Lemma~\ref{lem:zeasy}(a).

\input{figure4ZB.tex}

{\it Type 5: $l(\partial\calF)= [R_{12},L_{13}]$}, $I=\{2,3,4\}$. 
 As before, write $v=(a,b,c,0,*)$ (so that $v_I=(b,c,0)$) and we can assume that both edges at $v$ are incoming. Then  
the vertices of $\calF_I$ are $(b,c,0)$, $(b,c-a,0)$, $(b-a,c,0)$ and $(b-a,c-a,0)$.
We claim that the map in Figure~\ref{5ZB} is a reduction.
 
Indeed, each new vertex is obtained from a vertex of $\calF$ by replacing $0$ by $a$ or $a-c$ in the fourth coordinate. 
We know that $a$ is good. If $a-c$ is also good, this map is a reduction by Lemma~\ref{lem:zeasy}(a). And if $a-c$ is bad, then $c=a-(a-c)$ is good
whence $(b,c,0)<(b,c-a,0)$, contrary to the assumption that $v_I=(b,c,0)$.

\input{figure5ZB.tex}

\vskip .2cm

{\it Type 6R: $l(\partial\calF)= [R_{12},R_{35}]$}, $I=\{2,4,5\}$. In this case we can argue exactly as for type 1R. 

\vskip .2cm

{\it Type 7R: $l(\partial\calF)= [R_{21},R_{35}]$}, $I=\{1,4,5\}$. The argument in this case is similar to and easier than type 3R. 

\vskip .2cm

{\it Types 8-12}. The reduction diagrams in these cases have the same edge labelings as the corresponding diagrams in the good coordinate case
for $\IAR_{n,1}$ (type GA), and the justifications are also very similar. For completeness, we will present the diagrams
for types 8,9,10 and 11 in the appendix at the end of the paper.
\vskip .2cm
\subsection{Conclusion of the proof} \label{sec:4GBonedim}

As for the groups $\IAR_{n,1}$, it suffices to prove (the analogue of) Proposition~\ref{prop:removable}
for the groups $\IAC_{n,1}$. 
\begin{proof}[Proof of Proposition~\ref{prop:removable} for the groups $\IAC_{n,1}$] Similarly to the groups $\IAR_{n,1}$,
we consider $3$ cases: $m(v)>9$, $g(v)>8$ and $z(v)>8$ (recall that $m(v),g(v)$ and $z(v)$ denote the numbers of maximal, good and
zero coordinates of $v$, respectively). 

If $m(v)>9$, the proof is exactly the same as for the groups $\IAR_{n,1}$. The proof
in the case $z(v)>8$ is analogous to the case $g(v)>8$ for the groups $\IAR_{n,1}$. Finally, the proof in the case
$g(v)>8$ is mostly similar to the case $z(v)>8$ for the groups $\IAR_{n,1}$, but the argument requires a substantial modification
which is described below.
\vskip .12cm
Thus assume that $g(v)>8$. As in the case $z(v)>8$ for the groups $\IAR_{n,1}$, after two rounds of single-cell reductions,
we can assume that all essential cells containing $v$ have type $2,4,7$ or $12$ (recall that a cell is essential if its support contains $1$). Next we can use double-cell reductions to first eliminate cells of type $12$ and then cells of type $2$, so all the remaining essential cells
have type $4$ and $7$.
The difference comes at the next step. We cannot apply the commuting map of the form $C(R_{i1})$ to cells of type $4$, as this time the first coordinates in a cell of type $4$ may have
opposite signs and thus the map $C(R_{i1})$ may not be a reduction. 

Let us take a closer look at the reduction map in Figure~\ref{4GB}, as we will need to apply it not only to cells whose boundary label is exactly as in Figure~\ref{4GB}, but also to other cells of the same type. The direction and labels of the exterior horizontal edges (labeled by $R_{21}$) are important. On the other hand, if we change the direction of the exterior vertical edges or change their labels from $R_{31}$ to $L_{31}$, we still get a valid van Kampen diagram which yields a reduction map (with different vertex labels in the case of the direction change). Further, one can obtain an analogous reduction map by
replacing all labels $R_{31}$ by either $R_{ij}$ or $L_{ij}$ for fixed distinct $i,j\not\in \{1,2,3\}$ (this will be a single-cell reduction for a cell of type $7$). Below we will refer to such maps as {\it Figure~\ref{4GB}-type reductions}. Also, for brevity we will say that an edge has a 
{\it right} (resp. {\it left}) label if its label is equal to $R_{i1}$ (resp. $L_{i1}$) for some $i$.

The following lemma shows how one can construct multiple-cell reductions starting from Figure~\ref{4GB}-type reductions.

\begin{Lemma}\label{obs:4GB}
Let $\calF$ and $\calF'$ be cells which share an edge $e$ containing $v$, where $\calF$ has type $4$, and let 
$\phi$ be a Figure~\ref{4GB}-type $\calF$-reduction. Assume in addition that one of the following holds:
\begin{itemize}
\item[(i)] $\calF'$ has type $7$.
\item[(ii)] $\calF$ has type $4$ and the side cell of $\phi$ containing $e$ has type $3$ (so $e$ plays the role of the bottom horizontal edge in Figure~\ref{4GB}).
\item[(iii)] $\calF$ has type $4$ and the side cell of $\phi$ containing $e$ has type $4$. Moreover, if $e_1$ (resp. $e_1'$) is the edge of
$\calF$ (resp. $\calF'$) which contains $v$ and is different from $e$, then either $e_1$ and $e_1'$ have the same direction at $v$ and both have right labels or left labels, or $e_1$ and $e_1'$ have opposite directions at $v$ and their labels are $R_{i1}$ and $L_{j1}$ for some $i\neq j$.
\end{itemize}
Then there exists a Figure~\ref{4GB}-type reduction $\phi'$ at $\calF'$ which is compatible with $\calF$ at $e$, and thus $\phi$ can be extended to a double-cell $\calF\cup\calF'$-reduction.
\end{Lemma}
\begin{proof} We will sketch a proof of (iii). The proofs of other parts are similar but easier. If $e_1$ and $e_1'$ are both incoming at $v$
and have right labels, the statement is clear from Figure~\ref{4GB}, and by obvious symmetry, the same is true if $e_1$ and $e_1'$ are both incoming at $v$ and have left labels. 

Suppose now that $e_1$ and $e_1'$ are both outgoing at $v$ and both have right labels, say, $R_{1i}$ and $R_{1j}$ (again the case of two left labels is analogous). If $i=j$, the cells
$\calF$ and $\calF'$ are mirror images of each other, so the assertion of Lemma~\ref{obs:4GB} is trivially true (of course, in this case we could simply cancel $\calF$ and $\calF'$ right away), so assume that $i\neq j$. If we apply the automorphism $x_i\mapsto x_i^{-1}$, $x_j\mapsto x_j^{-1}$,
the edges $e_1$ and $e_1'$ will change their orientation, and their labels will change to $L_{1i}$ and $L_{1j}$, reducing to the case from the first paragraph. 

Finally, the last case where $e_1$ and $e_1'$ have opposite directions, $l(e_1)=R_{i1}$ and $l(e_1')=L_{j1}$ with $i\neq j$ is treated similarly to the previous case. The hypothesis $i\neq j$ is necessary to ensure that we can apply an automorphism to change the direction (and the label)
for exactly one of the edges $e_1$ and $e_1'$ and not both of them.
\end{proof}

We proceed with the proof of Proposition~\ref{prop:removable}. Recall that by assumption, all the remaining cells containing $v$ have type $4$ or $7$ or are not essential. If $\deg(v)=2$, it is routine to construct a double-cell reduction (which will eliminate $v$). 

Next suppose that $\deg(v)> 3$. 
Below we will use a case-by-case argument to show that there exists a triple-cell $\calG$-reduction for some gallery $\calG$ at $v$ such that 
the new cells at $v$ (which must be adjacent to each other) both have type $4$, both have type $7$, or have types $3$ and $7$. In the latter case,
these 2 new cells can then be eliminated using a double-cell reduction of the form $C(R_{i1})$ (this reduction may only produce new cells of types $4$ and $7$). 

\vskip .12cm
{\it Case 1: there is a length $3$ gallery $\calG$ at $v$ not containing any cells of type $4$}. In this case we use the commuting reduction
$C_{\calG}(R_{i1}^{-1})$ for any good index $i\not\in \supp(\calG)$ (as before, such $i$ exists since $g(v)>8$).
\vskip .12cm

{\it Case 2: there is a length $3$ gallery $\calG$ at $v$ whose middle cell has type $4$ and whose other two cells have other types.}
Denote the cells of $\calG$ by $\calF_1,\calF_2$ and $\calF_3$ (where $\calF_2$ is the middle cell). Let $\phi_2$ be any Figure~\ref{4GB}-type
$\calF_2$-reduction. By Lemma~\ref{obs:4GB}(i), there exist Figure~\ref{4GB}-type reductions $\phi_i$ at $\calF_i$ for $i=1,3$ which are compatible with $\phi_2$. Then the map $\phi_1\cup\phi_2\cup\phi_3$ is a $\calG$-reduction with required properties.
\vskip .12cm

{\it Case 3: there is a length $3$ gallery $\calG$ at $v$ which has exactly $2$ cells of type $4$ including the middle one.} Let $\calF$ and
and $\calF'$ be the cells of type $4$ in $\calG$. Thus, $\calF$ and $\calF'$ are adjacent, and the remaining cell of $\calG$
has type $7$ (it cannot be non-essential as it shares an edge with a cell of type $4$). Let $e$ be the unique common edge of $\calF$ and $\calF'$. 
Start with any Figure~\ref{4GB}-type $\calF$-reduction such that the side cell at $e$ has type $3$. By Lemma~\ref{obs:4GB}(ii),
we can extend it to a double-cell reduction $\phi_2$ on $\calF\cup\calF'$. And since the remaining cell of $\calG$ has type $7$,
by Lemma~\ref{obs:4GB}(i), we extend $\phi_2$ to a triple-cell $\calG$-reduction.
\vskip .12cm

{\it Case 4: there is a length $3$ gallery $\calG$ at $v$ all of whose cells have type $4$.} Let $e_1,e_2,e_3,e_4$ be the edges of $\calG$
containing $v$ in counterclockwise order, and for $i=1,2,3$ let $\calF_i$ be the cell of $\calG$ containing $e_i$ and $e_{i+1}$.

Assume first that $e_1$ and $e_3$ satisfy the hypothesis of Lemma~\ref{obs:4GB}(iii). We can start with any 
Figure~\ref{4GB}-type $\calF$-reduction such that the side cell at $e_2$ has type $4$, then use Lemma~\ref{obs:4GB}(iii)
to extend it to a double-cell $\calF_1\cup\calF_2$-reduction $\phi_2$ and finally use Lemma~\ref{obs:4GB}(ii)
to extend $\phi_2$ to a triple-cell $\calG$-reduction.

We now consider the general case. Without loss of generality assume that $l(e_3)=R_{i1}$ for some $i$.
We can always perform a double-cell $\calF_2\cup\calF_3$-reduction such that the side cells at $e_3$
(the common edge of $\calF_2$ and $\calF_3$) have type $3$ (this is possible by Lemma~\ref{obs:4GB}(ii)). This will replace 
$\calF_2\cup\calF_3$ by another gallery of length $2$ with cells of type $4$, but the label of $e_3$ (or, more precisely, the
new edge in the place of $e_3$) will change from $R_{1i}$ to $L_{1k}$, where $k$ is any good index not in $\supp(\calG)$ (which is for us to choose),
and the direction of $e_3$ will stay the same. After applying such double-cell reductions at most twice, we can ensure that $e_1$ and $e_3$ satisfy the hypothesis of Lemma~\ref{obs:4GB}(iii), reducing to the case in the previous paragraph.
\skv
This completes the proof in the case $\deg(v)>3$, and it remains to consider the case where $\deg(v)=3$. In this case $\calG$
will denote the full gallery at $v$. If $\calG$ has no cells of type~$4$, we can use a commuting reduction of the form $C(R_{i1})$,
so assume that $\calG$ has at least one cell of type $4$. Below we will consider the case where all cells have type $4$;
the other cases can be treated similarly but are slightly easier.

Denote the cells of $\calG$ by $\calF_1,\calF_2,\calF_3$ and the edges of $\calG$ containing $v$ by $e_1,e_2,e_3$, so that
$\calF_i$ contains $e_i$ and $e_{i+1}$, with indices taken mod $3$. The edges $e_i$ must have distinct supports since any two of them
are adjacent edges of a cell of type $4$. Applying suitable automorphisms as in the proof of Lemma~\ref{obs:4GB}(iii),
we can assume that all edges $e_i$ are incoming at $v$, and then arguing as in Case~4 above, we can assume that all $e_i$
have right labels. In this case we can construct a full $\calG$-reduction similarly to Case~3 above. Even though we cannot
make single-cell reductions compatible at all edges (we are forced to have side cells of types $3$ and $4$ at one of the edges),
we still obtain a full reduction after cell cancellation -- see Figure~\ref{4GBgallery} below.
\input{figure4GBgallery.tex}
\end{proof}

\section{Finite presentability of $\IAR_{n,d}$ and $\IAC_{n,d}$ for $d>1$}

In this section we will prove Theorem~\ref{thm:finpres} for arbirary $d$.

\subsection{Outline}

Throughout this section we fix integers $1\leq d\leq n$, let $G=\widetilde{\SAut(F_n)}$ and 
$H=\widetilde{\IAR_{n,d}}$ or $\widetilde{\IAC_{n,d}}$. Recall some of the notations and conventions introduced in \S~2: 
\begin{itemize}
\item $(\calX,\calR)$ denotes the optimized Gersten's presentation for $G$ (see \S~2.3); 
\item $Um_{d\times n}(\dbZ)$ denotes the set of $d\times n$ matrices over $\dbZ$ whose columns span $\dbZ^d$; 
\item $Um_{d\times n}(\dbZ)$ is isomorphic to $G/H$ as a $G$-set
where  $G$ acts on $Um_{d\times n}(\dbZ)$ via \eqref{eq:actionrow} if $H=\widetilde{\IAR_{n,d}}$ and
via \eqref{eq:actioncolumn} if $H=\widetilde{\IAC_{n,d}}$. 
\item As before, we will identify $G/H$ with $Um_{d\times n}(\dbZ)$ via the map from Lemma~\ref{obs:isomorphism}. 
\end{itemize}

Unless specifically stated otherwise, 
a diagram in this section will mean a 
$Um_{d\times n}(\dbZ)$-labeled diagram over $(\calX,\calR)$ with the respective action of $G$.
However, $Um_{k\times n}(\dbZ)$-labeled diagrams for $1\leq k<d$ will naturally arise as well 
(see Observation~\ref{obs:reduction} below).

Given a diagram $\Omega$ and a non-empty subset $I$ of $\{1,\ldots, d\}$, we define
$row_{I}\Omega$ to be the projection of $\Omega$ onto the set of rows indexed by $I$, that is,
$row_{I}\Omega$ has the same vertices, edges and edge labels as $\Omega$, and the label of each vertex
in $row_{I}\Omega$ is obtained from the label of the corresponding vertex of $\Omega$ by keeping only
the rows indexed by elements of $I$. We will use simplified notations in two special cases. 
Given $1\leq k\leq d$, we let

\begin{itemize}
\item $row_k \Omega$ be the projection of $\Omega$ onto its $k^{\rm th}$ row, that is, $row_k \Omega=row_{\{k\}} \Omega$ and
\item $bot_k \Omega$ be the projection of $\Omega$ onto its bottom $k$ rows, that is, 
$$bot_k \Omega=row_{\{d-k+1,d-k+2,\ldots, d\}} \Omega.$$
\end{itemize}

The following basic observation will allow us to use many results previously established in the case $d=1$
without an explicit reference to the groups $\IAR_{n,1}$ and $\IAC_{n,1}$:

\begin{Observation}
\label{obs:reduction}
Let $\Omega$ be a $Um_{d\times n}(\dbZ)$-labeled diagram with $G$-action given by \eqref{eq:actionrow} (resp. \eqref{eq:actioncolumn}), and let $I$ be any nonempty subset of $\{1,\ldots, d\}$. Then
$row_I\Omega$ is a $Um_{|I|\times n}(\dbZ)$-labeled diagram with $G$-action given by \eqref{eq:actionrow} (resp. \eqref{eq:actioncolumn}). In particular, for any $1\leq k\leq d$, the diagram $row_k\Omega$ is $Um_{1\times n}(\dbZ)$-labeled.  
\end{Observation}

The reason we retain a $G$-action by projecting to rows and not columns is that in both \eqref{eq:actionrow} and \eqref{eq:actioncolumn} $G$ acts by right multiplication. 
\skv

Before proceeding, we introduce some additional terminology. The invariants below will only be used for matrices
in $Um_{d\times n}(\dbZ)$ but can be defined for arbitrary $d\times n$ matrices over $\dbZ$. 
For each $0\leq k\leq d$
we will identify $\dbZ^k$ with the set of vectors in $\dbZ^d$ whose last $d-k$ coordinates are equal to $0$. 
The {\it depth} of a vector $c\in\dbZ^d$ is the smallest $k$ such that $c\in\dbZ^k$.

\begin{Definition}\rm Let $0\leq k\leq d$.
\begin{itemize}
\item[(a)] Elements of $\dbZ^k$ will be called {\it $k$-vectors}. In other words, $c\in\dbZ^d$ is a $k$-vector if and only if its
depth is $\leq k$ (in particular, a $0$-vector is just the zero vector).
\item[(b)] If $c\in\dbZ^d$ is a $k$-vector which was defined as a column
of some $v\in Mat_{d\times n}(\dbZ)$, we will say that $c$ is a {\it $k$-column} of $v$.
\item[(c)] We will say that $v\in Mat_{d\times n}(\dbZ)$ is 
{\it $k$-unimodular} if its $k$-columns span $\dbZ^k$.
\item[(d)] The {\it $k$-defect} of $v\in Mat_{d\times n}(\dbZ)$ denoted by  $def_k(v)$ is the number of non-$k$-columns of $v$.
\end{itemize}
Now let $\Omega$ be a $Um_{d\times n}(\dbZ)$-labeled diagram.
\begin{itemize}
\item[(d)] We will say that $\Omega$ is {\it $k$-unimodular} if all of its vertices are $k$-unimodular (thus,
$\Omega$ is always $d$-unimodular).
\item[(e)] We define $def_k(\Omega)$, the {\it $k$-defect of $\Omega$}, to be the maximum of the $k$-defects of its vertices.
\end{itemize}
\end{Definition}

\paragraph{\bf One-dimensional vs higher-dimensional settings.} Many definitions and constructions introduced in the special case $d=1$ admit
obvious generalizations for arbitrary $d$. In some other situations, we will use what we did in the case $d=1$ as a tool to establish the corresponding
result for $d>1$. We will refer to the cases $d=1$ and $d>1$ as {\it one-dimensional setting} and {\it higher-dimensional setting}, respectively.
\footnote{We will avoid explicit references to the value of $d$ in such situations to avoid confusion since we treat $d$ as a fixed integer throughout this section.}

\skv
\paragraph{\bf Constant $C$.} Throughout this section $C$ will denote a constant with the property that for any
gallery $\Delta$ of length $\leq 3$ we have $|\supp(\Delta)|\leq C$. By Observation~\ref{obs:support}(ii), the smallest integer with this property is $C=8$, so the reader may simply replace $C$ by $8$ in all formulas
below; however, we feel that some of the computations may be easier to follow if one thinks of $C$ as an unspecified constant.
\skv

\paragraph{\bf Lifting diagram maps.} In \S~3,4 we constructed a number of specific diagram maps in the one-dimensional setting.
If $\phi$ is one of those maps, we can consider the corresponding map $\phi^*$ of  $Um_{d\times n}(\dbZ)$-labeled diagrams.
By definition, the underlying unlabeled\footnote{Recall that this means that the vertices are not labeled, but the edges are still labeled.} van Kampen diagram for $\phi^*$ is the same as for $\phi$ and the labels of all new vertices are obtained from the labels of old vertices by the same formulas.
In this case we will say that $\phi^*$ is a {\it lift of $\phi$}. The majority of diagram maps considered in this section
will be constructed in this way.

\skv

\paragraph{\bf Extended labels.} A new key feature in the higher-dimensional setting is that the order on the vertices
will not be entirely determined by their labels and instead we will use the {\it extended labels}. Let $Z=Um_{d\times n}(\dbZ)$.
Later in this section we will introduce two super-Artinian orders on $Z$, called Step~1 pre-order and Step~2 pre-order 
(these orders will be used in different parts on this proofs).
Relative to each of these orders, we will consider the poset $Z'=Z\times \dbZ_{\geq 0}$ ordered lexicographically, that is,
$(z_1,h_1)<(z_2,h_2)$ if and only if $z_1<z_2$ or $z_1$ and $z_2$ are incomparable and $h_1<h_2$.
It is clear that $Z'$ is also super-Artinian.
 
To each vertex $v$ of a diagram $\Delta$ we will associate its {\it extended label} $L_{\Delta}(v)$
such that $$L_{\Delta}(v)=(l(v),h_{\Delta}(v))$$ for some $h_{\Delta}(v)\in \dbZ$, that is, the
$Z$-component of the extended label of $v$ is its usual label.

The group $G$ will only act on $Z$, not on $Z'$, but this will not introduce any complications for the proof.
\skv
A vertex $v\in V(\Delta)$ will be called {\it maximal for $\Delta$} if $L_{\Delta}(v)$ is a maximal element
of the set $\{L_{\Delta}(w): w\in V(\Delta)\}$.

\skv
Let $\psi:\Delta\to\Delta'$ be a diagram map. Recall that in the one-dimensional setting, a vertex $v\in V(\Delta')$
was called new (for $\psi$) if $v\in V(\Delta')\setminus V(\Delta)$ and old otherwise. With the use
of extended labels, things become more complicated as a vertex may remain in the diagram, but acquire a new extended label, and our original definition of reduction becomes ambiguous in the presence of such vertices.

To address this issue, in the higher-dimensional setting we will use the following terminology:

\begin{Definition}\rm \label{def:pseudonew}
Let $\psi:\Delta\to\Delta'$ be a diagram map. A vertex $v\in V(\Delta')$ will be called
\begin{itemize}
\item[(i)] {\it new} if $v\in V(\Delta')\setminus V(\Delta)$;
\item[(ii)] {\it old} if $v\in V(\Delta')\cap V(\Delta)$ and $h_{\Delta'}(v)=h_{\Delta}(v)$;
\item[(iii)] {\it pseudo-new} if $v\in V(\Delta')\cap V(\Delta)$ but $h_{\Delta'}(v)\neq h_{\Delta}(v)$.
\end{itemize}
\end{Definition}
Thus, by our convention a pseudo-new vertex belongs to both the old diagram $\Delta$ and the new diagram $\Delta'$,
but represents different elements of $Z'$ depending on which diagram we are considering. This is why we explicitly refer to the diagram in the notation for the extended label.

We can now give a formal definition of a reduction applicable in the higher-dimensional setting.
 
\begin{Definition}\rm
\label{def:reductionhigh}
Let $\psi:\Delta\to\Delta'$ be a diagram map and $v\in V(\Delta)$ a maximal vertex of $\Delta$. We will say that $\psi$
is a {\it reduction} if $L_{\Delta'}(w)<L_{\Delta}(v)$ for any new or pseudo-new vertex $w$.
\end{Definition}
As in the one-dimensional setting, it is straightforward to check that this definition does not depend on the choice
of a maximal vertex $v$.
\skv

While this definition technically does not fit the setting of Section~2, the assertion of Lemma~\ref{obs:Artinian} remains valid (with the same proof). 
\skv
\skv

\paragraph{\bf Outline of the proof of Theorem~\ref{thm:finpres} for $d>1$.}
As in the one-dimensional setting, given a vertex $v$ of some diagram, by $\|v\|$ we will denote the
$\ell^{\infty}$-norm of its label. Given a diagram $\Delta$, we let $\|\Delta\|=\max\{\|v\|: v\in V(\Delta)\}$.

\skv
We need one more technical definition:

\begin{Definition}\rm 
\label{def:small}
A matrix $v\in Um_{d\times n}(\dbZ)$ will be called {\it small} if
\begin{itemize}
\item[(i)] $\|v\|=1$;
\item[(ii)] $v$ differs from the $d\times n$ ``identity'' matrix $(I_{d\times d} \mid 0_{d\times (n-d)})$
in at most $8$ columns. In particular, every row of $v$ has at most $9$ nonzero entries and hence by (i) has $\ell^1$-norm $\leq 9$;
\item[(iii)] $v$ is $k$-unimodular for all $k\leq d$.
\end{itemize}
A subset of $Um_{d\times n}(\dbZ)$ is {\it small} if all of its elements are small.
\end{Definition}

The following result will be established at the end of this section.

\begin{Proposition}
\label{prop:essential2} There exists a small subset $E$ of $G/H$ whose preimage in $Cay(G,\calX)$ is connected
 (here we identify $G/H$ with $Um_{d\times n}(\dbZ)$ as before).
\end{Proposition}

Unlike the case $d=1$, we will formally prove finite presentability of $H$ for $d>1$ using 
our general finite presentability criterion, Proposition~\ref{thm:criterion}, bypassing Lemma~\ref{obs:Renz}.
The latter is actually not applicable since the partial orders we will be using for $d>1$ are only assumed super-Artinian, but not necessarily strongly Artinian.

More precisely, we will use Corollary~\ref{cor:criterion} below which is an easy consequence of
Propositions~\ref{prop:essential2}~and~\ref{thm:criterion}.

\begin{Corollary}
\label{cor:criterion}
Suppose that there exists a constant $M$ such that for any $Um_{d\times n}(\dbZ)$-labeled diagram $\Omega$ over $(\calX,\calR)$ with small boundary there exists a diagram map $\Omega\to\Omega'$ such that $\|\Omega'\|\leq M$. Then $H$ is finitely presented. 
\end{Corollary}
\begin{proof} Let $E$ be the small set from Proposition~\ref{prop:essential2}. Corollary~\ref{cor:criterion} implies that the hypotheses of Proposition~\ref{thm:criterion} hold if $A=E$ and $B$ is the set of matrices in $Um_{d\times n}(\dbZ)$ with $\ell^{\infty}$-norm at most $M$ (clearly, this is a finite set). Hence by Proposition~\ref{thm:criterion}, $H$ is finitely presented.
\end{proof}

\paragraph{\bf How to verify the hypothesis of Corollary~\ref{cor:criterion}.}
Let $\Omega$ be any diagram with small boundary. Our goal is to show that the hypothesis of Corollary~\ref{cor:criterion} is always satisfied, that is, to construct a diagram map $\Omega\to \Omega'$ with $\Omega'$ of bounded norm.
This will be accomplished in several stages where at each stage we will concentrate on a particular row of the diagram. We start by projecting the entire diagram $\Omega$ onto its last row $row_d \Omega$. The hypothesis that $\Omega$ has small boundary implies, in particular, that for any $v\in V(\partial\Omega)$ we have $\|v\|=1$ and $\|row_d(v)\|_1\leq 9$. Since $C\geq 8$, 
applying Proposition~\ref{prop:removable} sufficiently many times,
we obtain a reduction map $\phi:row_d \Omega\to\Delta_d$ such that $\|v\|=1$ and $\|v\|_{1}\leq C+1$ for every vertex $v$ of $\Delta_d$. We can now lift $\phi$ to construct a diagram map 
$\phi^*:\Omega\to\Delta$ such that for each vertex $v$ of $\Delta$ we have $\|row_d(v)\|=1$ and $\|row_d(v)\|_1\leq C+1$, and so $def_{d-1}(\Delta)\leq C+1$. The new diagram $\Delta$ 
is still $Um_{d\times n}(\dbZ)$-labeled and thus $d$-unimodular. 
\skv

Next we would like to apply a similar procedure to the $(d-1)^{\rm st}$ row of $\Delta$, but we have to be careful as we do not want to lose the nice structure of the $d^{\rm th}$ row. In order to accomplish this, we first perform additional reductions involving the $d^{\rm th}$ row to make our diagram $(d-1)$-unimodular. This can be done preserving the condition $\|row_d(v)\|=1$ for all $v$ and keeping the $(d-1)$-defect uniformly bounded (although with a slightly worse bound than $C+1$). Once we obtain a $(d-1)$-unimodular diagram satisfying these additional conditions, we can lift
one-dimensional reductions (with some minor modifications) to make 
both $\|row_{d}\|$ and $\|row_{d-1}\|$ equal to $1$ and making the $(d-2)$-defect bounded. We then successively apply the same procedure to each row (going from bottom to top), so in the end we make the $\ell^{\infty}$-norm of the entire diagram equal to $1$.
\skv

The algorithm we have outlined so far will only work when $d\leq \frac{n}{C}$ since it can only ensure that the $k$-defect 
is bounded above by $C(d-k)$ (and we need the $k$-defect to be bounded away from $n$ to deal with the $k^{\rm th}$ row). In order
to overcome this problem, we will make additional adjustments after dealing with the $k^{\rm th}$ row (as described in the previous paragraph) which will yield a better bound on $k$-defect: $def_k\leq d-k+C$. Unfortunately, this will be done at the expense of substantially increasing the norms of the last $d-k$ rows, which is why in the end we will only be able to claim that the $\ell^{\infty}$-norm is absolutely bounded (rather than bounded by $1$). This concludes a (very) informal sketch of the algorithm which will be described in this section.
\skv

We now formulate the main result of this section:

\begin{Theorem}
\label{thm:hdim}Assume that $1\leq d\leq n-(C^2+6C+3)$. Then there exists a sequence of positive integers $M_1,\ldots, M_d$ with the following property.
Suppose that for some $1\leq k\leq d$ we are given a diagram $\Omega$ such that
\begin{itemize}
\item[(a)] $\Omega$ is $k$-unimodular;
\item[(b)] $def_k(\Omega)\leq d-k+C$;
\item[(c)] if $k<d$, then $\|bot_{d-k}\Omega\|\leq M_{k+1}$;
\item[(d)] $\Omega$ has small boundary.
\end{itemize}
Then there exists a diagram map $\Omega\to \Omega_1$ such that
\begin{itemize}
\item[(i)] $\Omega_1$ is $(k-1)$-unimodular;
\item[(ii)] $def_{k-1}(\Omega_1)\leq d-k+C+1$;
\item[(iii)] $\|bot_{d-k+1}\Omega_1\|\leq M_{k}$.
\end{itemize}
\end{Theorem}

First we will deduce Theorem~\ref{thm:finpres} from Theorem~\ref{thm:hdim}.

\begin{proof}[Proof of Theorem~\ref{thm:finpres}]
Observe that $(C^2+6C+3)=115$ for $C=8$, so any $d\leq n-115$ satisfies the hypotheses of Theorem~\ref{thm:hdim}. 
We will prove finite presentability of $H$ using Corollary~\ref{cor:criterion}.

First note that any $Um_{d\times n}(\dbZ)$-labeled diagram $\Omega$ 
satisfies condition (a)-(c) from Theorem~\ref{thm:hdim} for $k=d$. Indeed,
$\Omega$ is $d$-unimodular by assumption, $def_d(v)=0$ for all $v\in\dbZ^d$ 
(and hence $def_d(\Omega)=0$), and (c) is vacuous for $k=d$. 

Assume now that $\Omega$ has small boundary, so (d) holds as well and thus we can apply Theorem~\ref{thm:hdim} with $k=d$ to 
$\Omega$. The obtained diagram $\Omega_1$ satisfies conditions
(i)-(iii) for $k=d$, which are the same as (a)-(c) for $k=d-1$. Since diagram maps do not change the boundary,
$\Omega_1$ also satisfies (d) (which is the same for all $k$). Thus, we can apply Theorem~\ref{thm:hdim} to $\Omega_1$ with $k=d-1$ and keep going up to (and including) $k=1$. If $\Omega_{fin}$ is the diagram obtained at the end, by condition (iii) for $k=1$ we have $\|\Omega_{fin}\|=\|bot_d(\Omega_{fin})\|\leq M_1$, so $H$ is finitely presented by Corollary~\ref{cor:criterion}.
\end{proof}

Before discussing the structure of the proof of Theorem~\ref{thm:hdim} in detail, we introduce some auxiliary terminology.

\begin{Definition}\rm \label{def:weakdefect}
Let $0\leq k\leq d$.
\begin{itemize}
\item[(i)] Define a relation $\sim_{k}$ on $\dbZ^d$ by $v\sim_k w$ if
either $w-v$ or $w+v$ is a $k$-vector, that is, $w$ agrees with either $v$ or $-v$ in the last $d-k$ rows.
It is clear that $\sim_{k}$ is an equivalence relation. 
\item[(ii)]Given $v\in Mat_{d\times n}(\dbZ)$, its {\it weak $k$-defect} denoted by $wdef_k(v)$ is the 
maximum number of non-$k$-columns of $v$ which are pairwise inequivalent with respect to $\sim_k$. 
\item[(iii)]If $\Omega$ is a diagram, we define $wdef_k(\Omega)$ as the maximum value of $wdef_k(v)$ where $v$ is a vertex of $\Omega$.
\end{itemize}
\end{Definition}

Clearly, $wdef_k(v)\leq def_k(v)$. Of course, the difference between $def_k(v)$ and $wdef_k(v)$ can be arbitrarily large; however, 
by Proposition~\ref{step1b} below any diagram admits a simple modification such that the difference
$def_k(v)-wdef_k(v)$ is bounded by an absolute constant for every vertex $v$ of the modified diagram. Thus, the $k$-defect is 
to a large extent controlled by the weak $k$-defect. The advantage of working with the latter is that, as we will see below, it is much easier to construct reductions which do not increase the weak $k$-defect.

\begin{Definition}\rm 
A diagram map $\Delta\to\Delta'$ will be called {\it $k$-safe} if
\begin{itemize}
\item[(A)] $\|row_i\Delta'\|\leq \|row_i\Delta\|$ for all $i>k$, so in particular $\|bot_{d-k}\Delta'\|\leq \|bot_{d-k}\Delta\|$;
\item[(B)] $wdef_k(\Delta')\leq wdef_k(\Delta)$;
\item[(C)] if $\Delta$ is $k$-unimodular, then so is $\Delta'$.
\end{itemize}
\end{Definition}

Very informally, a diagram map is $k$-safe if it does not make the structure of the last $d-k$ rows of the diagram more complicated.
\skv

For the rest of this section we fix $1\leq k\leq d$ and view it as $k$ from the statement of Theorem~\ref{thm:hdim}. 
The proof of Theorem~\ref{thm:hdim} will be divided into 3 steps: 
\begin{itemize}
\item In Step~1 we will make the $\ell^{\infty}$-norm of the $k^{\rm th}$ row equal to $1$. 
\item In Step~2 we will make the diagram $(k-1)$-unimodular.
\item Finally in Step~3 we will obtain the desired bound on the $(k-1)$-defect while keeping the norms of the last $d-k+1$ rows bounded.
\end{itemize} 

\subsection{Step~1}

\skv

We start by formulating the main result of the first step of the proof of Theorem~\ref{thm:hdim}. Recall that
the notion of super-Artinian partial order was defined in \S~2.2.

\begin{Theorem}[Step 1] 
\label{thm:step1}
Let $\Omega$ and $d$ be as in Theorem~\ref{thm:hdim}.
There exists a super-Artinian partial order (defined below and called the {\it Step~1 order}) and a $k$-safe reduction
$\phi:\Omega\to \Omega'$ (with respect to that order) such that
\begin{itemize}
\item[(i)] $def_{k}(\Omega')\leq def_k(\Omega)+C$ and therefore $def_{k}(\Omega')\leq d-k+2C$
\newline
(recall that $def_k(\Omega)\leq d-k+C$ by the hypotheses of Theorem~\ref{thm:hdim});
\item[(ii)] the $k^{\rm th}$ row of any vertex of $\Omega'$ has $\ell^{\infty}$-norm $1$.
\end{itemize}
\end{Theorem}

The assertion of Theorem~\ref{thm:step1} is an easy consequence of the following two propositions.

\begin{Proposition}
\label{step1a}
Let $\Delta$ be a $k$-unimodular diagram with small boundary such that
\begin{itemize}
\item[(a)] $def_{k}(\Delta)\leq n-4C-k-1$;
\item[(b)] $\|row_k(\Delta)\|\geq 2$.
\end{itemize}
Then for any maximal vertex $v$ of $\Delta$ there exists a $k$-safe reduction $\Delta\to\Delta'$ which eliminates $v$.
\end{Proposition}
\begin{Remark}\rm As we will explain later, the hypotheses of Proposition~\ref{step1a} imply that
$v$ must be an interior vertex (see Observation~\ref{obs:goodcolumn} below).
\end{Remark}

\begin{Proposition}
\label{step1b}
Any diagram $\Delta$ with small boundary admits a $k$-safe reduction $\Delta\to \Delta'$ such that $def_{k}(\Delta')\leq wdef_k(\Delta')+C$.
\end{Proposition}

Let us first explain how Theorem~\ref{thm:step1} follows from Propositions~\ref{step1a}~and~\ref{step1b}. 

\begin{proof}[Proof of Theorem~\ref{thm:step1}]
We start with the initial diagram $\Omega$ and keep applying one of these propositions as long as we can. More precisely, if at some
point we have a diagram $\Delta$ with $def_{k}(\Delta)> wdef_k(\Delta)+C$, apply Propositions~\ref{step1b}; otherwise apply
Proposition~\ref{step1a} as long as its hypotheses are satisfied. Since reductions in Proposition~\ref{step1a}
are full reductions and the order is Artinian, the process will terminate after finitely many steps by Lemma~\ref{obs:Artinian}. Let $\Omega'$ be the resulting diagram, and let $\phi:\Omega\to\Omega'$ be the composition of all reductions from Propositions~\ref{step1a}~and~\ref{step1b}. 
We claim that $\phi$ and $\Omega'$ satisfy all the required conditions.

\skv
Since the reductions in Propositions~\ref{step1a}~and~\ref{step1b} are $k$-safe, so is $\phi$.
In particular, $wdef_k(\Omega')\leq wdef_k(\Omega)$. Since Proposition~\ref{step1b} is not applicable to $\Omega'$ (in a non-trivial way),
we have $def_{k}(\Omega')\leq wdef_k(\Omega')+C$ and thus $def_{k}(\Omega')\leq wdef_k(\Omega)+C\leq def_k(\Omega)+C$,
so condition (i) in Theorem~\ref{thm:step1} holds. 

Now recall that $d\leq n-(C^2+6C+3)< n-6C-3$ and hence
$def_{k}(\Omega')\leq (n-6C-3)-k+2C=n-k-4C-3$, so $\Omega'$ satisfies condition (a)
of Proposition~\ref{step1a}. Since Proposition~\ref{step1a} is not applicable, 
condition (b) does not hold, that is, $\|row_k\Omega'\|=1$.
Thus $\Omega'$ satisfies both (i) and (ii) from the conclusion of Theorem~\ref{thm:step1}.
\end{proof}

Before turning to the proofs of Propositions~\ref{step1a}~and~\ref{step1b}, we need
to define the order that will be used in Step~1. First we introduce some additional terminology.

\begin{Definition}\rm Fix $k$ with $1\leq k\leq d$.
Let $v\in Mat_{d\times n}(\dbZ)$ and $c$ a column of $v$. We will say that
\begin{itemize}
\item[(i)] $c$ is {\it maximal} if $|c_k|\geq |d_k|$ for any column $d$ of $v$; 
\item[(ii)] $c$ is {\it good} if $c$ is a $k$-column, $c$ is not maximal and $c_k\neq 0$.
\end{itemize}
\end{Definition}
Note that $c$ is maximal if and only if $c_k$ is a maximal coordinate of $row_k(v)$. If $c$ is good, then $c_k$ must be a good
coordinate of $row_k(v)$, but the converse is not true (since $row_k(c)$ cannot tell us whether $c$ is 
a $k$-column or not).

\skv

As in the one-dimensional setting, we will also need the corresponding notions relative to a $Um_{d\times n}(\dbZ)$-labeled diagram $\Omega$.
Given such a diagram $\Omega$, let $M=\|row_k(\Omega)\|$. A vector $c\in\dbZ^d$ with $|c_k|\leq M$ will be called 
{\it $\Omega$-maximal} (resp. {\it $\Omega$-good}) if $|c_k|=M$ (resp. $c$ is a $k$-vector and $0<|c_k|<M$).
\skv

As we will see shortly, there is a clear distinction between the roles played by $k$-columns and
non-$k$-columns in the definition of the Step~1 order and the construction of reductions in the proof of Proposition~\ref{step1a}. To make this distinction
more transparent, we will say that
\begin{itemize}
\item A vector $c\in\dbZ^d$ is {\it unrestricted} if $c$ is a $k$-vector and
{\it restricted} otherwise.
\end{itemize}
We will explain why we call $k$-vectors {\it unrestricted} in the remark below.

\begin{Observation}
\label{obs:restricted}
Let $v,v'\in Mat_{d\times n}(\dbZ)$, and suppose that $v'$ is obtained from $v$
by replacing $col_i(v)$ by $col_i(v)\pm col_j(v)$ for some $i\neq j$. 
If $col_j(v)$ is a $k$-column, then
\begin{itemize}
\item[(a)] $v$ and $v'$ agree in the last $d-k$ rows;
\item[(b)] $wdef_k(v')=wdef_k(v)$;
\item[(c)] the $k$-columns of $v$ and $v'$ span the same subspace of $\dbZ^d$.
\end{itemize}
 \end{Observation}
\begin{proof}
(a) holds by definition. It also implies that each column of $v'$ is equivalent
to the corresponding column of $v$ with respect to $\sim_k$ which yields (b) (recall that
$\sim_k$ was defined in Definition~\ref{def:weakdefect}).
Finally, since $col_j(v)$ is a $k$-column, $col_i(v)$ and $col'_i(v)=col_i(v)\pm col_j(v)$
are both $k$-columns or both non-$k$-columns; in either case (c) holds.
\end{proof}

\begin{Remark}\rm
Recall that the reduction in the statement of Theorem~\ref{thm:step1} is 
required to be $k$-safe.
Suppose now that $\psi$ is a diagram map such that every new vertex 
$v'$ is obtained from some vertex $v$ by replacing $col_i(v)$ by $col_i(v)\pm col_j(v)$ for some $i\neq j$. If $col_j(v)$ is a $k$-column in all cases,
Observation~\ref{obs:restricted} immediately implies that $\psi$ is $k$-safe
(see Lemma~\ref{lem:ksafe} below). Thus, when constructing a $k$-safe map, we can add (resp. subtract) a $k$-column to (resp. from) any other column without any restrictions. 
\end{Remark}
\skv

\paragraph {\bf Step 1 order.} Let us now define the Step~1 order. As in the one-dimensional setting, it will be defined as a comparison algorithm, where we stop as soon as two vertices are comparable via a given criterion. We warn the reader that this order will only be used in step 1. Different orders will be used in steps 2 and 3.
\skv 
As we already mentioned, a key difference with the one-dimensional setting is that to compare two vertices, it will in general be insufficient to know their labels. But first we define the Step 1 pre-order which only takes labels into account.

\skv
\paragraph{\bf Step 1 pre-order:}
\begin{itemize}
\item[(1)] The vertex with smaller $\ell^{\infty}$-norm of the $k^{\rm th}$ row is smaller.
\item[(2)] The vertex with fewer maximal columns is smaller.
\item[(3)] The vertex with fewer non-maximal restricted columns is smaller. Equivalently, the vertex
with more non-maximal unrestricted columns is smaller.
\item[(4)] The vertex with more good columns is smaller.
\item[(5)] The vertex with more $(k-1)$-columns is smaller.
\item[(6)] Let $v^{\rm ur}$ be the matrix obtained from the label of $v$ by removing all the restricted columns. 
The vertex $v$ with the smaller value of the vector 
$$(\|row_k(v^{ur})\|_1,\|row_{k-1}(v^{ur})\|_1,\ldots,\|row_{1}(v^{ur})\|_1)$$ 
(with respect to the lexicographical order, moving from left to right) is smaller.
\end{itemize}
Given two vertices $v$ and $w$, we will write $v<w$ if $v$ is smaller than $w$ relative to Step~1 pre-order and $v\sim w$
if $v$ and $w$ are incomparable. 
\skv
  
Before proceeding, we need a technical definition.

\begin{Definition}\rm Let $\Delta$ be a diagram.
\begin{itemize}
\item[(a)] A vertex $v$ of $\Delta$ will be called {\it pre-maximal} if its label $l(v)$ is maximal relative to the Step~1 pre-order.
\item[(b)] A cell $\calF$ of $\Delta$ will be called {\it purely maximal} if all of its vertices are pre-maximal.
\item[(c)] Given a vertex $v$ of $\Delta$, we define $h_{\Delta}(v)$ to be the number of purely maximal cells of $\Delta$ which contain $v$ and have type $5$.
\end{itemize}
\end{Definition} 

As we will see below, given a cell $\calF$ containing a pre-maximal vertex $v$,
one can construct a single-cell $\calF$-reduction with respect to the Step~1 pre-order in almost all cases, 
except when $\calF$ is purely maximal and has type 5. What we can do in the latter case is construct a diagram map $\phi$ such that
for every new vertex $w$ of $\phi$ either $w < v$ (relative to the Step~1 pre-order) or $w\sim v$, but $h_{\Delta}(w)<h_{\Delta}(v)$. In order to turn this map into a reduction, we introduce extended labels as follows.

The extended labeling set is $Z'=Z\times \dbZ_{\geq 0}$, where $Z=Um_{d\times n}(\dbZ)$ as before. Given a vertex $v\in V(\Delta)$, its 
extended label is $L_{\Delta}(v)=(l(v),h_{\Delta}(v))$ where $l(v)$ is the usual label.
Recall that $Z'$ is ordered lexicographically, that is, 
\skv
\centerline{$L_{\Delta}(v)<L_{\Delta'}(v')$ if and only if
$l(v)<l(v')$ or $l(v)\sim l(v')$ and $h_{\Delta}(v)<h_{\Delta'}(v')$.}

\skv
\paragraph{\bf Notation:} Given vertices $v$ and $w$, we will still write $v<w$ only when $l(v)<l(w)$. Whenever an inequality between vertices takes
extended labels into account, those extended labels will be mentioned explicitly.
\skv

Recall from Definition~\ref{def:pseudonew} that given a diagram map $\psi:\Delta\to\Delta'$, a vertex $v$ is called pseudo-new for $\psi$ if $v\in V(\Delta')\cap V(\Delta)$
but $L_{\Delta'}(v)\neq L_{\Delta}(v)$. Thus, a pseudo-new vertex will be considered to belong to both the old and the new diagrams,
but with different extended labels.

While it is difficult to control which vertices become pseudo-new, there is a sufficient condition for a diagram map to be a reduction 
which only depends on the extended labels of new vertices.

\begin{Lemma}
\label{lem:enhancedreduction}
Let $\psi:\Delta\to\Delta'$ be a diagram map and $v$ a maximal vertex of $\Delta$ (with respect to Step~1 order). 
Suppose that 
\begin{itemize}
\item[(i)] $L_{\Delta'}(w) < L_{\Delta}(v)$ for every new vertex $w$ of $\psi$; 
\item[(ii)] every new cell of $\psi$ of type $5$ has at least one vertex $u$ (not necessarily new) with $u<v$. 
\end{itemize}
Then $\psi$ is a reduction (with respect to Step~1 order) .
\end{Lemma}

If $w<v$ for every new vertex of $\psi$, conditions (i) and (ii) hold automatically 
(for (ii) this is true since every new cell must contain at least one new vertex). Thus, we have the following special
case of Lemma~\ref{lem:enhancedreduction} which will be applicable to most cell types:

\begin{Lemma}
\label{lem:easyreduction}
Let $\psi:\Delta\to\Delta'$ be a diagram map and $v$ a maximal vertex of $\Delta$.
If $w<v$ for every new vertex of $\psi$, then $\psi$ is a reduction.
\end{Lemma}

\begin{proof}[Proof of Lemma~\ref{lem:enhancedreduction}] Because of condition (i), we just need to show that $L_{\Delta'}(w)< L_{\Delta}(v)$ whenever $w$ is a pseudo-new vertex.

Let $w$ be a pseudo-new vertex, so that $h_{\Delta}(w)\neq h_{\Delta'}(w)$.
If $h_{\Delta'}(w)< h_{\Delta}(w)$, then $L_{\Delta'}(w)=(l(w),h_{\Delta'}(w)) < L_{\Delta}(w)=(l(w),h_{\Delta}(w))$ and hence 
$L_{\Delta'}(w) < L_{\Delta}(v)$  (as $v$ is maximal for $\Delta$).

Assume now that  $h_{\Delta'}(w)> h_{\Delta}(w)$. This means that $\Delta'$ has a new cell $\calF$ which is purely maximal, has type $5$ and contains $w$. By condition (ii), $\calF$ contains a vertex $u$ with $l(u)<l(v)$. Since
$\calF$ is purely maximal and contains $w$, we must have $l(w)\sim l(u)$. Hence  $l(w)< l(v)$ as well, and therefore 
$L_{\Delta'}(w)< L_{\Delta}(v)$ as desired.
\end{proof}
 
\skv

\begin{proof}[Proof of Proposition~\ref{step1b}]
If $def_{k}(\Delta)\leq wdef_k(\Delta)+C$, there is nothing to do, so suppose that the opposite inequality holds. Thus there exists
a vertex $v$ such that $def_{k}(v)> wdef_k(v)+C$; let us call any such $v$ {\it irregular}. 
Condition (ii) in the definition of a small matrix (Definition~\ref{def:small}) implies that small vertices cannot
be irregular. Since we assume that $\Delta$ has small boundary, irregular vertices must be interior. Among all irregular vertices choose one with the largest $k$-defect and denote it by $v$. 

Let $\calG$ be any gallery at $v$ of length $\min(3,\deg(v))$, that is, a
length $3$ gallery at $v$ or the full gallery at $v$ if $\deg(v)=2$. 
Recall that $|\supp(\calG)|\leq C$ (by the definition of $C$).
Since $v$ is irregular, there exist distinct indices $i,j\not\in \supp(\calG)$ such that $col_i(v)\sim_k col_j(v)$ and $col_i(v)$ (and hence also $col_j(v)$) is restricted.  
Then the commuting map $C_{\calG}(R_{ij}^{\pm 1})$ is a reduction (for a suitable choice of sign) as it does not increase the $\ell^{\infty}$-norm or the number of maximal columns and increases the number of non-maximal unrestricted columns. It is also 
straightforward to check that this map is $k$-safe. 
Since $C_{\calG}(R_{ij}^{\pm 1})$
decreases the degree of $v$, applying this operation to various $\calG$, we will eliminate $v$ after finitely many iterations.

Also note that the new vertices arising from the maps $C_{\calG}(R_{ij}^{\pm 1})$ have $k$-defect smaller than $def_k(v)$. Since $v$ was an irregular
vertex with largest $k$-defect, after applying the procedure in the previous paragraph to finitely many vertices, we will
obtain a diagram $\Delta'$ which has no irregular vertices, that is, $def_{k}(\Delta')\leq wdef_k(\Delta')+C$.
\end{proof}

Before turning to the proof of Proposition~\ref{step1a}, we state the higher-dimensional analogues of  Lemma~\ref{lem:goodeasy}~and~\ref{lem:zeasy} which provide
a sufficient condition for a diagram map to be a reduction. The proofs are completely analogous to the
one-dimensional setting.

\begin{Lemma}[analogue of Lemma~\ref{lem:goodeasy}]
\label{lem:goodeasyH}
Let $\Delta$ be a diagram and $M=\|row_k(\Delta)\|$.
The following hold:
\begin{itemize}
\item[(a)] Let $w$ be a vertex of $\Delta$, let $I$ be a subset of $\{1,\ldots, n\}$, and assume that $|w_{kt}|=M$
for some $t\in I$ (here $w_{kt}$ is the $(k,t)$-entry of $w$). Let $\widetilde w\in\dbZ^n$ be such that
\begin{itemize}
\item[(i)] $\widetilde w_{kj}=w_{kj}$ for all $j\not\in I$;
\item[(ii)] $|\widetilde w_{ki}|<M$ for all $i\in I$.
\end{itemize}
\skv
If $w'\in Mat_{d\times n}(\dbZ)$ is obtained from $\widetilde w$ by a signed permutation of columns, then $w'<w$.
\skv
\item[(b)] Let $\psi:\Delta\to\Delta'$ be a diagram map and assume that every new vertex $w'$ of $\psi$ can be constructed from some old vertex $w$ as in (a). Then $\psi$ is a reduction.
\end{itemize}
\end{Lemma}

\begin{Lemma}[analogue of Lemma~\ref{lem:zeasy}]
\label{lem:zeasyH}
Let $\psi:\Delta\to \Delta_1$ be a diagram map and let $v$ be a maximal vertex of $\Delta$.
The following hold:
\begin{itemize}
\item[(a)] Let $w,w'\in Mat_{d\times n}(\dbZ)$, and suppose that $w'$ is obtained from $w$ by replacing one or several $\Delta$-bad columns by  $\Delta$-good columns. If $w$ is a vertex of $\Delta$, then $w'<v$. Hence if every new vertex $w'$ of $\psi$ is obtained from some old vertex $w$ in this way, then $\psi$ is a reduction.
\item[(b)] Suppose that $col_i(v)$ is good for some $i$. Let $u$ be another vertex of $\Delta$, which agrees with $v$ apart from the $i^{\rm th}$ column. Then the $i^{\rm th}$ column of $u$
is $\Delta$-good.
\end{itemize} 
\end{Lemma}

We will also need the following sufficient condition for a diagram to be $k$-safe. Recall that we call $c\in \dbZ^d$ unrestricted (resp. restricted)
if $c$ is a $k$-vector (resp. non-$k$-vector).

\begin{Lemma}
\label{lem:ksafe}
Let $\psi$ be a diagram map. Suppose that every new vertex $w$ of $\psi$ is obtained from an old vertex by the composition of some of the following operations:
\begin{itemize}
\item[(a)] a signed permutation of columns;
\item[(b)] adding (resp. subtracting) an unrestricted column to (resp. from) another column, e.g. 
$(a,b,c,x,\ldots)\mapsto (a+x,b,c,x,\ldots)$ where $x$ is unrestricted;
\item[(c)] replacing a $0$-column by another (possibly restricted) column;
\item[(d)] replacing a restricted column or a $0$-column by an unrestricted vector;
\item[(e)] if the vertex has 2 identical columns, replacing one of them by $0$, e.g. $(a,b,a,\ldots)\mapsto (a,b,0,\ldots)$.
\end{itemize}
Then $\psi$ is $k$-safe.
\end{Lemma}
\begin{proof}
Operations (a)-(e) do not increase the $\ell^{\infty}$-norm for any of the last $d-k$ rows, do not
increase the weak $k$-defect and preserve $k$-unimodularity. This holds for (b)
by Observation~\ref{obs:restricted} and is obvious for the remaining operations. Thus, any map obtained as composition of 
these operations is $k$-safe.
\end{proof}

We now begin the proof of Proposition~\ref{step1a}.

\begin{proof}[Proof of Proposition~\ref{step1a}]
Recall that $def_k(v)\leq n-4C-k-1$, so $v$ has at least $4C+k+1$ unrestricted columns. Hence (at least) one of the following holds:
\begin{itemize}
\item[(1)] $v$ has at least $C+k$ columns of depth between $1$ and $k-1$;
\item[(2)] $v$ has at least $C+2$  maximal unrestricted columns;
\item[(3)] $v$ has at least $C+1$ good columns;
\item[(4)] $v$ has at least $C+1$ $0$-columns.
\end{itemize}

We will consider these 4 cases separately. Cases 1 and 2 are easier, and here we will explicitly describe how to construct a sequence of reductions which eliminates $v$. Cases 3 and 4 correspond to the good coordinate case and the zero coordinate case in the one-dimensional setting, respectively. In these
cases the desired reduction map will be constructed as a composition of single-cell, double-cell and triple-cell reductions. Double-cell and triple-cell reductions will be obtained as combinations of compatible single-cell reductions mostly as in the one-dimensional setting,
but we will need a minor modification involving cells of type 4GB.

\skv
For the rest of the proof we fix a maximal vertex $v$.

\begin{Observation}
\label{obs:goodcolumn}
The vertex $v$ is interior and has at least one good column.
\end{Observation}
\begin{proof}
Let $M=\|row_k(v)\|$. Since $v$ is maximal, by the definition of Step~1 order $\|row_k(v)\|=\|row_k(\Delta)\|$,
so $M\geq 2$ by hypothesis (b) in Proposition~\ref{step1a}. Since we also assume that
$\Delta$ has small boundary, for any boundary vertex $w$ of $\Delta$ we have $\|row_k(w)\|=1$,
so $v$ must be interior.

Since $v$ is $k$-unimodular, it must have a column $c$ of depth $k$ whose $k^{\rm th}$ coordinate, call it $c_k$, is not divisible by $M$. Thus, $0< |c_k| < M$, so by definition $c$ is good. 
\end{proof}
Observation~\ref{obs:goodcolumn} will be crucial for the argument in Case~4 below.

\skv
{\it Case 1: $v$ has at least $C+k$ columns of depth between $1$ and $k-1$}. 

Take any gallery $\calG$ at $v$ of length $\min(3,\deg(v))$.
By the hypotheses in this case there exist distinct $i,j\not\in\supp(\calG)$ such that $col_i(v)$ and $col_j(v)$ have the
same depth $m$ with $1\leq m\leq k-1$, so that both entries $v_{m,i}$ and $v_{m,j}$ are nonzero.

Consider the commuting maps $C_{\calG}(R_{ij}^{\pm 1})$ and $C_{\calG}(R_{ji}^{\pm 1})$.
At least one of these 4 maps, call it $\psi$, has the following property: for every new vertex $w'$
there is a vertex $w$ of $\calG$ such that $w$ and $w'$ are incomparable based on the criteria (1)-(5) in the Step~1 pre-order,
$\|row_t^{\rm{ur}}(w')\|=\|row_t^{\rm{ur}}(w)\|$ for all $m<t\leq k$ and $\|row_m^{\rm{ur}}(w')\|<\|row_m^{\rm{ur}}(w)\|$
in the notations from criterion (6). Then $w'<w$ and thus $\psi$ is a reduction. It is also $k$-safe since
every new vertex is obtained from an old vertex using operation (b) from Lemma~\ref{lem:ksafe}.

Since the commuting maps $C_{\calG}(R_{ij}^{\pm 1})$ and $C_{\calG}(R_{ji}^{\pm 1})$ decrease the degree of $v$, after repeating this operation finitely many times we will eliminate $v$.
\skv

In Cases~2, 3 and 4 we will construct a reduction $\phi^*$ by lifting a suitable one-dimensional reduction $\phi$.

\skv
{\it Case 2: $v$ has at least $C+2$ maximal unrestricted columns.} 

We construct the one-dimensional reduction $\phi$ as in the maximal coordinate case (see \S~3,4). One can prove that $\phi^*$ is a reduction exactly as we proved that $\phi$ is a reduction in \S~3,4. As in Case~1, every new vertex is obtained from an old vertex using operation (b) from Lemma~\ref{lem:ksafe}, so $\phi^*$ is also $k$-safe.

\skv
{\it Case 3: $v$ has at least $C+1$ good columns.} The majority of work in this case will be devoted to constructing single-cell reductions.
Once this is done, we will describe a minor change in the definition of multiple-cell reductions (compared to the one-dimensional setting).

So let us fix a cell $\calF$ containing $v$. As in the one-dimensional setting, we will specify a small set $I$ (depending on the type of $\calF$) such that the replacement diagram of our reduction map is supported on $I$. We can permute the columns so that
the first column is maximal, the fourth column is good and $4\not\in supp(\calF)$ (the last condition can be arranged since $|supp(\calF)|\leq C$). 
In addition, we can also assume that the first and fourth columns of $v$ have positive $k^{\rm th}$ coordinate.

Also recall that $\calF_I$ and $v_I$ denote the $I$-traces of $\calF$ and $v$, respectively. 
\skv

First let us assume that $\calF$ has type other than 4GB, 5GA or 5GB. 
Similarly to Case~2, we construct the one-dimensional reduction $\phi$ as in the good coordinate case (see \S~3,4). By straightforward case-by-case verification, every new vertex $w'$ of $\phi^*$ satisfies the following 2 conditions:

\begin{itemize}
\item[(1)] $w'$ is obtained from some old vertex $w$ either as in Lemma~\ref{lem:goodeasyH}(a)
or by replacing a $0$-column by a $\Delta$-good column (a special case of the condition from Lemma~\ref{lem:zeasyH}(a));

\item[(2)] $w'$ is obtained from some old vertex $w_1$ (possibly different from $w$ in (1)) by adding/subtracting an unrestricted column
to/from another column (operation (b) from Lemma~\ref{lem:ksafe}).
\end{itemize}
Condition (1), Lemma~\ref{lem:goodeasyH}(b)~and~Lemma~\ref{lem:zeasyH}(a) imply that $\phi^*$ is a reduction while (2) and Lemma~\ref{lem:ksafe} imply that $\phi^*$ is $k$-safe.
\skv
{\it Type 4GB}. Next consider the case where $\calF$ has type 4GB. As in the one-dimensional setting, we set $I=\{1,2\}$
and assume that $v_I=(a,b)$ and the other three vertices of $\calF$ have $I$-traces $(a-b,b)$, $(a-c,b)$
and $(a-b-c,b)$. 

If $b_k\neq 0$, we define $\phi$ as in Figure~\ref{4GB}. The map $\phi^*$ is $k$-safe by Lemma~\ref{lem:ksafe}(b),
and one can prove that $\phi^*$ is a reduction exactly as in the one-dimensional setting. On the other hand, if
$b_k=0$, the first columns of all vertices of $\calF$ have non-negative $k^{\rm th}$ coordinate and hence
the commuting map $C(R_{41})$ is a reduction. This map is also $k$-safe by Lemma~\ref{lem:ksafe}(b).  

\skv
{\it Type 5GA}. As in the one-dimensional setting, we set $I=\{1,2\}$
and assume that $v_I=(a,b)$ and the other three vertices of $\calF$ have $I$-traces $(a+b,b)$, $(a+c,b)$
and $(a+b+c,b)$. Define $\phi$ as in Figure~\ref{5GA}. The map $\phi^*$ is $k$-safe by Lemma~\ref{lem:ksafe}(b),
so we only need to prove that $\phi^*$ is a reduction.

If either $b_k\neq 0$ or $b$ is unrestricted, we can argue as in the one-dimensional setting to show that $w<v$ for every new vertex $w$ of $\phi^*$ and hence $\phi^*$ is a reduction by Lemma~\ref{lem:easyreduction}. If $c_k\neq 0$ or $c$ is unrestricted, we can construct
a reduction by swapping the roles of the $2^{\rm nd}$ and $3^{\rm rd}$ columns.

Thus, it remains to consider the case where $b_k=c_k=0$ and $b$ and $c$ are both restricted. In this case the cell $\calF$
is purely maximal, so $h_{\Delta}(v)>0$. We will use Lemma~\ref{lem:enhancedreduction}. Set $\Delta_{new}=\phi^*(\Delta)$. 
Let us first verify condition (i). If $w$ is a new vertex of $\phi^*$, then $w<v$ except
when $w_I=(a+b,b+d)$ or $w_I=(a+b+c,b+d)$, in which case $l(w)\sim l(v)$. Suppose that $w$ is one of the latter two vertices.
The diagram $\Delta_{new}$ has exactly one cell $\calF_{5}$ of type $5$ containing $w$, and its vertices have $I$-traces
$(a+b,b+d)$, $(a+b+c,b+d), (a-d,b+d)$ and $(a+c-d,b+d)$. If $z$ is any of the last two vertices, then $z<w$,
so $\calF_{5}$ is not pre-maximal. Hence $h_{\Delta_{new}}(w)=0$, so $L_{\Delta_{new}}(w)<L_{\Delta}(v)$ and thus condition (i) 
holds. 

The map $\phi^*$ has 2 new cells of type $5$, one of which is $\calF_5$. We already checked condition~(ii) for $\calF_5$,
and verification of (ii) for the other cell is analogous. Thus $\phi^*$ is a reduction by Lemma~\ref{lem:enhancedreduction}.

\skv
{\it Type 5GB}. As in the one-dimensional setting, we set $I=\{1,2,3\}$ and assume that $v_I=(a,b,c)$ and the other three vertices of $\calF$ have $I$-traces $(a,b-a,c)$, $(a,b,c-a)$
and $(a,b-a,c-a)$. We consider 4 subcases.
\skv

{\it Subcase~1: $b$ is unrestricted}. In this case we define $\phi$ as in the one-dimensional setting. Every new vertex of $\phi^*$
can be obtained from an old vertex by adding/subtracting the unrestricted second column $b$ to/from other columns and
signed permutations of columns, so $\phi^*$ is $k$-safe by Lemma~\ref{lem:ksafe}. One can prove that 
$\phi^*$ is a reduction as in the one-dimensional setting.
\skv

{\it Subcase~2: $b$ is restricted and $b-a$ is unrestricted}. Note that in this subcase $a$ must also be restricted.
Again we define $\phi$ as in the one-dimensional setting. Up to a signed permutation of columns, $\phi^*$ has 3 new vertices whose $I$-traces are $(a-b,b,c)$, $(a-b,b,b+c-a)$ and $(a-b,b,c-a)$. The first vertex is obtained from the old vertex $(a,b,c,*)$ by replacing the restricted first column $a$ by $a-b$ (which is unrestricted by assumption). The third vertex is obtained from the old vertex $(a,b,c-a,*)$ in the same way. Finally, the second vertex is obtained from $(a,b,c,*)$ by replacing the restricted first column $a$ by $a-b$ and then subtracting the unrestricted first column $a-b$ from the third column. All of these operations come from the list
in Lemma~\ref{lem:ksafe}, so $\phi^*$ is $k$-safe. As in subcase~1, $\phi^*$ is a reduction by the same argument as in the one-dimensional setting.

\skv
{\it Subcase~3: $c$ or $c-a$ is unrestricted}. This subcase is reduced to subcases~1 and 2 by swapping the roles of the $2^{\rm nd}$ and $3^{\rm rd}$ columns.

{\it Subcase~4: $b,c,b-a$ and $c-a$ are all restricted}. In this case we define $\phi^*$ by a completely different diagram
given in Figure~\ref{5GBM}. Here $d$ denotes the common $4^{\rm th}$ column of the vertices of $\calF$.
By symmetry we can assume that either $b$ and $c$ are non-maximal or $b$ is maximal.

\input{figure5GBM.tex}

All new vertices with first column $a-d$ are smaller than $v$ by Lemma~\ref{lem:goodeasyH}(a). If $b$ is maximal, the same is true
for the remaining two new vertices (whose first 2 columns are $a$ and $b-d$) and hence $\phi^*$ is a reduction by Lemma~\ref{lem:goodeasyH}(b).
On the other hand, if neither $b$ nor $c$ is maximal, it is straightforward to check that the cell $\calF$ is purely maximal, and
we can prove that $\phi^*$ is a reduction similarly to type 5GA.
\skv
\paragraph{\bf Multiple-cell reductions.} We now discuss a minor change in the construction
of double-cell and triple-cell reductions. It only affects type GB (the good coordinate case for the groups
$\IAC_{n,d}$). Recall that $v$ is a fixed maximal vertex with at least $C+1$ good columns and the first column of $v$ is maximal with
positive $k^{\rm th}$ coordinate.

As in the one-dimensional setting, we can get to the stage where all the remaining cells containing $v$ have type $4$, $7$ or are not essential. However, at the next step (whose one-dimensional counterpart was described in \S~\ref{sec:4GBonedim}) we need some extra care since our definition of the single-cell reductions for cells of type 4GB in the higher-dimensional setting depends on additional data. 

It will convenient to introduce the following definition. 

\begin{Definition}\rm
\label{def:problematicedge} 
Let $e$ be an edge containing $v$ with $l(e)=R_{i1}$ or $L_{i1}$ for some $i$. We will say that $e$ is {\it problematic for $v$}
if the $k^{\rm th}$ coordinate of the $i^{\rm th}$ column of $v$ is $0$.
\end{Definition}

If there are no problematic edges for $v$, we can continue the process and eventually eliminate $v$ as in the one-dimensional setting. On the other hand, if $e$ is a problematic edge and $\calF$ and $\calF'$ are the two cells containing $v$, then the first columns of all the vertices of 
$\calG=\calF\cup\calF'$ have non-negative first coordinate. Hence if $j$ is any good index not in $\supp(\calG)$, then 
$C_{\calG}(R_{j1}^{\pm 1})$ is a double-cell reduction (for a suitable choice of sign) which replaces $e$ by a non-problematic edge.
Applying this operation finitely many times, we can eliminate all problematic edges and then proceed as before.
This completes the proof in Case~3.
\skv

{\it Case 4: $v$ has at least $C+1$ $0$-columns.} 

Here it is crucial that $v$ has at least one good column (see Observation~\ref{obs:goodcolumn} above). 
Unlike Case~3, we will only describe single-cell reductions, as there are no non-trivial changes in the construction of multiple-cell reductions.  
Thus we fix a cell $\calF$ containing $v$. As in the one-dimensional seeting, we can assume that 
the first column is good with positive $k^{\rm th}$ coordinate, the fourth column is zero and $4\not\in supp(\calF)$.

\skv
Before proceeding, we introduce a technical definition.

\begin{Definition}\label{def:ZSmap}
\rm Let $\psi$ be a diagram map and $1\leq i\leq n$. We will say that $\psi$ is a $ZS_i$-map (where ZS stands for
{\it zero substitution}) if for every new vertex $w'$ of $\psi$ there is an old vertex $w$ such that
$col_i(w)=0$ and $w'$ is obtained from $w$ by replacing $0$ by $c$ in the $i^{\rm th}$ column where $c$ is a column of some old vertex of $\psi$.
If $\psi$ is a $ZS_i$ map, the set of the $i^{\rm th}$ columns of the new vertices will be called the substitution set of $\psi$ 
and denoted by $Sub_i(\psi)$.
\end{Definition}

Here is a simple criterion for a $ZS_i$ map to be $k$-safe and to be a reduction.

\begin{Observation}
\label{obs:ZS}
Let $\psi$ be a $ZS_i$ map.
\begin{itemize}
\item[(a)] If every element of $Sub_i(\psi)$ is unrestricted, then $\psi$ is $k$-safe.
\item[(b)] If every element of $Sub_i(\psi)$ is good, then $\psi$ is a $k$-safe reduction. 
\end{itemize}
\end{Observation}
\begin{proof} (a) holds by Lemma~\ref{lem:ksafe} as the hypothesis of (a) implies that every new vertex of $\psi$
is obtained from an old vertex by an operation of type (d) in the statement of Lemma~\ref{lem:ksafe}.

(b) $\psi$ is $k$-safe by (a) and a reduction by Lemma~\ref{lem:zeasyH}(a). 
\end{proof}
\begin{Remark}\rm Even if neither part of Observation~\ref{obs:ZS}
is applicable to a $ZS_i$ map $\psi$, one can still use the same idea
to shorten verification of the fact that $\psi$ is a reduction
or that $\psi$ is $k$-safe. In particular, if $\psi:\Delta\to \Delta_1$ is a $ZS_i$ map and $v$ is a maximal
vertex of $\Delta$, to prove that $\psi$ is a reduction it suffices to check that $w<v$ for every new vertex $w$ which has 
a restricted $i^{\rm th}$ column. Likewise, to prove that $\psi$ is $k$-safe it suffices to check
that every new vertex with restricted $i^{\rm th}$ column can be obtained
from an old vertex using operations from Lemma~\ref{lem:ksafe}.
\end{Remark}

In the one-dimensional setting, the single-cell reduction $\phi$  we used in the zero coordinate case was a $ZS_4$ map with the exception
of type 5ZA. Moreover, we were able to prove that $\phi$ is a reduction using Observation~\ref{obs:ZS}(b) (even though we did not formally
refer to the latter). 

Let us resume the proof of Proposition~\ref{step1a}. Fix a cell type, and let $\phi$ be the reduction used for that type in the one-dimensional setting, and as before, let $\phi^*$ be the lift of $\phi$. Apart from type 5ZA, $\phi$ is a 
$ZS_4$ map, whence $\phi^*$ is also a $ZS_4$ map. A straightforward verification shows that
for most cell types, all elements of $Sub_4(\phi^*)$ are good (and this can be proved exactly as in the one-dimensional setting), so that
$\phi^*$ is a $k$-safe reduction by Observation~\ref{obs:ZS}(b). The
only exceptions are the types 1ZA, 2ZA, 2ZB, 3ZB and 5ZB (in addition to type 5ZA we excluded earlier), and these types will be treated
separately below.

For types 1ZA, 2ZA, 2ZB, 3ZB we will use the same map $\phi^*$, but 
to prove that $\phi^*$ is a $k$-safe reduction we will use the remark
following Observation~\ref{obs:ZS} rather than Observation~\ref{obs:ZS} itself. For type 5ZB, we will slightly modify the map used in the one-dimensional setting, but the rest of the argument will still be quite similar. 

Finally, type 5ZA does not require any special treatment,
and we can use the same $\phi$ is in the one-dimensional setting. Even though $\phi^*$ is not a $ZS_4$ map in this case, one can prove that
$\phi^*$ is a $k$-safe reduction similarly to the one-dimensional setting.
\end{proof}

\centerline {\bf Exceptional types (Case 4)}
\skv

In the discussion below by saying that an element of $\dbZ^d$ is good (resp. bad) we will mean that it is $\Delta$-good (resp. $\Delta$-bad).
Also recall that $v$ denotes the (chosen) maximal vertex of $\mathcal F$ and we assume that $col_1(v)$ is good while $col_4(v)=0$.

\skv For types 1ZA, 2ZA, 2ZB and 3ZB our definition of the reduction map
in the one-dimensional setting was dependent on whether a particular coordinate of $v$ is good or bad (namely, the $3^{\rm rd}$ coordinate for types 1ZA, 2ZA and 3ZB and the $2^{\rm nd}$ coordinate for type 2ZB), and in the
case where that coordinate is good we argued that a certain commuting map is a reduction. An analogous argument shows that in the 
higher-dimensional setting
the same commuting map is a $k$-safe reduction provided the respective column of $v$ is good, so from now on we will assume that 
\begin{itemize}
\item for types 1ZA, 2ZA and 3ZB the $3^{\rm rd}$ column of $v$ is bad and 
\item for type 2ZB the $2^{\rm nd}$ column of $v$ is bad. 
\end{itemize} 

\skv
{\it Type 1ZA, $I=\{1,2,4\}$}. As in Figure~\ref{1ZA}, the cell $\calF$ we are replacing (that is, the domain of $\phi^*$) has 5 vertices whose $I$-traces are $(a,b,0)$, $(a+b,b,0)$, $(a+b,b+c,0)$, $(a+b+c,b+c,0)$ and $(a,b+c,0)$.
The substitution set of $\phi^*$ is $\{a,a+b,a+b+c\}$. 

Suppose first that $v$ is one of the three vertices with second column $b+c$. Since these three vertices
only differ in the first column and $v$ has a good first column, the other two vertices from this triple must also have a good first column 
by Lemma~\ref{lem:zeasyH}(b). Therefore, $a$, $a+b$ and $a+b+c$ are all good, and we are done by Observation~\ref{obs:ZS}.

Now consider the remaining cases where $v_I=(a,b,0)$ or $(a+b,b,0)$. Arguing as in the previous paragraph, $a$ and $a+b$ are both good.
Also recall that $c$ is bad by the initial assumption. If $c$ and unrestricted, then $a+b+c$ is also good, and we are done, so
let us assume that $c$ is restricted, in which case $a+b+c$ and $b+c$ are also restricted. In this case $\phi^*$ has 2 new vertices with a new
restricted column (namely restricted $4^{\rm th}$ column): $u_1=(a+b+c,b+c,c,a+b+c,*)$ and $u_2=(a+b,b+c,c,a+b+c,*)$,
and by the remark following Observation~\ref{obs:ZS} to finish the proof
it suffices to check that
\begin{itemize}
\item[(i)] both $u_1$ and $u_2$ are obtained from an old vertex
using operations from Lemma~\ref{lem:ksafe} (this will prove
that $\phi^*$ is $k$-safe);
\item[(ii)] $u_1<v$ and $u_2<v$ (this will prove that
$\phi^*$ is a reduction).
\end{itemize}
To prove (i) we just note that $u_1$ is obtained from the old
vertex $u_0=(a+b+c,b+c,c,0,*)$ by replacing $0$ by $a+b+c$ in the
$4^{\rm th}$ column (operation (c) in Lemma~\ref{lem:ksafe}) 
while $u_2$ is obtained from the same vertex $u_0$ by replacing
$0$ by the good vector $a+b$ in the
$4^{\rm th}$ column (operation (d))
followed by a permutation of coordinates (operation (a)). 

Let us now prove (ii). Recall that $a+b+c$ and $b+c$ are
both restricted. In particular $(a,b+c,0)$ and $(a+b,b+c,0)$ are both vertices of $\calF_I$ with a restricted column among the first two. Since the first two columns of $v$ are unrestricted and the first one is good while $v$ is maximal, the only possibility is that the second column of $v$ (which we know is equal to $b$) is maximal. Thus, $b_k=\pm M$, and without loss of generality we can assume that $b_k=M$.

Since $(u_1)_I=(a+b+c,b+c,a+b+c)$ and $(u_2)_I=(a+b,b+c,a+b+c)$, to prove (ii) it will be enough to show that none of the columns $a+b$, $b+c$ and $a+b+c$ is maximal (since we already showed that
the second column of $v$ is maximal). We already know that $a+b$ is good (and hence not maximal).
If $b+c$ is maximal, then $v_I<(a+b+c,b+c,0)$, a contradiction. Finally suppose that $a+b+c$ is maximal. Since $b$ is maximal
and $a$ is good (in particular not maximal), $(a+b+c)_k=M$ whence $c_k=-a_k$. Since $(a+b)_k=M+a_k$ and
$(b+c)_k=b_k+c_k=M-a_k$ and $(a+b)_k,(b+c)_k\leq M$, we must have $a_k=0$, contrary to the assumption that $a$ is good.
\skv

For the types 2ZA, 2ZB and 3ZB one can prove that the map $\phi^*$
as defined below is $k$-safe similarly to type 1ZA. Thus, for those
types we will only explain why $\phi^*$ is a reduction.

\skv
{\it Type 2ZA}, $I=\{1,2,4\}.$ The vertices of $\calF_I$ are $(a,b,0)$, $(a+c,b,0)$, $(a,a+b,0)$, $(a+c,a+b,0)$ and $(a+c,a+b+c,0)$,
and we assume that $c$ is bad. We define $\phi$ is in Figure~\ref{2ZA}, so that $Sub(\phi^*)=\{a,a+c\}$. 

If $v_I$ equals any of the first 4 vertices in the above list, arguing as in type 1ZA, we conclude that $a$ and $a+c$ are both good, and we are done, so assume from now on that $v_I=(a+c,a+b+c,0)$. This means that $a+c$ is good.
Since $c$ is bad, if it is also unrestricted,
then $a=(a+c)-c$ is good, and we are done, so let us assume that $c$ is restricted, in which case $a$ is also restricted. 

The only new vertices which have $a$ as their $4^{\rm th}$ column 
have $I$-traces $(a,b,a)$, $(a,a+b,a)$ and $(a+c,b,a)$. To prove that these vertices are $<v$,
it suffices to show that $v_I$ has a maximal column while none of $a,b$ and $a+b$ is maximal (we already
know that $a+c$ is good and hence not maximal).

Since $a$ is restricted and $a+c$ is unrestricted, $(a,b,0)$ has more restricted columns than $v_I=(a+c,a+b+c,0)$.
Since $v$ is maximal and has good first column, this is only possible if $a$ and $b$ are not maximal and $a+b+c$ is maximal.
And if $a+b$ is maximal, we have $v_I=(a+c,a+b+c,0)<(a,a+b,0)$, a contradiction. Thus $a,b$ and $a+b$ are all non-maximal and $a+b+c=col_2(v)$ is maximal, as desired.
\skv
\skv
{\it Type 2ZB}, $I=\{1,3,4\}$. The vertices of $\calF_I$ are $(a,c,0)$, $(a+b,c,0)$, $(a,a+b+c,0)$, $(a+b,a+b+c,0)$ and $(a,b+c,0)$,
and this time we assume that $b$ is bad. We will use the diagram map $\phi^*$ where $\phi$ is given by Figure~\ref{2ZB}. Thus,
$Sub(\phi^*)=\{a,a+b\}$.

If $v_I$ is equal to any vertex other than $(a,b+c,0)$, then $a$ and $a+b$ are good by Lemma~\ref{lem:zeasyH}(b), and we are done. Let us now assume that $v_I=(a,b+c,0)$, in which case $a$ is still good.
Since $b$ is bad, if it is also unrestricted, then $a+b$ is good, so again we are done.
Thus, we can assume that $b$ is restricted, which means that $a+b$ is also restricted. 
The only new vertices which have $a+b$ as their $4^{\rm th}$ column have $I$-traces $(a+b,c,a+b)$ and 
$(a+b,a+b+c,a+b)$, and it suffices to show that $v_I=(a,b+c,0)$ has a maximal column while
none of $a+b,c$ and $a+b+c$ is maximal.

Note that $(a+b,a+b+c,0)$ has more restricted columns than $v_I=(a,b+c,0)$ (since $a+b$ is restricted,
while $a$ is not and hence $a+b+c$ and $b+c$ are both restricted or both unrestricted). 
Since $(a+b,a+b+c,0)\leq (a,b+c,0)$, this implies that $(a,b+c,0)$ has more maximal columns
than $(a+b,a+b+c,0)$. And since $a$ is good (hence non-maximal), we conclude that $b+c$ is maximal and $a+b$ and $a+b+c$ are not maximal. 
If $c$ is not maximal, we are done, so assume
that $c$ is maximal. Since $b+c$ is also maximal and $|b_k|\leq M$, we have $(b+c)_k=c_k$ whence $b_k=0$. But then $(a,b+c)<(a+b,c)$ since
these two vertices have the same number of maximal columns while $(a+b,c)$ has at least as many restricted non-maximal columns as $(a,b+c)$
and fewer good columns. This inequality contradicts the assumption that $v_I=(a,b+c,0)$.

\skv
{\it Type 3ZB,} $I=\{1,2,4\}$. The vertices of $\calF_I$ are $(a-c,b,0)$, $(a,b,0)$, $(a+b,b,0)$, $(a+b,b+c,0)$ and $(a-c,b+c,0)$, and we assume that
$c$ is bad. We will use the diagram map $\phi^*$ where $\phi$ is given by Figure~\ref{3ZB}. Thus, $Sub(\phi^*)=\{a,a+b,a-c\}$.

If $v_I=(*,b,0)$, then $a,a+b,a-c$ are all good by Lemma~\ref{lem:zeasyH}(b), and we are done. So let us assume that $v_I=(*,b+c,0)$, in which
case $a+b$ and $a-c$ are still good. Since $c$ is bad, if it is also unrestricted, then $a=(a-c)+c$ is good, and we are done.

Thus we can assume that $c$ is restricted. Since $a+b$ and $a-c$ are good, $b+c=(a+b)-(a-c)$ is unrestricted; on the other hand, $a=(a-c)+c$ and $b=(b+c)-c$ must both be restricted. Since $(a,b,0)$ cannot be strictly larger than $v_I$ and
$v_I=(x,b+c,0)$ where $x\in\{a+b,a-c\}$ is good,
$b+c$ must be maximal while $a$ and $b$ are not maximal. We also know that $a+b$ is good (and hence
non-maximal). Hence $(a+b,b,a)< v_I$ and $(a,b,a) < v_I$ (as $col_2(v)=b+c$ is maximal). Since $(a+b,b,a)$ and $(a,b,a)$
are precisely the $I$-traces of the new vertices which have $a$
as their $4^{\rm th}$ column, we are done.
\skv

\skv
{\it Type 5ZB,} $I=\{2,3,4\}$.  
 Unlike the one-dimensional setting, we will use one of the two distinct maps shown in Figure~\ref{5ZBM}, depending on the values of $a,b$ and $c$. The first map is the lift of the map in Figure~\ref{5ZB} that we used for type 5ZB in the one-dimensional setting.

\input{figure5ZBM.tex}

As in the one-dimensional setting, we can assume that $v=(a,b,c,0,*)$, so in particular $a$ is good. 

Suppose first that $c$ is unrestricted,
so $a-c$ is also unrestricted. If $a-c$ is bad, then
$c=a-(a-c)$ is good, whence $(a,b,c,0,*)<(a,b,c-a,0,*)$, a contradiction.
Thus, $a-c$ is good, whence the first map in Figure~\ref{5ZBM} is a 
$k$-safe reduction by Observation~\ref{obs:ZS}(b).

If $b$ is unrestricted, we can apply the same argument swapping the roles of the second and third columns.
Suppose now that $b$ and $c$ are both restricted (so $b-a$ and $c-a$ are also restricted). By symmetry, we can also assume that
 either $b$ and $c$ are both non-maximal or $c$ is maximal.
We will show that the second map in Figure~\ref{5ZBM} (call it $\phi^*$) is a reduction. This map is also $k$-safe by
Observation~\ref{obs:ZS}(a). 

We have $Sub_4(\phi^*)=\{a,0\}$. Recall that $a$ is good, and the only new vertices of $\phi^*$ whose fourth column
is $0$ have $I$-traces $(b,c-a,0)$ and $(b-a,c-a,0)$. If $c$ is maximal, both of these vertices are $<v$ by Lemma~\ref{lem:goodeasyH}(a),
and we are done. And if $b$ and $c$ are both non-maximal (and restricted by an earlier assumption), the cell $\calF$ is purely maximal,
in which case $\phi^*$ is a reduction by the same argument as in type~5GA.

\subsection{Step~2}
The main result of Step~2 is the following theorem:

\begin{Theorem}[Step~2]
\label{step2}
There exists a super-Artinian order (called Step~2 order) with the following property.
For any diagram $\Omega'$ satisfying the conclusion of Step~1, there exists a $k$-safe reduction 
$\Omega'\to \Omega''$ (with respect to the Step~2 order) such that $\Omega''$ is $(k-1)$-unimodular.
\end{Theorem}

Before defining the Step~2 order, let us introduce a total order $\prec$ on the set of subgroups of $\dbZ^m$  
(below we will use this order for different values of $m$). 
The precise definition will not be important for our purposes, and the only properties relevant for the proof
are
\begin{itemize}
\item[(i)] $\prec$ refines the order by inclusion;
\item[(ii)] $\prec$ is Noetherian (that is, there are no infinite strictly ascending chains).
\end{itemize}

First, for each $t\in\dbN$ choose a total order on $\dbZ^t$ which refines the $\ell^{1}$-norm. Since $\ell^{1}$-balls are finite,
this order is automatically Artinian. 

Now let $A$ be a nonzero subgroup of $\dbZ^m$ for some $m$, and let $r=rk(A)$. 
Viewing $r$-tuples of elements of $\dbZ^m$ as elements of $\dbZ^{mr}$, define
the {\it norm} of $A$, denoted by $N(A)$, to be the smallest $r$-tuple in $A$ consisting of linearly independent vectors (with respect to the above order on $\dbZ^t$ for $t=mr$). Given two subgroups $A$ and $A'$, we set $A\prec A'$ if either $rk(A)<rk(A')$
or $rk(A)=rk(A')$ and $N(A')<N(A)$. It remains to define the order on the set of subgroups with fixed rank and fixed norm.

 Fix $r\in\dbN$
and an $r$-tuple $S$ consisting of linearly independent vectors. Then any subgroup of rank $r$ and norm $S$ contains $\dbZ S$ (the $\dbZ$-span
of $S$) and is contained in $\dbQ S \cap \dbZ^m$. Since $\dbZ S$ is a finite index subgroup of $\dbQ S \cap \dbZ^m$, the subset of such subgroups
is finite, and we choose an arbitrary order on this subset refining the inclusion order. We have now defined a total order on the set of
subgroups of $\dbZ^m$, and it is clear from the construction that it satisfies (i) and (ii) above.

\skv
Next we introduce the notions of the active and semi-active subgroups of a vertex $v$, denoted by $A(v)$ and $SA(v)$, respectively:

\begin{Definition}\rm Let $v\in Mat_{d\times n}(\dbZ)$. The subgroup of $\dbZ^{k-1}$ generated by the $(k-1)$-columns of $v$
will be called the {\it active subgroup} of $v$ and denoted by $A(v)$.
\end{Definition}
\begin{Remark}\rm
A diagram $\Omega$ is $(k-1)$-unimodular if and only if $A(v)=\dbZ^{k-1}$ for every $v\in V(\Omega)$.
\end{Remark}

We now define the {\it semi-active subgroup} of $v$ denoted by $SA(v)$.

\begin{Definition}\rm 
Consider the equivalence relation $\sim$ on $\dbZ^k$ where $y\sim z$ $\iff $ $y-z\in A(v)$. 
\begin{itemize}
\item[(a)] An equivalence class of $k$-vectors will be called {\it frequent} if it does not lie in $A(v)$ (that is, its elements
have depth exactly $k$) and
$v$ contains at least $C+1$ columns from that class (these columns need not
be distinct as elements of $\dbZ^d$).
\item[(b)] The {\it semi-active subgroup} $SA(v)$ is the subgroup generated by $A(v)$ and all columns from frequent classes (clearly, it suffices to take just one column from each class). Thus, $A(v)\subseteq SA(v)\subseteq \dbZ^k$.
\end{itemize}
\end{Definition} 
 
As in Step~1, the new order will take into account the number of restricted columns. Good columns will not play any role in this step. Indeed, 
by assumption all vertices of $\Omega'$ (the diagram obtained at the end of Step~1) have $k^{\rm th}$ row
of norm $1$ and thus cannot have any good columns. Instead we will be tracking the number of {\it helpful} columns:
 
\begin{Definition}\rm Let $v\in Mat_{d\times n}(\dbZ)$ and $c\in\dbZ^d$. We will say that
\begin{itemize}
\item[(*)] $c$ is {\it $v$-helpful} if $c$ is unrestricted and $c\not\in SA(v)$;
\end{itemize}
If $c$ is a column of $v$ and $c$ is $v$-helpful, we will say that $c$ is a helpful column of $v$.   
\end{Definition}

We are now ready to define the new order on the vertices. As before, we compare two vertices by successively applying the following criteria
(and stop as soon as one of the criteria is applicable).
\skv

\paragraph{\bf Step 2 order:}

Similarly to the Step~1 order, we first define Step~2 pre-order based on the vertex labels:
\begin{itemize}
\item[(1)] The vertex with the larger active subgroup is smaller.
\item[(2)] The vertex with the larger semi-active subgroup is smaller.
\item[(3)] The vertex with the larger number of helpful columns is smaller
\item[(4)] The vertex with the larger number of $(k-1)$-columns is smaller. 
\item[(5)] The vertex with the smaller number of restricted columns is smaller. 
\item[(6)] The vertex $v$ with the smaller value of the vector 
$$(\|row_k(v^{\rm ur})\|_1,\|row_{k-1}(v^{\rm ur})\|_1,\ldots,\|row_{1}(v^{\rm ur})\|_1)$$
is smaller (with respect to the lexicographical order). Here $v^{\rm ur}$ is defined as in Step~1 order. 
\end{itemize}
We now define Step~2 order in terms of this pre-order exactly as in Step~1, except that the notions of a pre-maximal vertex and a purely
maximal cell are now taken with respect to the Step~2 pre-order. Unlike Step~1, where extended labels were used to treat cells of types
5GA, 5GB and 5ZB, in Step~2 extended labels will only be needed for type 5ZB.

\skv
Our goal is to prove the following proposition:

\begin{Proposition}
\label{step2b}
Let $\Lambda$ be a diagram which is $k$-unimodular, but
not $(k-1)$-unimodular. Also assume that $def_k(\Lambda)\leq n-(C^2+3C+k+3)$ and $\|row_k(\Lambda)\|=1$.
Then for any vertex $v$ of $\Lambda$ which is maximal relative to the Step~2 order there exists a $k$-safe reduction $\Lambda\to \Lambda'$ which eliminates $v$.
\end{Proposition}

Before proving Proposition~\ref{step2b}, let us deduce Theorem~\ref{step2} from it.

\begin{proof}[Proof of Theorem~\ref{step2}]

By assumption, $\Omega'$ satisfies the conclusion of 
Theorem~\ref{thm:step1}, so $\|row_k(\Omega')\|=1$
and $def_k(\Omega')\leq d-k+2C$. 
Since $d\leq n-(C^2+6C+3)$, we have 
$def_k(\Omega')\leq n-(C^2+4C+k+3)$. Thus either $\Omega'$ is $(k-1)$-unimodular and there is nothing to prove
or Proposition~\ref{step2b} is applicable to $\Lambda=\Omega'$. Let us proceed with the latter case.

Let $\Lambda_1$ be the diagram obtained by applying
Proposition~\ref{step2b} to $\Omega'$, and let $\Lambda_2$ be obtained by
applying Proposition~\ref{step1b} to $\Lambda_1$ (it is possible that
$\Lambda_2=\Lambda_1$), so that $def_{k}(\Lambda_2)\leq wdef_{k}(\Lambda_2)+C$.
Proposition~\ref{step1b} only asserts that the map $\Lambda_1\to\Lambda_2$ is a reduction 
relative to the Step~1 order, but it follows immediately from the proof that
this map is also a reduction relative to the Step~2 order, so
the composite map $\Omega'\to \Lambda_2$ is also a reduction relative to the Step~2 order.

By construction, the composite map $\Omega'\to \Lambda_2$ is $k$-safe, so we have 
$\|row_k(\Lambda_2)\|\leq \|row_k(\Omega')\|=1$ (and hence $\|row_k(\Lambda_2)\|=1$
since a unimodular diagram cannot have a zero row) and
$$wdef_k(\Lambda_2)\leq wdef_k(\Omega')\leq def_k(\Omega')\leq n-(C^2+4C+k+3),$$
whence $def_{k}(\Lambda_2)\leq n-(C^2+3C+k+3)$. Thus, either 
$\Lambda_2$ is $(k-1)$-unimodular (and we are done) or we can apply 
Proposition~\ref{step2b} to $\Lambda=\Lambda_2$ and keep going.

Since the Step~2 order is Artinian, by Lemma~\ref{obs:Artinian} 
the process will terminate after finitely many steps, that is, we will obtain
a $(k-1)$-unimodular diagram, call it $\Omega''$.

The obtained map $\Omega'\to\Omega''$ is a composition of $k$-safe reductions
(relative to the Step~2 order) and thus is itself a $k$-safe reduction, as desired.
\end{proof}

Next we state and prove suitable counterparts of Lemmas~\ref{lem:goodeasyH}~and~\ref{lem:zeasyH} which will be applicable to the Step~2 order.
It will be convenient to introduce one more technical definition.

\begin{Definition}\rm
Let $v,w\in \dbZ^d$. We will say that $w$ is {\it $v$-soft} if either $A(w)>A(v)$ or $A(w)=A(v)$ but $S(w)>S(v)$.  
\end{Definition}
Equivalently, $w$ is $v$-soft if $w<v$ (relative to the Step~2 pre-order) and the inequality can be checked using one of the first 2 criteria in the Step~2 pre-order.

\begin{Lemma}
\label{lem:zeasyStep2a}
Let $\Delta$ be a diagram, $v$ a maximal vertex of $\Delta$ and $J$ a subset of $\{1,\ldots,n\}$.
Let $w$ be any vertex of $\Delta$ such that $col_i(w)=col_i(v)$ for all $i\not\in J$.
The following hold:
\begin{itemize}
\item[(a)] If $col_j(w)$ is restricted for all $j\in J$, then $coi_j(v)$ is restricted for all $j\in J$.
\item[(b)] Assume that $col_j(v)$ is helpful for all $j\in J$. Then either $w$ is $v$-soft or $S(v)=S(w)$
and $col_j(w)$ is $v$-helpful for all $j\in J$.
\end{itemize}
\end{Lemma}
\begin{Remark}\rm We will primarily apply Lemma~\ref{lem:zeasyStep2a} in the case $|J|=1$.
\end{Remark}
\begin{proof} (a) Let $j\in J$. Since $col_j(w)$ is restricted, it cannot contribute to the active or semi-active subgroups of $w$,
the number of helpful columns or the number of $(k-1)$-columns. Since $col_i(w)=col_i(v)$ for all $i\not\in J$, 
we cannot prove that $w>v$ using the first 4 criteria
of the Step~2 order. If in addition $col_j(v)$ is unrestricted for some $j\in J$, then $w$ has more restricted columns than $v$ and thus $w>v$ (by criterion~5 in the Step~2 order), a contradiction.

(b) Suppose $w$ is not $v$-soft. Since $v$ is maximal, we must have $A(w)=A(v)$ and $S(w)=S(v)$. 
If $col_j(w)$ is not $v$-helpful for some $j\in J$, it is also not $w$-helpful (as $S(w)=S(v)$), so $w$ has fewer helpful columns than $v$, and
hence $w>v$ (by criterion~3 in the Step~2 order), a contradiction.
\end{proof}

\begin{Lemma}
\label{lem:zeasyStep2b}
Let $\psi:\Delta\to \Delta_1$ be a diagram map and $v$ be a maximal vertex of $\Delta$. The following hold:
\begin{itemize}
\item[(a)] Let $w'$ be a new vertex of $\psi$. Suppose that $w'$ is obtained from an old vertex $w$ by replacing 
a zero column by some $c\in \dbZ^d$ where either 
\begin{itemize}
\item[(i)] $w$ is $v$-soft or 
\item[(ii)] $c$ is $v$-helpful. 
\end{itemize}
Then $w'<v$. 
\item[(b)] Suppose that $\psi$ is a $ZS_i$ map for some $i$ (see Definition~\ref{def:ZSmap}) and all elements of $Sub_i(\psi)$ are $v$-helpful.
Then $\psi$ is a reduction.
\end{itemize}
\end{Lemma}
\begin{proof} 
(a) Since a zero column does not contribute to the active or semi-active subgroups, we have
$A(w')\geq A(w)$ and $S(w')\geq S(w)$. If one of these inequalities is strict or if $w$ is $v$-soft, then $w'$ is $v$-soft and hence $w'<v$.
Otherwise, $c$ is $v$-helpful and we have $A(w')=A(w)=A(v)$ and $S(w')=S(w)=S(v)$, whence $c$ is also $w'$-helpful. Thus
$w'$ has more helpful columns than $w$, so again $w'<w$ and hence $w'<v$.

(b) follows directly from (a).
\end{proof}

We are now ready to prove Proposition~\ref{step2b}.

\begin{proof}[Proof of Proposition~\ref{step2b}] Since $\Lambda$ is not $(k-1)$-unimodular,
$A(w)\neq \dbZ^{k-1}$ for some $w\in V(\Lambda)$. Since $v$ is a maximal vertex of $\Lambda$,
by definition of the Step~2 order we must have $A(v)\neq \dbZ^{k-1}$. 

We will consider $4$ cases. Lemmas~\ref{lem:zeasyStep2a}~and~\ref{lem:zeasyStep2b} will only be needed in the most technically demanding 
Case~4. Recall that $C$ is a fixed constant with the property that for any gallery $\Delta$ of length $\leq 3$ we have $|\supp(\Delta)|\leq C$.
\skv

{\it Case 1: $SA(v)=\dbZ^k$. } Since $A(v)\neq \dbZ^{k-1}$ and $SA(v)$ is generated by $A(v)$ and representatives of frequent classes of $v$, there must be at least two frequent classes. Recall that each frequent class contains at least $C+1$ columns. This means that for any gallery $\calG$ at $v$ of length $\min(3,\deg(v))$, we can find distinct indices $i,j\not\in supp(\calG)$ such that $col_i(v)$ and $col_j(v)$ are representatives of distinct frequent classes of $v$. 

Recall that $\|row_k(v)\|=1$ by assumption. Both $col_i(v)$ and $col_j(v)$ have depth exactly $k$ (since
they lie in frequent classes), so $v_{ki}$ and $v_{kj}$ (the $(k,i)$ and $(k,j)$ entries of $v$) are both equal to $\pm 1$. Multiplying $col_i(v)$ or $col_j(v)$ by $-1$ if needed, we can assume that $v_{ki}=v_{kj}$. 

Let $\phi=C_{\calG}(R_{ij})$ or $C_{\calG}(R_{ji}^{-1})$  depending on whether 
$H=\widetilde{\IAR_{n,d}}$ or $\widetilde{\IAC_{n,d}}$. Then every new vertex of $\phi$ has
$col_i(v)-col_j(v)$ as its $i^{\rm th}$ column (and coincides with one of the old vertices in the remaining columns). By construction, $col_i(v)-col_j(v)$ is a $(k-1)$-column which does not lie in $A(v)$, so
all new vertices of $\phi$ have active subgroup larger than $A(v)$ and hence $\phi$ is a reduction
(relative to Step~2 order). It is also routine to check that $\phi$ is $k$-safe.
Since $\phi$ also decreases the degree of $v$, we can eliminate $v$ repeating this operation finitely many
times. This completes the proof in case~1.

\skv
Let us now assume that $SA(v)\neq \dbZ^k$. Since $v$ is $k$-unimodular (as $\Lambda$ is $k$-unimodular), this means that $v$ has at least one helpful column. Since $def_k(v)\leq n-(C^2+3C+k+3)$, one of the following must hold:
\begin{itemize}
\item[(ii)] $v$ has at least $C+k$ columns of depth between $1$ and $k-1$;
\item[(iii)] $v$ has at least $C^2+C+4$ columns of depth $k$;
\item[(iv)] $v$ has at least $C+1$ zero columns.
\end{itemize}
We consider these $3$ cases separately.
\skv

{\it Case 2: $SA(v)\neq \dbZ^k$ and $v$ has at least $C+k$ columns of depth between $1$ and $k-1$}. We will not make any use of the condition $SA(v)\neq \dbZ^k$ in this case. As in Case~1, take any gallery $\calG$ at $v$ of length $\min(3,\deg(v))$.
By the hypotheses in this case, there exist distinct $i,j\not\in\supp(G)$ such that $col_i(v)$ and $col_j(v)$ have the
same depth $m$ with $1\leq m\leq k-1$. 

Let $\phi$ be one of the $4$ commuting maps $C_{\calG}(R_{ij}^{\pm 1})$ or $C_{\calG}(R_{ji}^{\pm 1})$.
Every new vertex $w'$ of $\phi$ is obtained from an old vertex $w$ by replacing the $i^{\rm th}$ or $j^{\rm th}$ column
by $col_i(w)\pm col_j(w)$ or $col_j(w)\pm col_i(w)$. In all cases $w$ and $w'$ are incomparable based on criteria (1)-(5) in Step~2 order, and for some
$\phi\in \{C_{\calG}(R_{ij}^{\pm 1}),C_{\calG}(R_{ji}^{\pm 1})\}$ (chosen independently of $w$) we have 
$\|row_m((w')^{\rm ur})\|_1<\|row_m(w^{\rm ur})\|_1$,
whence $w'<w$, so $\phi$ is a reduction, and again it is straightforward to check that
$\phi$ is $k$-safe. Since $\phi$ decreases the degree of $v$, we are done as in Case~1.
\skv

{\it Case 3: $SA(v)\neq \dbZ^k$ and $v$ has at least $C^2+C+4$ columns of depth (exactly) $k$}. If there are at least $2$ frequent classes or at least $C+2$ distinct classes (mod $A(v)$) of columns of depth $k$, there exist distinct indices $i,j\not\in\supp(\calG)$
such that $col_i(v)$ and $col_j(v)$ both have depth $k$ and belong to distinct classes, and we can argue exactly as in Case~1.

Thus, we can assume that there are at most $C+1$ distinct classes of columns of depth $k$, and among these classes there is at most one frequent class. Since a non-frequent class has at most $C$ representatives while $v$ has at least $C^2+C+4$ columns of depth $k$, 
it follows that one of the classes (which in particular has to be frequent)
must have at least $C+4$ representatives. Let us denote this class by $\calC$.

Given a gallery $\calG$ at $v$ of length $\min(3,\deg(v))$, we can find
distinct $i,j\not\in \supp(\calG)$ such that $col_i(v)$ and $col_j(v)$ lie in $\calC$.
As before, one of the commuting maps $\psi=C_{\calG}(R_{ij}^{\pm 1})$ has the property that every new vertex $w'$ 
is obtained from an old vertex $w$ by replacing a column of depth $k$ (equal to either $col_i(w)=col_i(v)$ or $col_j(v)=col_j(v)$) by a $(k-1)$-column. Such $\psi$ is $k$-safe, similarly to previous cases.
Without loss of generality, we can assume that $col_i(w)=col_i(v)$ is the column that is being replaced.
 We claim that $w'<v$ for every new vertex $w'$, so that $\psi$ is a reduction.
We consider two subcases:

\skv
{\it Subcase 1:} $A(w)>A(v)$ (where $w$ is as above). Since $w'$ is obtained from $w$ by removing a column of depth $k$ 
(which does not contribute to the active subgroup), we have $A(w')\geq A(w)>A(v)$, so $w'<v$.
\skv
{\it Subcase 2:} $A(w)=A(v)$. If $A(w')>A(w)$, we are done, so assume from now on that $A(w')=A(w)=A(v)$. 
This means that the equivalence relation defining the equivalence classes of $k$-columns is the same for $v,w$ and $w'$.

We first claim that $SA(w')\geq SA(w)$. Since $w'$ and $w$ only differ in the $i^{\rm th}$ column, it suffices
to show that the class of $col_i(v)$ is frequent for $w'$ as well. The latter holds since $w'$ differs from $w$
in at most $3$ columns and by assumption the class of $col_i(v)$ has $C+4$ representatives for $v$.

The same argument shows that the $i^{\rm th}$ column of $w$ is not helpful, so either $SA(w')>SA(w)$
(in which case we are done) or $SA(w')=SA(w)$ and $w'$ and $w$ have the same number of helpful columns. 
In the latter case, $w'$ and $w$ cannot be separated by the first 3 criteria of the Step~2 order,
but $w'$ has more $(k-1)$-columns, so $w'<w$ and hence $w'<v$.

\skv
{\it Case 4: $SA(v)\neq \dbZ^k$ and $v$ has at least $C+1$ $0$-columns}. In this case we will use the same maps as in Case~4 of Step~1, but the role of good columns will be played by $v$-helpful columns (recall that $v$ has at least one helpful column). We already know from Step~1 that these maps are 
$k$-safe. Showing that they are also reductions (now with respect to Step~2 order) requires a new proof, although the arguments
for most cell types are quite similar. Below we explicitly consider the types which were exceptional in Step~1, namely 1ZA, 2ZA, 2ZB, 3ZB and 5ZB.
For the remaining types, the argument in Step~1 relied primarily on the fact that $\Delta$-good columns form the complement of a subgroup inside the group of unrestricted columns, and the latter property is shared by the set of $v$-helpful columns.
\skv
As in Case~4 of Step~1, we fix a cell $\calF$ containing $v$ and permute the columns so that $col_1(v)$ is helpful (and in particular unrestricted), $col_4(v)=0$ and $4\not\in supp(\calF)$.

\skv
{\it Type 1ZA.} Recall that $I=\{1,2,4\}$ and $\calF_I$ has vertices $(a,b,0)$, $(a+b,b,0)$, $(a+b,b+c,0)$, $(a+b+c,b+c,0)$ and $(a,b+c,0)$.
First we prove that $a,b$ and $c$ must be unrestricted. If $v_I=(*,b+c,0)$, Lemma~\ref{lem:zeasyStep2a}(a) applied with $J=\{1\}$ implies that $a$, $a+b$ and $a+b+c$ are unrestricted and hence $b$ and $c$ are also unrestricted. If $v_I=(*,b,0)$, we can still deduce from Lemma~\ref{lem:zeasyStep2a}(a)
that $a$ and $b$ are unrestricted. And if $c$ is restricted, then $b+c$ is also restricted, whence $(a,b+c,0)>(a,b,0)$
and $(a+b,b+c,0)>(a+b,b,0)$, contrary to the assumption that $v_I=(*,b,0)$.

If $c$ is $v$-helpful, the commuting map $C_{\calF}(R_{34})$ is a reduction by Lemma~\ref{lem:zeasyStep2b}(b),
so from now on we can assume that $c$ is not $v$-helpful and thus $c\in S(v)$ (as $c$ is unrestricted). We will
show that if $\phi$ is the map in Figure~\ref{1ZA}, then $\phi^*$ is a reduction.
\skv

{\it Subcase 1: $col_2(v)\in S(v)$}. Since $col_2(v)$ equals $b$ and $b+c$ and $c\in S(v)$, in either case we must
have $b,c\in S(v)$. This means that $a,a+b$ and $a+b+c$ are all $v$-helpful or none of them is $v$-helpful, and the 
latter is impossible since $col_1(v)$ is helpful by assumption. Thus, all elements of $Sub_4(\phi^*)=\{a,a+b,a+b+c\}$ are 
$v$-helpful and hence $\phi^*$ is a reduction by Lemma~\ref{lem:zeasyStep2b}(b).

\skv

{\it Subcase 2: $col_2(v)\not \in S(v)$}. Since $b$ and $c$ are unrestricted, this implies that $col_2(v)$ is helpful
and hence the first 2 columns of $v$ are helpful. Since $v$ is maximal and every vertex of $\calF$ can only differ from $v$ in the first 2 columns, 
applying Lemma~\ref{lem:zeasyStep2a}(b) with $J=\{1,2\}$, we deduce that for any vertex $w$ of $\calF$ either 
\begin{multline}
\label{crit:soft}
\mbox{(a) $w$ is $v$-soft} \quad \mbox{ or } \quad \mbox{(b) $S(v)=S(w)$ and  $col_1(w)$ and $col_2(w)$ are $v$-helpful}.
\end{multline}
The following claim can be checked by straightforward case-by-case verification. Let $w'$ be a new vertex of $\phi^*$. Then for
any $i\in I=\{1,2,4\}$ there exists a vertex $w_i$ of $\calF$ such that $w'$ can be obtained from $w_i$ by replacing a zero
column by another vector, possibly followed by a permutation of columns, and moreover $col_i(w')$ equals one of the first two columns
of $w_i$. 

If $w_i$ is $v$-soft for some $i$, then $w'<v$ by Lemma~\ref{lem:zeasyStep2b}(a)(i). By \eqref{crit:soft}, the only other
possibility is that the first two columns of $w_i$ are $v$-helpful for each $i\in I$. But this means that $col_i(w')$ is $v$-helpful
for all $i\in I$ and hence $w'<v$ by Lemma~\ref{lem:zeasyStep2b}(a)(ii).

\skv
{\it Type 2ZA.} Recall that  $I=\{1,2,4\}$ and the vertices of $\calF_I$ are $(a,b,0)$, $(a+c,b,0)$, $(a,a+b,0)$, $(a+c,a+b,0)$
and $(a+c,a+b+c,0)$.

First, if $v_I\neq (a+c,a+b+c,0)$, Lemma~\ref{lem:zeasyStep2a}(a) applied with $J=\{1\}$ implies that $a$ and $a+c$ are unrestricted. Suppose now that $v_I=(a+c,a+b+c,0)$, in which case $a+c$ is still unrestricted. And if $a$ is restricted, the first two columns
of either $(a,b,0)$ or $(a,a+b,0)$ are restricted, and hence by Lemma~\ref{lem:zeasyStep2a}(a) applied with $J=\{1,2\}$ the same is true for $v$, a contradiction.

Thus, we proved that $a$ and $a+c$ are unrestricted and hence $c$ is also unrestricted. If $c$ is $v$-helpful, the map $C_{\calF}(R_{43})$
is a reduction. And if $c$ is not $v$-helpful, then $c\in SA(v)$. Since one of the vectors $a,a+c$ must be $v$-helpful, they are both
$v$-helpful. Hence if we define $\phi$ as in Figure~\ref{2ZA}, then $\phi^*$ is a $ZS_4$-map and all vectors of $Sub_4(\phi^*)=\{a,a+c\}$
are $v$-helpful, so $\phi^*$ is a reduction by Lemma~\ref{lem:zeasyStep2b}(b). 
\skv
{\it Type 2ZB.} Recall that $I=\{1,3,4\}$ and the vertices of $\calF_I$ are $(a,c,0)$, $(a+b,c,0)$, $(a,a+b+c,0)$, $(a+b,a+b+c,0)$ and $(a,b+c,0)$. Similarly to Step~1, if $b$ is $v$-helpful, then $C_{\calF}(L_{24})$ is a reduction by Lemma~\ref{lem:zeasyStep2b}(b). In the remaining cases we will use the map $\phi^*$
where $\phi$ is given by Figure~\ref{2ZB}. Recall that $Sub_4(\phi^*)=\{a,a+b\}$.
\skv
If $b\in S(v)$, then both $a$ and $a+b$ are $v$-helpful (since one of them must be $v$-helpful), and we are
done by Lemma~\ref{lem:zeasyStep2b}(b). Thus, we can assume that $b$ is restricted, in which case
$v_I=(a,b+c,0)$ (otherwise we get a contradiction with Lemma~\ref{lem:zeasyStep2a}(a)), so in particular
$a$ is unrestricted. If $b+c$ is unrestricted, then $c$ is restricted, whence $(a,b+c,0)<(a,c,0)$, a contradiction. And if $b+c$ is restricted, then $a+b$ and $a+b+c$ are both restricted, so 
$(a,b+c,0)<(a+b,a+b+c,0)$, again a contradiction.

\skv
{\it Type 3ZB.} Recall that $I=\{1,2,4\}$ and the vertices of $\calF_I$ are $(a-c,b,0)$, $(a,b,0)$, $(a+b,b,0)$, $(a+b,b+c,0)$ and $(a-c,b+c,0)$. 

First we claim that $a,a+b$ and $a-c$ are all unrestricted (whence $b$ and $c$ are also unrestricted). If $v_I=(*,b,0)$, this is automatic by Lemma~\ref{lem:zeasyStep2a}(a). Suppose now that $v_I=(*,b+c,0)$. Lemma~\ref{lem:zeasyStep2a}(a) still implies that
$a+b$ and $a-c$ are both unrestricted, whence so is $b+c=(a+b)-(a-c)$. Thus the first two columns
of $v$ are unrestricted, whence (by maximality of $v$) every vertex of $\calF$ has at least one unrestricted column among the first two; in particular, this is true for $w=(a,b,0,*)$. If $a$ is restricted, then so
is $b=(a+b)-a$, a contradiction. Thus, $a$ is unrestricted, and we are done.

If $c$ is $v$-helpful, $C_{\calF}(L_{34})$ is a reduction, so from now on we assume that $c\in S(v)$.
We will use the diagram map $\phi^*$ where $\phi$ is given by Figure~\ref{3ZB}.

\skv
{\it Subcase 1: $b\in S(v)$}. Since $c\in S(v)$ as well, the vectors $a$, $a-c$ and $a+b$ are all congruent modulo $S(v)$, and since
one of them is $v$-helpful, they must all be $v$-helpful, so $\phi^*$ is a reduction by Lemma~\ref{lem:zeasyStep2b}(b).

\skv
{\it Subcase 2: $b\not\in S(v)$}. In this subcase $b$ and $b+c$ are both $v$-helpful. Hence the first two columns
of $v$ are helpful, and we can finish the proof exactly as in Subcase~2 for type 1ZA.

\skv
{\it Type 5ZB.} Recall that $I=\{2,3,4\}$, all vertices of $\calF$
have the same first column $a$ (which thus must be $v$-helpful),
$v_I=(b,c,0)$, and the other 3 vertices of $\calF_I$ are $(b-a,c,0)$,
$(b,c-a,0)$ and $(b-a,c-a,0)$. As in Step~1, we will use one of the two maps in  Figure~\ref{5ZBM}.
\skv

First suppose that $c-a$ or $b-a$ is $v$-helpful. By symmetry, we can
assume that $c-a$ is $v$-helpful. In this case the first map
in Figure~\ref{5ZBM} is a reduction by Lemma~\ref{lem:zeasyStep2b}(b). 
\skv

Thus from now on we can assume that each of the vectors $c-a$ and $b-a$
either lies in $S(v)$ or is restricted. We consider several subcases.

{\it Subcase~1: $c-a,b-a\in S(v)$.} Since $a$ is $v$-helpful, $b$ and $c$ must also be $v$-helpful. Since
$v=(a,b,c,0,*)$ is maximal, $v$ only differs from other vertices of $\calF$ in the $2^{\rm nd}$ and $3^{\rm rd}$
columns, and each of those vertices has a non-$v$-helpful $2^{\rm nd}$ or $3^{\rm rd}$ column, each of the remaining
vertices of $\calF$ must be $v$-soft. Hence  the first map in Figure~\ref{5ZBM} is a reduction by Lemma~\ref{lem:zeasyStep2b}(a)(i).
\skv

{\it Subcase~2: $c-a\in S(v)$ while $b-a$ is restricted or vice versa.} Without loss of generality, we can
assume that $c-a\in S(v)$ and $b-a$ is restricted (in which case $b$ is also restricted). In this case
$(a,b-a,c,0,*)$ is also a maximal vertex while $(a,b,c-a,0,*)$ and $(a,b-a,c-a,0,*)$ are $v$-soft by the same
argument as in Subcase~1. Hence if $\phi^*$ is the second map in Figure~\ref{5ZBM}, the four new vertices of $\phi^*$ with $4^{\rm th}$ column $a$ are $<v$ by Lemma~\ref{lem:zeasyStep2b}(a)(ii) while the remaining
two new vertices are $v$-soft, so $\phi^*$ is a reduction.
\skv

{\it Subcase~3: $c-a$ and $b-a$ are both restricted.} In this case $\calF$ is a purely maximal cell, and
the second map in Figure~\ref{5ZBM} is a reduction by the argument we used for types 5GA, 5GB and 5ZB in Step~1.
\skv

This concludes the proof in Case~4.
\end{proof}

\subsection{Step~3} The goal in this final step is to reduce the $(k-1)$-defect of the diagram while preserving the $(k-1)$-unimodularity property.
Unfortunately, in order to achieve the latter, we will have to allow the norms of the last $d-k+1$ rows of the diagram to increase; however,
these norms will still stay uniformly bounded, which is sufficient for our purposes. More precisely, our goal in this step is to prove the
following result.

\begin{Theorem} 
\label{step3}
There exist natural numbers $1=M_{d+1}\leq M_d\leq M_{d-1}\leq \cdots \leq M_1$ with the following property. Suppose that
for some $1\leq k\leq d$ we are given a diagram $\Lambda$ such that 
$\|bot_{d-k+1}(\Lambda)\|\leq M_{k+1}$. Then there exists
a diagram map $\Lambda\to\Lambda'$ such that
\begin{itemize}
\item[(a)] $def_{k-1}(\Lambda')\leq d-k+C+1$;
\item[(b)] $\|bot_{d-k+1}(\Lambda')\|\leq M_{k}$;
\item[(c)] if $\Lambda$ is $(k-1)$-unimodular, then so is $\Lambda'$.
\end{itemize}
\end{Theorem}

Before proving Theorem~\ref{step3}, let us deduce Theorem~\ref{thm:hdim} from the main results of Steps~1, 2 and 3 (Theorems~\ref{thm:step1},~\ref{step2}~and~\ref{step3}, respectively).

\begin{proof}[Proof of Theorem~\ref{thm:hdim}]
We will prove that the assertion of Theorem~\ref{thm:hdim} holds for the same $\{M_i\}$ as
in Theorem~\ref{step3}.
Let $\Omega$ be any diagram satisfying the hypotheses of Theorem~\ref{thm:hdim}, that is,
\begin{itemize}
\item[(a)] $\Omega$ is $k$-unimodular;
\item[(b)] $def_k(\Omega)\leq d-k+C$;
\item[(c)] if $k<d$, then $\|bot_{d-k}\Omega\|\leq M_{k+1}$;
\item[(d)] $\Omega$ has small boundary.
\end{itemize}

Let $\Omega\to\Omega'$ and $\Omega'\to\Omega''$ be the reductions from
Step~1 (Theorem~\ref{thm:step1}) and Step~2 (Theorem~\ref{step2}). By Theorem~\ref{thm:step1} we have $\|row_k(\Omega')\|=1$. Since both reductions are
$k$-safe, we have $\|row_k(\Omega'')\|\leq \|row_k(\Omega')\|=1$ and
$\|bot_{d-k}\Omega''\|\leq \|bot_{d-k}\Omega\|\leq M_{k+1}$ if $k<d$.
We have 
\begin{align*}
&\|bot_{d-k+1}\Omega''\|=\max\{\|bot_{d-k}\Omega''\|, \|row_k\Omega''\|\}\leq \max\{M_{k+1},1\}= M_{k+1} 
&&\mbox { if }k<d\\
&\|bot_{d-k+1}\Omega''\|=\|bot_1\Omega''\|=\|row_d\Omega''\|=1=M_{d+1} &&\mbox { if }k=d.
\end{align*}
Thus, $\|bot_{d-k+1}\Omega''\|\leq M_{k+1}$ regardless of the value of $k$, so we can apply Theorem~\ref{step3} to $\Omega''$.
Since $\Omega''$ is $(k-1)$-unimodular, Theorem~\ref{step3} yields a diagram map $\Omega''\to \Omega'''$
such that
\begin{itemize}
\item[(a)] $def_{k-1}(\Omega''')\leq d-k+C+1$;
\item[(b)] $\|bot_{d-k+1}(\Omega''')\|\leq M_{k}$.
\item[(c)] $\Omega'''$ is $(k-1)$-unimodular,
\end{itemize}
which proves Theorem~\ref{thm:hdim}.
\end{proof}

\begin{proof}[Proof of Theorem~\ref{step3}]
We will define the numbers $M_i$ by downward induction. Recall that we already set $M_{d+1}=1$.

In this proof we will use a completely different order. It will only take into account the $\ell^1$-norms of the bottom
$d-k+1$ rows, but give different weights to different rows, with the lower rows receiving much higher weight.
Also, unlike Steps~1 and 2, we will explicitly define a norm function which gives the order.

Define a sequence of functions $f_d(x_d)$, $f_{d-1}(x_{d},x_{d-1})$, $\ldots$, $f_k(x_d,\ldots, x_k)$ inductively by
$f_d(x_d)=x_d$ and $f_i(x_d,\ldots, x_i)=3^{f_{i+1}(x_d,\ldots, x_{i+1})}\cdot x_i$ for $k\leq i\leq d-1$.
Now given a matrix $v\in Mat_{d\times n}(\dbZ)$, define $N(v)=f_k(x_d,\ldots, x_k)$ where $x_i=\|row_i v\|_1$.
Finally, given $v,v'\in Mat_{d\times n}(\dbZ)$, set $v<v'$ if and only if $N(v)<N(v')$.
\skv

Below we will prove that if a diagram $\Sigma$ contains a vertex $v$ with $def_{k-1}(v)> (d-k+1)+C$, there exists a reduction
(with respect to the above order) $\Sigma\to\Sigma'$ which eliminates $v$ and preserves $(k-1)$-unimodularity. If this
is proved, then by Lemma~\ref{obs:Artinian}, applying such reductions finitely many times (starting with $\Lambda$), we will obtain a $(k-1)$-unimodular diagram $\Lambda'$ such that $def_{k-1}(\Lambda')\leq (d-k+1)+C$, so conditions (a) and (c) from Proposition~\ref{step3} hold. Moreover, we have $\|bot_{d-k+1}(\Lambda')\|\leq N(\Lambda')\leq N(\Lambda)\leq f_k(M_{k+1},\ldots, M_{k+1})$, so if we set $M_{k}=f_k(M_{k+1},\ldots, M_{k+1})$, then (b) holds as well, and the proof is complete.
\skv

We now explain how to construct the desired reduction. Fix a vertex $v$ with $def_{k-1}(v)> (d-k+1)+C$.
Without loss of generality we can assume that $v$ has maximal $(k-1)$-defect among all vertices.
Choose any gallery $\calG$ at $v$ of length $\min(3,\deg(v))$. 
As before, it suffices to construct a $\calG$-reduction at $v$. By assumption $v$ has
at least $d-k+C+2$ columns of depth at least $k$, that is, of depth between $k$ and $d$. Since there are $d-k+1$ integers
between $k$ and $d$ and $|\supp(\calG)|\leq C$, there exist distinct $i,j\not\in supp(\calG)$ such that $col_i(v)$ and $col_j(v)$
have the same depth $m\geq k$. Without loss of generality we can assume that $v_{mi}\geq v_{mj}>0$. 

Let $\phi$ be the commuting map $C_{\calG}(R_{ij})$ if $H=\widetilde{\IAR_{n,d}}$ and $C_{\calG}(R_{ji}^{-1})$ if $H=\widetilde{\IAC_{n,d}}$. We claim that $\phi$ has the desired properties. Since $\phi$ does not modify any of the $(k-1)$-columns, it automatically
preserves $(k-1)$-unimodularity. 

It remains to prove that $\phi$ is a reduction. Any new vertex $u$ arising from $\phi$ is obtained from a vertex $w$ of $\calG$
by replacing $col_i(w)=col_i(v)$ by $col_i(v)-col_j(v)$. It is clear that $def_{k-1}(u)\leq def_{k-1}(w)\leq def_{k-1}(v)$.
Also if we let $x_t=\|row_t(u)\|_1$ and $y_t=\|row_t(w)\|$, then by construction $x_t=y_t$ for $t>m$, $x_m\leq y_m-1$
and $x_t\leq 2y_t$ for $k\leq t<m$. 

\skv
Let us prove that $f_t(x_d,\ldots, x_t)\leq f_t(y_d,\ldots, y_t)-1$ for all $1\leq t\leq m$ using downward induction on $t$.
The result is clear for $t=m$. Suppose now that $f_t(x_d,\ldots, x_t)\leq f_t(y_d,\ldots, y_t)-1$ for some $1<t\leq m$.
Then 
\begin{multline*}
f_{t-1}(x_d,\ldots, x_{t-1})=3^{f_t(x_d,\ldots, x_{t})}x_{t-1}\leq 3^{f_t(y_d,\ldots, y_{t})-1}\cdot 2y_{t-1}\\
= \frac{2}{3}\cdot 3^{f_t(y_d,\ldots, y_{t})}y_{t-1}=\frac{2}{3}\cdot f_{t-1}(y_d,\ldots, y_{t-1})\leq f_{t-1}(y_d,\ldots, y_{t-1})-1,
\end{multline*}
where the last inequality holds since clearly $f_{s}(y_d,\ldots, y_{s})\geq 3$ for all $s<d$.

Applying the above inequality with $t=k$, we deduce that $N(w)<N(v)$ for any new vertex $w$, so $\phi$ is indeed a reduction, as desired.
\end{proof}

\subsection{Proof of Proposition~\ref{prop:essential2}}
We start by recalling the statement of Proposition~\ref{prop:essential2} as well as the definition of a small subset (Definition~\ref{def:small}), combined in the following statement:

\begin{Proposition}
\label{prop:essential3} There exists a finite subset $A\subset G/H$ such that the preimage of $A$ in $G$ is connected in $Cay(G,\calX)$ and $A$ is small, that is, every $v\in A$ satisfies the following conditions:  
\begin{itemize}
\item[(i)] $\|v\|=1$;
\item[(ii)] $v$ differs from the matrix 
$\begin{pmatrix}I_{d\times d} \mid 0_{d\times (n-d)}\end{pmatrix}$ in at most 8 columns;
\item[(iii)] $v$ is $k$-unimodular for all $k\leq d$.
\end{itemize}
\end{Proposition}

\begin{proof}
As before, let $\rho:G\to G/H$ and $\theta:F(\calX)\to G/H$ denote the natural projections. 
Recall that we are identifying $G/H$ with $Um_{d\times n}(\dbZ)$ via the map from Lemma~\ref{obs:isomorphism}.

An explicit formula for $\theta$ viewed as a map from $F(\calX)$ to $Um_{d\times n}(\dbZ)$ is as follows.

\begin{Lemma}
\label{lem:theta}
Given $w\in F(\calX)$, let $[w]$ be its projection to $\SL_n(\dbZ)$. Then 
\begin{itemize}
\item[(i)] $\theta(g)$ consists of the first $d$ rows of $[g]^{-1}$ if $H=\widetilde{IAR_{n,d}}$ and
\item[(ii)] $\theta(g)$ consists of the first $d$ rows of $[g]^{T}$ if $H=\widetilde{IAC_{n,d}}$.
\end{itemize}
\end{Lemma}

We will prove Proposition~\ref{prop:essential3} for $H=\widetilde{\IAR_{n,d}}$;
the proof for $H=\widetilde{\IAC_{n,d}}$ is quite similar.
\skv

As we explained in the proof of Proposition~\ref{thm:criterion}, if we are given an explicit finite subset $A$ of $G/H$
such that $\rho^{-1}(A)$ is connected in $Cay(G,\calX)$, it is easy to construct a finite generating set for $H$ in terms of $A$.
In this proof we will proceed backwards, starting with a finite generating set for $H$ and using the following observation to construct $A$ with the above property. 

\begin{Observation}
\label{obs:Schreier}
Let $S_H$ be a finite generating set for $H$. For each $s\in S_H^{\pm 1}$ choose a lift $\widetilde s\in F(\calX)$, let
$E$ be the set of all suffixes of those lifts and let $A=\theta(E)$. 
Then $\rho^{-1}(A)$ is connected in $Cay(G,\calX)$.
\end{Observation}

In view of Observation~\ref{obs:Schreier}, to prove Proposition~\ref{prop:essential3} we just need to find a finite generating
set $S_H$ for $H$ with the following property: every $s\in S_H^{\pm 1}$ admits a lift $\widetilde s\in F(\calX)$ such that 
\begin{itemize}
\item[(*)] $\theta(w)$ is small for any suffix $w$ of $\widetilde s$.
\end{itemize}

We start by describing the finite generating set $S_H$ for $H$ that we will be using. Recall that $H=\widetilde{\IAR_{n,d}}$ is defined as the preimage of the subgroup $\Row_{n,d}$ of
$\SL_n(\dbZ)$ consisting of matrices whose first $d$ rows coincide with those of the identity matrix.
Thus, if we denote by $\widetilde {\IA_n}$ the kernel of the projection $G=\widetilde{\SAut(F_n)}\to\SL_n(\dbZ)$, we can produce a generating set for $H$ by taking the union of a generating set for $\widetilde {\IA_n}$ and any subset of $\widetilde{\SAut(F_n)}$
which maps onto a generating set of $\Row_{n,d}$.

Magnus~\cite{Ma} showed that $\IA_n$ is generated by the automorphisms $K_{ij}=(x_i\mapsto x_j^{-1}x_i x_j)$
with $i\neq j$ and $K_{ijk}=(x_i\mapsto x_i[x_j,x_k])$ with $i,j,k$ distinct. The group $\widetilde {\IA_n}$
is generated by (chosen) preimages of these elements, e.g. $R_{ij}L_{ij}^{-1}$ and $[R_{ij},R_{ik}]$,
and $w_{12}^4$ which generates the kernel of the map $\widetilde{\SAut(F_n)}\to \SAut(F_n)$.

The group $\Row_{n,d}$ is generated by the elementary matrices $E_{ji}$, $i\neq j$, where $1\leq i\leq n$
and $d+1\leq 1\leq j\leq n$. As their preimages we can use the elements $R_{ij}$ with the same restrictions on $i$ and $j$. Thus, the set $\{R_{ij}L_{ij}^{-1}\}\cup \{[R_{ij},R_{ik}]\}\cup\{R_{ij}: j\geq d+1\}\cup\{w_{12}^4\}$ generates $H$ and hence the set
$$S_H=\{R_{ij}L_{ij}^{-1}, [R_{ij},R_{ik}]: j\leq d\}\cup\{L_{ij}, R_{ij}: j\geq d+1\}\cup\{w_{12}^4\} \eqno(***)
$$ also generates $H$. 

A simple way to ensure that a matrix $v\in Um_{d\times n}(\dbZ)$ is $k$-unimodular for all $k\leq d$ is to require that
every element of the standard basis of $\dbZ^d$ (denoted below by $e_1,\ldots, e_d$) appears among the columns of $v$, up to sign.
Thus, to prove Proposition~\ref{prop:essential3} it suffices to show that 
every $s\in S_H^{\pm 1}$ admits a lift $\widetilde s$ such that every suffix $w$ of $\widetilde s$
satisfies the following two conditions:

\begin{itemize}
\item[(1)] for every $1\leq i\leq d$ one of the columns of $\theta(w)$ is equal to $e_i$ or $-e_i$;
\item[(2)] every entry of $\theta(w)$ is equal to $0$ or $\pm 1$ and $\theta(w)$
differs from $\begin{pmatrix}I_{d\times d} \mid 0_{d\times (n-d)}\end{pmatrix}$ in at most 8 columns.
\end{itemize}
Clearly, we only need to consider generators of the form $R_{ij}L_{ij}^{-1}$ and $[R_{ij},R_{ik}]$ with $j\leq d$ and their inverses.

Let us first consider $s=R_{ij}L_{ij}^{-1}$. If we let $\widetilde s$ be the ``natural'' lift 
(that is, $\widetilde s=R_{ij}L_{ij}^{-1}$ where $L_{ij}$ and $R_{ij}$ are considered as elements of $F(\calX)$),
it is easy to see that both (1) and (2) holds for $i>d$, in which case for every suffix $w$ of $\widetilde s$,
the matrix $\theta(w)$ is of the block form $(I_d | 0_{d\times (n-d)})$.

If $i\leq d$, the suffix $L_{ij}^{-1}$ still satisfies (2), but not (1), although it is not far off. Indeed, by Lemma~\ref{lem:theta}, the matrix $\theta(L_{ij}^{-1})\in Um_{d\times n}(\dbZ)$ has the block form $( E_{ji} | 0)$ where $E_{ji}$ is the $d\times d$ elementary matrix having $1$ in the $(j,i)$-entry. Thus, all elements of the standard basis of $\dbZ^d$ except $e_i$ are present among the columns of $\theta(L_{ij})$.
To fix this issue we replace $R_{ij}L_{ij}^{-1}$ by a more involved word in $F(\calX)$ which has the same projection to $G$.

\begin{Claim}
\label{claim:lift}
Fix integers $k,m$ satisfying $d<k<m\leq n$ (recall that $i\neq j$ are both $\leq d$, so $i,j,k$ and $m$ are distinct)
The following relation holds in $G$:
\begin{equation}
\label{eq:lift}
R_{ij}L_{ij}^{-1}=R_{ki}^{-1}L_{mj}L_{mi}^{-1}L_{mj}^{-1}R_{ij}R_{ki}R_{kj}^{-1}L_{ij}^{-1}L_{mi}.
\end{equation}
\end{Claim}
\begin{proof} By assumption, $i,j,k$ and $m$ are distinct, and without loss of generality we can assume that
$i=1,j=2,k=3$ and $m=4$ (we do this for better readability; there are no mathematical simplifications).

Using the basic relations  
$R_{31}^{-1}R_{12}R_{31}=R_{12}R_{32}$ and $L_{41}^{-1}L_{12}L_{41}=L_{12}L_{42}$ (see Observation~\ref{rel:conj}), we get
$R_{12}=R_{31}^{-1}R_{12}R_{31}R_{32}^{-1}$ and $L_{12}=L_{41}^{-1}L_{12}L_{41}L_{42}^{-1}$, whence
$L_{12}^{-1}=L_{42}L_{41}^{-1}L_{12}^{-1}L_{41}$ and thus
$R_{12}L_{12}^{-1}=R_{31}^{-1}R_{12}R_{31}R_{32}^{-1}L_{42}L_{41}^{-1}L_{12}^{-1}L_{41}$.
Since $L_{42}$ commutes with $R_{12},R_{31}$ and $R_{32}$ and $L_{41}$ commute with $R_{31}$ and $R_{32}$,
it follows that
$$R_{12}L_{12}^{-1}=R_{31}^{-1}L_{42}R_{12}L_{41}^{-1}R_{31}R_{32}^{-1}L_{12}^{-1}L_{41}. \eqno (***)$$
From the relation $L_{41}R_{12}L_{41}^{-1}=L_{42}^{-1}R_{12}$, we get $R_{12}L_{41}^{-1}=L_{41}^{-1}L_{42}^{-1}R_{12}$
and hence $$R_{12}L_{12}^{-1}=R_{31}^{-1}L_{42}(L_{41}^{-1}L_{42}^{-1}R_{12})R_{31}R_{32}^{-1}L_{12}^{-1}L_{41}\mbox{ in }G,$$
as desired.
\end{proof}
We claim that the word on the right-hand side of \eqref{eq:lift} is a lift of $s=R_{ij}L_{ij}^{-1}$ with required properties,
that is, all of its suffixes satisfy (1) and (2). The latter is shown in the following table where
again we assume that $i=1,j=2,k=3$ and $m=4$ and this time also $d=2, n=4$.
 \vskip .3cm
\hskip -1cm
\small
{
\begin{tabular}{|c|c|c|c|c|c|}
\hline
$w$&$w_0=1$&$w_1=L_{41}$&$w_2=L_{12}^{-1}w_1$&$w_3=R_{32}^{-1}w_2$&$w_4=R_{31}w_3$\\
\hline
$\theta(w)$&
$\begin{pmatrix}1&0&0&0 \\ 0&1&0&0  \end{pmatrix}$&
$\begin{pmatrix}1&0&0&-1 \\ 0&1&0&0   \end{pmatrix}$&
$\begin{pmatrix}1&0&0&-1  \\ 1&1&0&0  \end{pmatrix}$&
$\begin{pmatrix} 1&0&0&-1  \\ 1&1&1&0  \end{pmatrix}$&
$\begin{pmatrix} 1&0&-1&-1  \\ 1&1&0&0 \end{pmatrix}$\\
\hline
&&&&&\\
\hline
$w$&$w_5=R_{12}w_4$&$w_6=L_{42}^{-1}w_5$&$w_7=L_{41}^{-1}w_6$&$w_8=L_{42}w_7$&$w_9=R_{31}^{-1}w_8$\\
\hline
$\theta(w)$&
$\begin{pmatrix} 1&0&-1&-1  \\ 0&1&0&0 \end{pmatrix}$&
$\begin{pmatrix}1&0&-1&-1  \\ 0&1&0&1 \end{pmatrix}$&
$\begin{pmatrix} 1&0&-1&0  \\ 0&1&0&1 \end{pmatrix}$&
$\begin{pmatrix} 1&0&-1&0  \\ 0&1&0&0 \end{pmatrix}$&
$\begin{pmatrix} 1&0&0&0  \\ 0&1&0&0 \end{pmatrix}$
\\
\hline
\end{tabular}
}
\normalsize
\vskip .2cm
Similarly one can construct desired lifts of $(R_{ij}L_{ij}^{-1})^{-1}=L_{ij}R_{ij}^{-1}$.
\skv
The lifts for $s=[R_{ij},R_{ik}]$ are slightly easier to construct. As before, the natural lift works
if $i>d$, so assume that $i\leq d$, and fix $m>d$ which is distinct from $j$ and $k$.

Using the relations $R_{ij}^{R_{mi}}=R_{ij}R_{mj}$ and $R_{ik}^{R_{mi}}=R_{ik}R_{mk}$, we get
$[R_{ij},R_{ik}]^{R_{mi}}=[R_{ij}R_{mj},R_{ik}R_{mk}]$. Since $R_{ij}$ and $R_{ik}$
both commute with $R_{mj}$ and $R_{mk}$, we have $$[R_{ij}R_{mj},R_{ik}R_{mk}]=[R_{ij},R_{ik}][R_{mj},R_{mk}],$$
and therefore 
\begin{equation}
\label{eq:Kijk}
[R_{ij},R_{ik}]=R_{mi}^{-1}[R_{ij},R_{ik}]R_{mi}[R_{mk},R_{mj}]. 
\end{equation}
It is straightforward to check that all suffixes of the word on the right-hand side of \eqref{eq:Kijk} satisfy both (1) and (2),
which finishes the proof.
\end{proof}

\section{Additional diagrams for the ZB case}

In this short section we will describe the single-cell reductions for types 8-11 for the group $\IAC_{n,1}$ in the zero coordinate case (ZB case). As we already mentioned in \S~4, the reduction diagrams will be identical to the GA case. For each type, the map described below
is a reduction since each new vertex is obtained from an old vertex of a diagram $\Omega$ by replacing a zero coordinate by an 
$\Omega$-good coordinate (this is a special case of Lemma~\ref{lem:zeasy}(a)). But in order to show that Lemma~\ref{lem:zeasy}(a) is applicable, we need to check that certain vertex coordinates are good, which is done below. We keep all the notations from \S~3,4. In particular, $\calF$ denotes the cell at which the reduction is being performed, and $v$ denotes a chosen maximal of $\calF$.
\skv
  {\it Type 9: $l(\partial\calF)= (L_{21}L_{12}^{-1}R_{21}=w_{21})$}, $I=\{1,2,4\}$. The vertices of $\calF_I$ are
$(a,b,0), (a,b-a,0),(b,b-a,0),(a-b,b,0)$. We consider 2 cases. If $v_I=(a,b,0)$ or $(a-b,b,0)$, then by Lemma~\ref{lem:zeasy}(b) $a$ and $a-b$
are both good whence the first map in Figure~\ref{9ZB} is a reduction by Lemma~\ref{lem:zeasy}(a). Likewise, if $v_I=(a,b-a,0)$ or $(b,b-a,0)$,
then $a$ and $b$ are good whence the second map in Figure~\ref{9ZB} is a reduction.
\skv
\input{figure9ZB.tex}

\skv
{\it Type 8: $l(\partial\calF)= (L_{12}L_{21}^{-1}R_{12}=w_{12})$}, $I=\{1,2,4\}$. The vertices of $\calF_I$ are $(a,b,0), (a-b,b,0),(a-b,a,0),(a,b-a,0)$.
If $v_I=(a,b,0)$ or $(a-b,b,0)$, then $a$ and $a-b$ are both good, whence the first map in Figure~\ref{8ZB} is a reduction. And if $v_I=(a-b,a,0)$
or $(a,b-a,0)$, then $\calF$ has a maximal vertex with a good second coordinate, whence by swapping the roles of the indices $1$ and $2$, we are reduced to type 8.
\skv

\input{figure8ZB.tex}

\skv

{\it Type 10:  $l(\partial\calF)= (w_{12}^{-1}R_{12}w_{12}=L_{21}^{-1})$}, $I=\{1,2\}$. The vertices of $\calF_I$ are
$(a,b,0), (b,-a,0),(b-a,-a,0)$ and $(a,b-a,0)$. If $v_I=(*,-a,0)$, then $b$ and $b-a$ are both good, whence the first map in Figure~\ref{10ZB} is a reduction. If $v_I=(a,*,0)$, then $a$ is good, whence the second map in Figure~\ref{10ZB} is a reduction. 
\input{figure10ZB.tex}

\skv
{\it Type 11R:  $l(\partial\calF)= (w_{12}^{-1}R_{13}w_{12}=R_{23})$}, $I=\{1,2,3,4\}$. As in type~11A, we will assume that the edge labeled by
$w_{12}$ is incoming at $v$, which forces us to consider different boundary label representatives depending on the support of the other edge
at $v$. The maps in both cases are shown in Figure~\ref{11ZB}. By assumption $a$ is good in the first case and $b$ is good in the second case, so
in both cases the map is a reduction Lemma~\ref{lem:zeasy}(a).

\input{figure11ZB.tex}

\end{document}